\documentclass[reqno,11pt,final]{article}

\usepackage{amssymb,amsmath,amsthm,enumerate,latexsym}
\usepackage{amsfonts,amsmath,mathrsfs,bm}
\usepackage{mathdots}

\usepackage{graphicx,psfrag,fullpage}
\usepackage{subfig}

\usepackage[utf8]{inputenc}   
\usepackage[english]{babel}  
\usepackage[T1]{fontenc}  
\usepackage{lmodern}

\usepackage{bbm}
\usepackage{array}
\usepackage{mathtools}

\usepackage{hyperref}
\hypersetup{
    colorlinks = true,
    citecolor=black,
    filecolor=black,
    linkcolor=blue,
    urlcolor=black
}

\usepackage{graphicx}
\usepackage{tikz}
\usetikzlibrary{arrows,automata,trees}
\usetikzlibrary{shapes}
\usetikzlibrary{calc}

\newtheorem{thm}{Theorem}[section]
\newtheorem{definition}{Definition}[section]

\newtheorem{prop}{Proposition}[section]
\newtheorem{lem}{Lemma}[section]

\theoremstyle{definition}

\newtheorem{rem}{Remark}[section]

\newcommand{\C}{\mathbb C}
\newcommand{\R}{\mathbb R}
\newcommand{\N}{\mathbb N}

\newcommand{\dS}{\mathbb S}

\newcommand{\dP}{\mathbb{P}}
\newcommand{\TR}{\Tr}
\newcommand{\dN}{\N}
\newcommand{\dE}{\E}
\newcommand{\dR}{\R}
\newcommand{\dC}{\C}

\newcommand{\cP}{\mathcal {P}}
\newcommand{\cN}{\mathcal {N}}

\newcommand{\cD}{\mathcal {D}}
\newcommand{\cF}{\mathcal {F}}
\newcommand{\cT}{\mathcal {T}}
\newcommand{\cB}{\mathcal {B}}
\newcommand{\cM}{\mathcal {F}}

\newcommand{\rF}{\mathrm{F}}
\newcommand{\T}{\intercal}
\newcommand{\cH}{\mathcal {H}}
\newcommand{\cE}{\mathcal {E}}
\newcommand{\cR}{\mathcal {R}}
\newcommand{\cA}{\mathcal {A}}

\renewcommand{\rho}{\varrho}

\newcommand{\veps}{\varepsilon}

\newcommand{\iC}{\mathrm{i}}
\newcommand{\Hsym}{$(H_\mathrm{sym}) $ }
\newcommand{\Hac}{$(H_\mathrm{ac})$ }

\DeclareMathOperator{\E}{\mathbb E}

\DeclareMathOperator{\Tr}{Tr}

\newcommand{\1}{\mathbbm 1}
\newcommand{\IND}{\1}

\newcommand{\PAR}[1]{{{\left(#1\right)}}} 
\newcommand{\ABS}[1]{{{\left| #1 \right|}}} 

\newcommand{\NRMHS}[1]{\NRM{#1}_\mathrm{{\tiny{HS}}}}  
\newcommand{\NRM}[1]{{{\left\| #1\right\|}}} 

\usepackage{float}

\DeclareMathOperator{\wind}{wind}
\DeclareMathOperator{\spectrum}{sp}
\DeclareMathOperator{\supp}{supp}

\DeclareMathOperator{\bfa}{\bf a}
\DeclareMathOperator{\diag}{diag}

\DeclareMathOperator{\adj}{adj}

\DeclareMathOperator{\SPAN}{span}

\usepackage{todonotes}

\title{Outliers of perturbations of  banded Toeplitz  matrices}
\author{Charles Bordenave\thanks{Aix-Marseille Univ, CNRS, I2M, Marseille, France. Email: charles.bordenave@univ-amu.fr}\,, \; Mireille Capitaine\thanks{Univ Toulouse, CNRS, IMT, Toulouse, France. Email: mireille.capitaine@math.univ-toulouse.fr}\quad and \; Fran\c cois Chapon\thanks{Univ Toulouse, CNRS, IMT, Toulouse, France. Email: francois.chapon@math.univ-toulouse.fr}}

\begin{document}
\maketitle

\begin{abstract}
Toeplitz matrices form a rich class of possibly non-normal matrices whose asymptotic spectral analysis in high  dimension is well-understood. The spectra of these matrices are notoriously highly sensitive to small perturbations. In this work, we analyze the spectrum of a banded Toeplitz matrix perturbed by a random matrix with iid entries of variance $\sigma_n^2 / n$ in the asymptotic of high dimension and $\sigma_n$ converging to $\sigma \geq 0$. Our results complement and provide new proofs on recent progresses in the case $\sigma = 0$. For any $\sigma \geq 0$, we show that the point process of outlier eigenvalues is governed by a low-dimensional random analytic matrix field, typically Gaussian, alongside an explicit deterministic matrix that captures the algebraic structure of the resonances responsible for the outlier eigenvalues.  On our way, we prove a new functional central limit theorem for trace of polynomials in deterministic and random matrices and present new variations around Szeg\H{o}'s strong limit theorem. 
\end{abstract}
\setcounter{tocdepth}{3}
\tableofcontents

\section{Introduction}

\subsection{Perturbation of banded Toeplitz matrices}
Let $T_n(\bfa)$ be a  $n\times n$ banded Toeplitz   matrix, with symbol $\bfa\colon \mathbb S^1 \to \C$ given by the Laurent polynomial
\begin{equation}
{\bf a}(\lambda) = \sum_{k=-r}^s a_k \lambda^k, \quad \lambda\in \mathbb S^1   = \{ z\in \C: |z|=1\},  \label{symbol-def}
\end{equation}
where we assume without loss of generality  that $r\geq 0$, $s>0$ and the $a_k$'s are complex numbers with $a_s \not=0$ (the symmetric case $r>0$ and $s\geq0$ readily follows by considering the transpose). For integer $n > \max(r,s) $, the matrix $T_n(\bf a)$ is thus the matrix $(a_{j-i})_{1\leq i,j \leq n}$ whose diagonals are constant:
\begin{equation} \label{defToeplitz}
T_n(\bfa) = 
\begin{pmatrix}
a_0 & a_1 & \cdots & a_s & 0 & \cdots & 0 \\
a_{-1} & a_0  &  \ddots &  \ddots &  \ddots & \ddots & \vdots \\
\vdots  &  \ddots & \ddots &  \ddots &  \ddots &  \ddots & 0 \\
a_{-r}  &  \ddots &  \ddots & \ddots &  \ddots &  \ddots & a_s \\
0 & \ddots  &  \ddots &  \ddots &  \ddots &  \ddots & \vdots   \\
\vdots  & \ddots & \ddots  &  \ddots &  \ddots & a_0 &  a_1 \\
0 & \cdots & 0 & a_{-r} & \cdots & a_{-1} & a_0
\end{pmatrix}.
\end{equation}

In this paper, we are interested in random perturbations of the matrix $T_n(\bfa)$ in the asymptotic of large dimension $n \to \infty$. That is, we consider the matrix 
\begin{equation}\label{eq:defMn}
M_n = T_n(\bfa) + \sigma_n Y_n,
\end{equation}
where $\sigma_n >0$ and $Y_n = (Y_{ij})_{1 \leq i,j \leq n}$ is some noise matrix normalized so that its normalized Hilbert-Schmidt norm is of order one: 
$$
\dE \frac{1}{n} \| Y \|_{\text{HS}} ^2 = \dE  \frac {1} {n} \sum_{i,j} |Y_{ij}|^2  = 1,
$$
(see below for the precise assumptions on $Y_n$ considered in this paper). With this scaling, the matrices $Y_n$ and $T_n$ are of comparable  Hilbert-Schmidt norms. We will assume that the following limit exists: for some real $\sigma$, 
\begin{equation}\label{eq:defsigma}
\lim_{n \to \infty} \sigma_n = \sigma.
\end{equation}
We will distinguish two distinct regimes: 
\begin{enumerate}[(i)]
\item
$\sigma >0$ 
\item
$\sigma = 0$ and   $\lim_{n \to \infty} \frac{1}{n } \ln \sigma_n = 0$.
\end{enumerate}
The first case corresponds to a macroscopic perturbation of $T_n(\bfa)$. The second case corresponds to a microscopic perturbation which, as illustrated notably \cite{SjostrandVogel,basak-zeitouni20}, is however large enough to have an important effect on the spectrum. We will refer in the sequel these two cases as $\sigma >0$ and $\sigma = 0$.

There are multiple motivations to study random matrices of the form \eqref{eq:defMn}. One of them comes from numerical matrix analysis. The eigenvalues and eigenvectors of non-normal matrices (such as $T_n(\bfa)$ when $a_i \ne \bar a_{-i}$ for some $i$) can be very sensitive to matrix entries and rounding errors. As a consequence numerical outputs for eigenvalues can be completely wrong even for matrices of low dimensions, see Figure \ref{fig:example_num} for an illustration. Attempts to capture this phenomenon include the notion of pseudo-spectrum, see Trefethen and Embree \cite{trefethen2005spectra}. Von Neumann and Goldstine \cite{bams/1183511222} initiated the idea to study numerical rounding errors in matrix computations by random noise. This idea has now a very rich history, see for example Spielman and Teng \cite{10.1145/380752.380813} or Edelman and Rao
\cite{edelman_rao_2005}. 
As pointed in \cite[Chapter 7]{trefethen2005spectra}, non-normal Toeplitz matrices are the best understood non-normal matrices and they are the most obvious choice of matrices to study spectral instability in great details.  This fruitful approach has notably been followed in \cite{bordenave-capitaine16,SjostrandVogel0,SjostrandVogel,BasakPaquetteZeitouni20,basak-zeitouni20,BVZ,alt2024spectrumoccupiespseudospectrumrandom}. As pointed in \cite{SjostrandVogel0,SjostrandVogel}, beyond numerical analysis, another motivation for studying random matrices of the form \eqref{eq:defMn} comes from non-Hermitian quantum mechanics, see \cite{PhysRevLett.77.570,GoldsheidKhoruzhenko}.

\begin{figure}[ht]
    \centering
    \subfloat[\centering $S T(\bfa) S^*$, $n=10$.] {{\includegraphics[width=3.9cm]{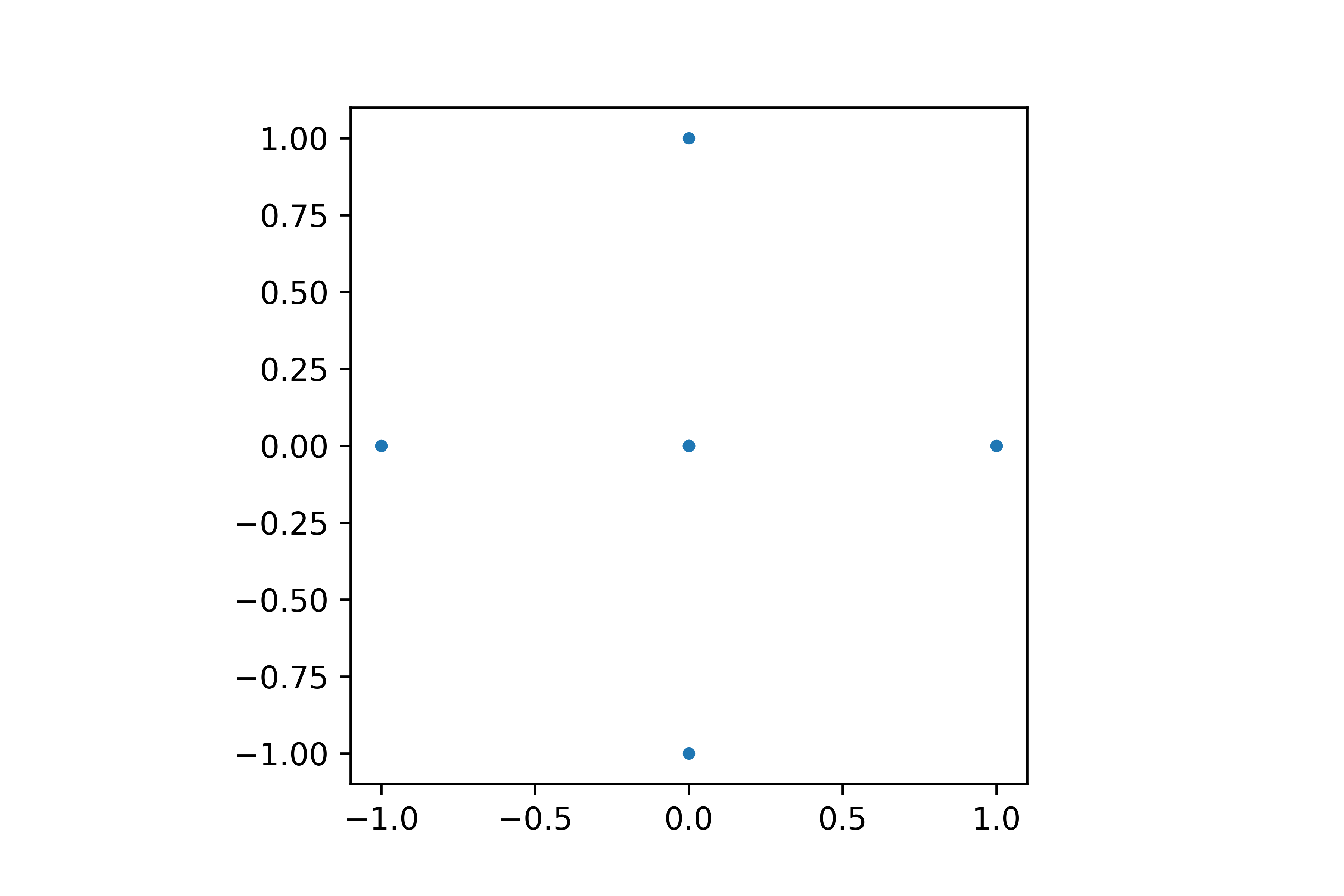} }}        \subfloat[\centering \centering $S T(\bfa) S^*$, $n=2107$.]{{\includegraphics[width=3.9cm]{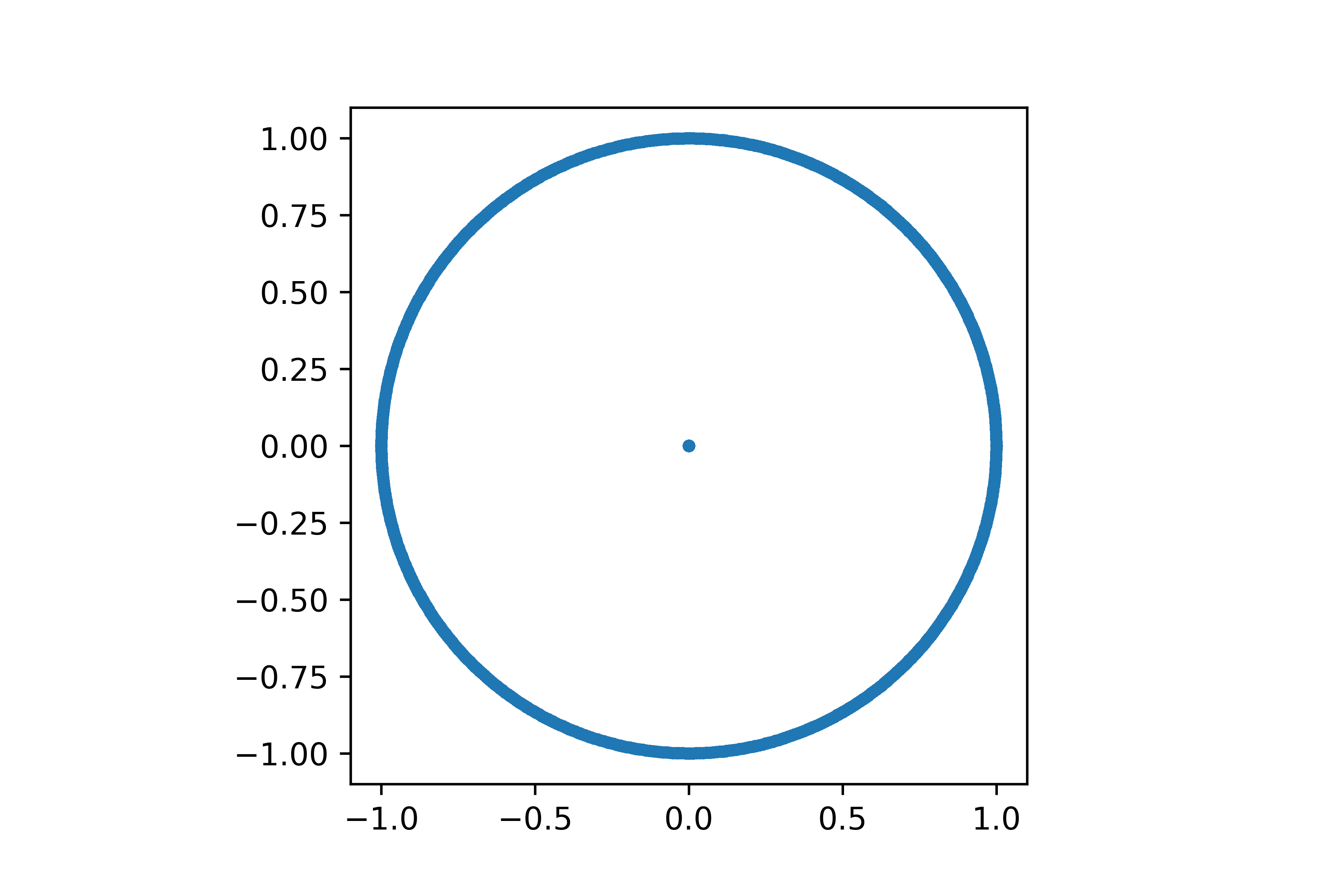} }}     \subfloat[\centering $F T(\bfa) F^*$, $n=10$.] {{\includegraphics[width=3.9cm]{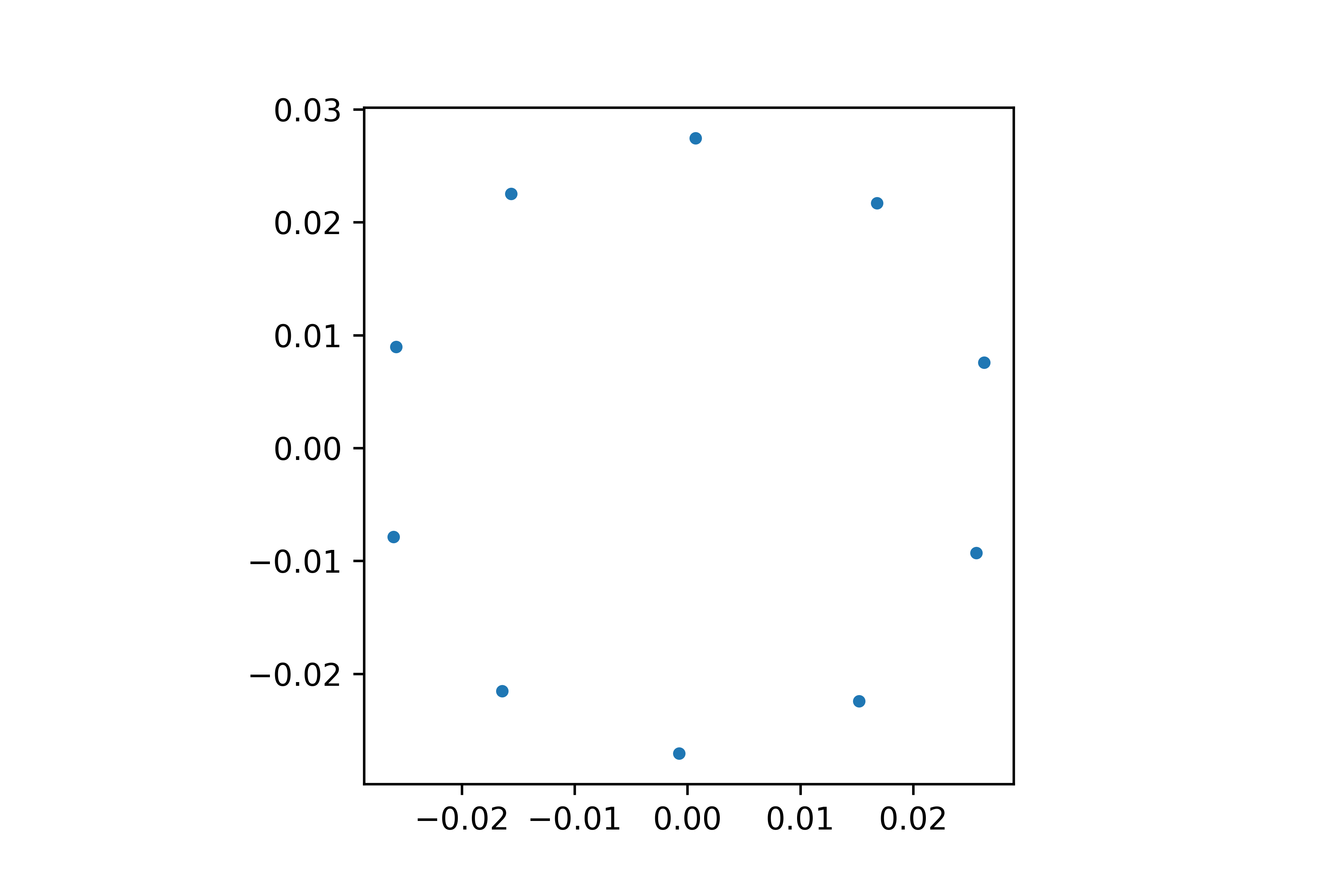} }}%
    \subfloat[\centering $F T(\bfa) F^*$, $n=2107$.]      {{\includegraphics[width=3.9cm]{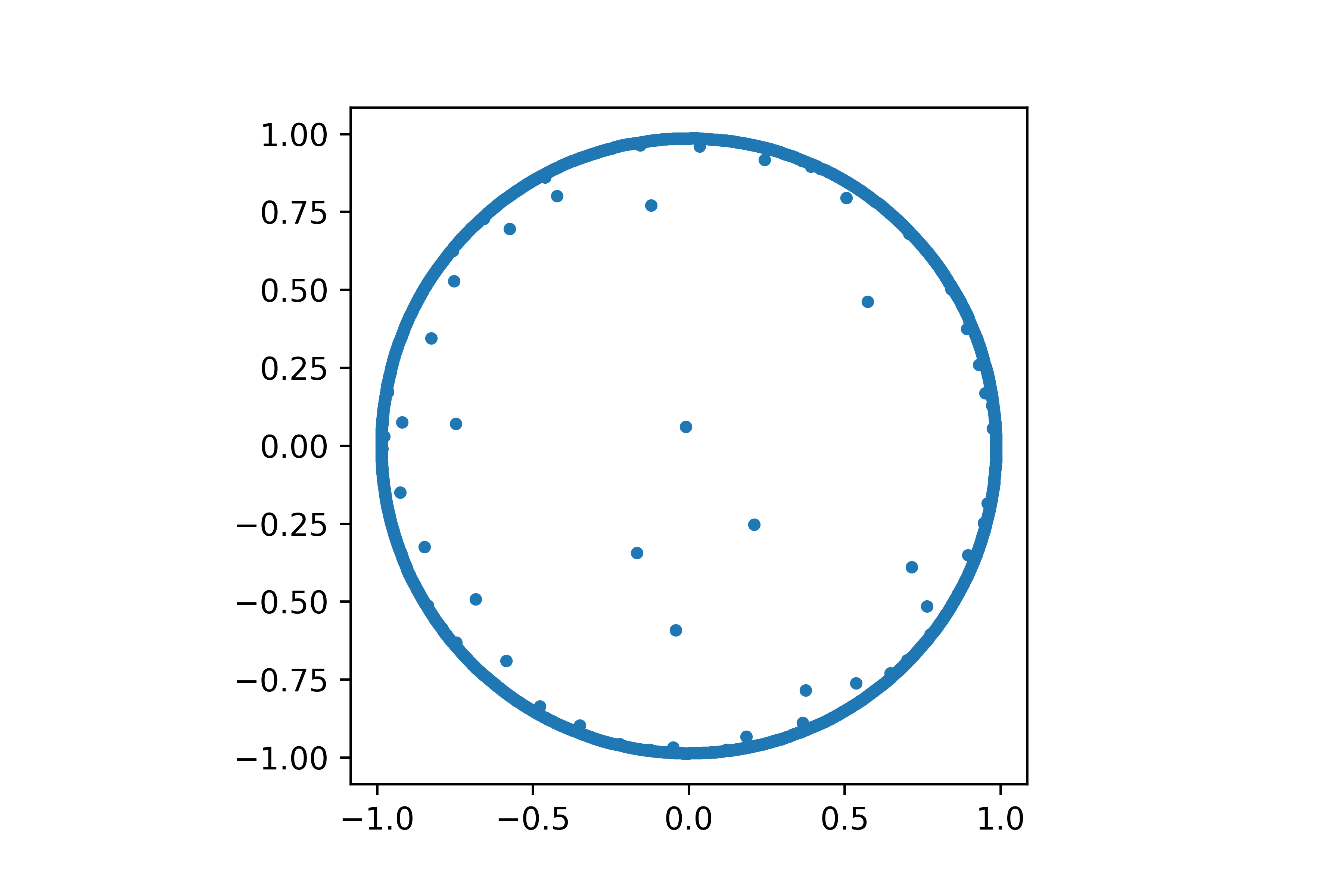} }}
    \caption{Output for eigenvalues in Python using {\ttfamily numpy.linalg.eig} for the nilpotent matrix $\bfa(\lambda) = \lambda$. (a)-(b): $S$ is a (random) permutation matrix which hides the upper-triangular structure, and (c)-(d) $F$ is the Fourier basis, see \eqref{eq:defFn}, which introduces rounding errors in matrix entries.}%
    \label{fig:example_num}%
\end{figure}

\subsection{Assumptions on the noise matrix}

Our matrix $Y_n$ will have independent and identically distributed entries in an orthogonal basis. More precisely, we assume that 
$$
Y_n = U_n \frac{X_n}{\sqrt n} U_n^*,
$$ 
where $U_n$ is a  $n\times n$ unitary matrix and $X_n = (X_{ij})_{1 \leq i,j \leq n}$ will satisfy the assumption:  
\begin{enumerate}
\item[$(H_0)$]\label{H0} $(X_{ij})_{i,j \geq 1}$  are independent and identically distributed complex random variables with $\dE X_{ij} = 0$, $\dE |X_{ij}|^2 =1$, $\dE X_{ij}^2=\varrho$, for some $0 \leq \varrho \leq 1$.
\end{enumerate}

To be precise, the common distribution of the random variables $X_{ij}$ is assumed to be independent of $n$. Throughout this paper, we implicitly assume that $(H_{0})$ holds. For technical reasons, some of the following assumptions will also be needed in the statements of our results. Below, we say that the distribution of a complex random variable $X$ is symmetric if $X$ and $-X$ have the same distributions.  For real $k > 2$, we consider the assumptions:
\begin{enumerate}
\item[$(H_{k})$] \label{H1}  $ \dE |X_{ij}|^k   < \infty$.
\item[\Hsym] \label{H2} The distribution of $X_{ij}$ is symmetric.
\item[\Hac]\label{H3} The law of $X_{11}$   is absolutely continuous with respect to the Lebesgue measure on $\dC \simeq \dR^2$ or on $\omega \dR$, $\omega \in \dS^1$.
\end{enumerate}

We conjecture than none of the assumptions $(H_k)$-\Hsym are necessary for the main results below of this introduction to hold (that is, $(H_0)$ should suffice), this conjecture is notably supported by \cite{BasakPaquetteZeitouni20,basak-zeitouni20,BCG} where closely related results are established without  these assumptions.

There are two particularly interesting choices of unitary matrices $U_n$ which correspond respectively to the canonical basis and the discrete Fourier basis: that is, $U_n = I_n$ and $U_n = F^*_n$ with $F_n = (F_n(i,j))_{1 \leq i,j \leq n}$ given, with $\omega_n = e^{2\iC \pi / n}$, by 
%
\begin{equation}\label{eq:defFn}
F_n= \frac{1}{\sqrt{n}} 
\begin{pmatrix} 1&1&1&\cdots&1\\ 1& \omega_n & \omega_n^2& \cdots &\omega_n^{n-1}\\
\\ 1& \omega_n^2 & \omega_n^4& \cdots &\omega_n^{2(n-1)}\\
\vdots&\vdots& \vdots& & \vdots\\\\ 1& \omega_n^{n-1} & \omega_n^{2(n-1)}& \cdots &\omega_n^{(n-1)(n-1)}\end{pmatrix}.
\end{equation}
The Fourier basis is especially natural because it is the basis which diagonalizes the circulant matrix $C_n(\bfa)$ associated to the symbol $\bfa$ (see below). Note finally that if the distribution of $X_{ij}$ is the complex standard Gaussian distribution on $\dC$ (corresponding to $X_{ij}$ Gaussian and $\rho =0$) then the distribution of $Y_n$ does not depend on the unitary matrix $U_n$.

\subsection{Convergence of the empirical spectral distribution}

Before turning to the main focus of this paper, we describe the asymptotic behavior of typical eigenvalues of $M_n = T_n(\bfa) + \sigma_n Y_n$ defined in \eqref{eq:defMn}.

In the case $\sigma = 0$, under mild assumptions on $Y_n$,  \cite{BasakPaquetteZeitouni20} proved that the empirical measure of the eigenvalues (or empirical spectral distribution) of $M_n$ converges weakly, in probability to $\beta_0$ defined as the law of ${\bf a}(U)$ with $U$ uniformly distributed on 
$\mathbb{S}^1$ (see also \cite{SjostrandVogel} for the case of Gaussian noise). Their assumptions are for example met if $\sigma_n = n^{-\gamma}$, $\gamma > 0$ and if assumptions $(H_0)$ holds. Beware that if $Y_n = 0$, that is $M_n = T_n(\bfa)$ the empirical spectral distribution does not converge toward $\beta_0$ in the non-normal case, see Subsection \ref{subsec:limitESD} below.

In the case $\sigma > 0$, if Assumption $(H_0)$ holds, the empirical spectral distribution of $M_n$ converges towards a probability measure $\beta_\sigma$ on $\dC$. The measure $\beta_\sigma$ is the Brown's spectral measure of ${\bf a}(u)+\sigma c$ where $c$ and $u$ are free noncommutative random variables in some tracial noncommutative probability space $(\mathcal A, \phi)$ such that $c$ is a circular noncommutative random variable and $u$ is Haar distributed. This result will be explained in Section \ref{model} below and is a consequence of the combinations of the main results of \cite{sniady02} and \cite{tao-vu08}.

For $\sigma >0$, the measure $\beta_\sigma$ has a relatively explicit expression (see Section \ref{model}). We will prove the following result valid for all $\sigma  \geq 0$. 
 \begin{lem} \label{le:supportbeta}
For $\sigma \geq 0$, the support of $\beta_\sigma$ is given by
 \begin{equation*}
\supp (\beta_\sigma) = \left\{ z \in \dC  :   \frac{1}{2\pi}\int_0^{2\pi} \frac{d\theta}{\vert {\bf a}(e^{i\theta})-z\vert^2}  \geq \sigma^{-2}\right\} 
\end{equation*}
where by convention $1/0 = \infty$, so that $\supp(\beta_0) = \bfa (\dS^1)$.
\end{lem}

Figure~\ref{fig:example_brown} shows an example, for the symbol $\bfa(t)=t^2+2t+it^{-2}-0.5it^{-3}$ and $\sigma=0.8$, of the curve $\bfa(\mathbb S^1)$ and  the support of the Brown measure $\beta_\sigma$. 

\begin{figure}[ht]
    \centering
    \subfloat[\centering The curve $\bfa(\mathbb S^1)$. ]{{\includegraphics[width=5cm]{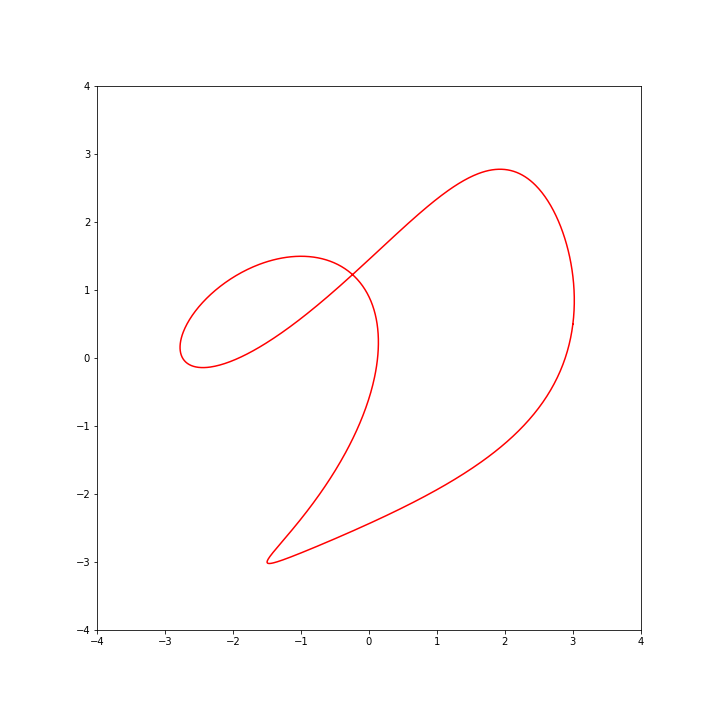} }}%
    \qquad
    \subfloat[\centering The curve $\bfa(\mathbb S^1)$ and the boundary of $\supp \beta_\sigma$. ]{{\includegraphics[width=5cm]{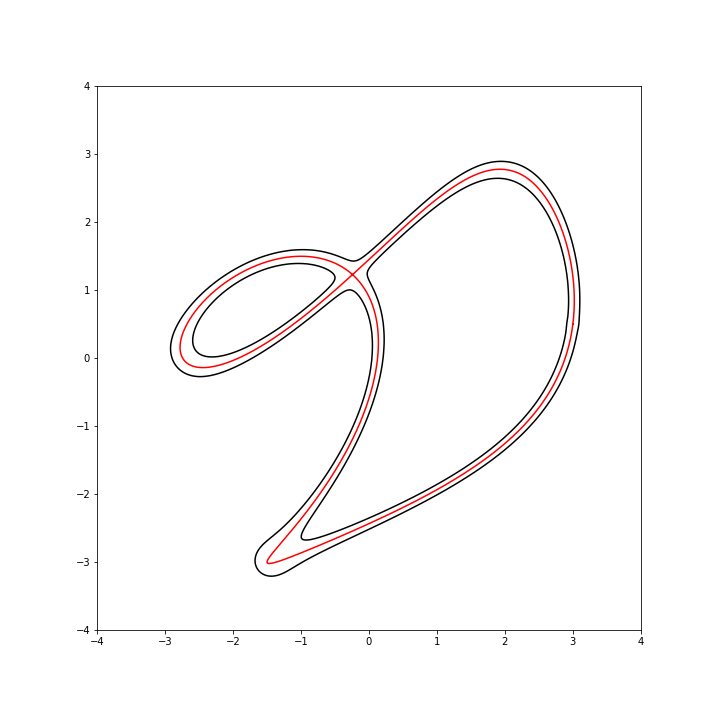} }}%
    \caption{Example of the closed curve  $\bfa(\mathbb S^1)$ and the support of the Brown measure $\beta_\sigma$, for the symbol $\bfa(t)=t^2+2t+it^{-2}-0.5it^{-3}$ and $\sigma=0.8$.}%
    \label{fig:example_brown}%
\end{figure}

The objective of this paper is to describe the eigenvalues of $M_n$ in closed regions of $\dC \backslash \supp(\beta_\sigma)$ which we will call the outlier eigenvalues. For $\sigma = 0$, $U_n = I_n$ (canonical basis) this objective has essentially been completed in \cite{SjostrandVogel,basak-zeitouni20} and, with a very different approach, for $\sigma >0$, $r = 0$, $s=1$ in  \cite{bordenave-capitaine16} (in the Fourier basis).

\subsection{Stable region}

The infinite dimensional  Toeplitz operator
 $T({\bf a})$ is the operator on $\ell^2 (\N)$ defined by 
\[
(T({\bf a})x )_i= \sum_{l=-r}^s a_l  x_{l+i}, \; \mbox{for $i\in \N$, \; where}\;x  =(x_1,x_2,\ldots)
\]
and we set $x_i=0$ for non-positive integer values of $i$. 
The spectrum of $T({\bf a})$ is given by, see \cite[Corollary~1.11]{bottcher-grudsky05},
\[
\spectrum T(\bfa) = \bfa(\mathbb S^1) \cup \left\{z \in \C\setminus \bfa(\mathbb S^1) :  \wind(\bfa-z)\not=0\right\},
\]
where $\wind(\mathbf{b})$ is the winding number around $0$ of a closed curve $\mathbf{b} : \dS^1  \to \C$. If $b$ extends to a meromorphic function, it is equal to the number of zeroes of $\mathbf{b}$ in the unit disk $\mathbb D$ minus the number of its poles in $\mathbb D$.

The following result asserts that there are typically no outliers of $M_n$ in $\spectrum(T(\bfa))^c$ as $n \to \infty$.

 \begin{thm} \label{thm-no-outliers}
Assume that $(H_4)$ holds.
 Let $\Gamma \subset  \dC \backslash ( \supp(\beta_\sigma) \cup \spectrum T(\bfa) )$ be a compact set. Then, almost surely,  for all $n$ large enough, there are no eigenvalues of $M_n$ in $\Gamma$.
 \end{thm}
 
Note that if $\sigma = 0$ then $\supp(\beta_0) = \bfa(\dS^1) \subset \spectrum (T(\bfa))$. In the language of \cite{bordenave-capitaine16,SjostrandVogel}, the region $\spectrum(T(\bfa))^c$ is stable. The random perturbation $M_n$ of $T_n(\bfa)$ has typically no outliers in this region. Our moment condition is probably not optimal: when $U_n = I_n$, if $\bfa = 0$ or $\bfa \ne 0$, $\sigma = 0$, $\sigma_n = n^{-\gamma}$ for some $\gamma >0$, then \cite{BCG} and \cite{basak-zeitouni20} prove respectively that Theorem \ref{thm-no-outliers} holds on the sole condition $(H_0)$. 

\subsection{Unstable region}

To complement Theorem \ref{thm-no-outliers}, we are now interested in the individual eigenvalues of $M_n$ in a compact subsets of $ \spectrum T(\bfa) \backslash \supp(\beta_\sigma)$. In the language of \cite{bordenave-capitaine16,SjostrandVogel}, the region $\spectrum(T(\bfa))$ is unstable and the study of outliers exhibits many interesting phenomena.

When $U_n = I_n$,  $\sigma = 0$, $\sigma_n = n^{-\gamma}$ for some $\gamma >0$ and the entries of $X$ satisfy  a density assumption, Basak and Zeitouni \cite{basak-zeitouni20} proved that, for each $-r \leq \mathfrak{\delta}\leq s$, $\delta \ne 0$, the process of outliers in the region $$ \mathcal R_{0,\sigma} = \left\{z \in \C\setminus \bfa(\mathbb S^1)  :  \wind(\bfa-z)=\mathfrak{\delta}\right\} \subset \spectrum T(\bfa) \backslash \supp(\beta_0)$$
converges to the point process described by the zero set of some random analytic function which differs across the above regions indexed by $\delta$. When $U_n = F_n$ is the Fourier basis, $\sigma >0$ and $r=0$, $s= 1$, \cite{bordenave-capitaine16} established the convergence of the point process of outliers inside $\spectrum T(\bfa) \backslash \supp(\beta_\sigma)$ toward the zero set of a Gaussian analytic function.

In the present paper, we are both interested in the cases $\sigma = 0$ and $\sigma > 0$. We establish the convergence of the point process of outliers of $M_n$ in $\spectrum T({\bf a})\backslash \supp(\beta_\sigma) $ toward the point process of  the zeroes of some random analytic function with an explicit correlation kernel. We clarify when this random analytic function will be Gaussian or not. In particular, when $\sigma = 0$, in comparison to \cite{basak-zeitouni20}, we give an alternative expression for the random analytic function.

\paragraph{Banded Toeplitz matrices as perturbed circulant matrices.} To state our main results in the unstable region, we start with a basic but nevertheless key observation. As in \cite{bordenave-capitaine16}, we see banded Toeplitz matrices  as small rank perturbations of normal matrices (that is a matrix with an orthonormal basis of eigenvectors).  More precisely, given a row vector $(c_0,c_1,\ldots,c_{n-1}) \in \C^n$, the associated circulant matrix $\text{circ}(c_0,c_1,\ldots,c_{n-1})$  is the $n\times n$ matrix whose first row is  $(c_0,c_1,\ldots,c_{n-1})$ and such that each row  is rotated one element to the right relative to the preceding row, that is: 
\[
\text{circ}(c_0,c_1,\ldots,c_{n-1}) =
\begin{pmatrix}
c_0 & c_1 & c_2 & \cdots & c_{n-1} \\
c_{n-1} & c_0 & c_1 & \cdots & c_{n-2} \\
c_{n-2} & c_{n-1} & c_0 & \cdots & c_{n-3} \\
\vdots & \vdots & \vdots & \ddots & \vdots \\
c_1 & c_2 & c_3 & \cdots & c_0
\end{pmatrix}. 
\]

Circulant matrices are special cases of (not necessarily banded) Toeplitz matrices, but are much simpler from the viewpoint of spectral theory. 
To the Toeplitz matrix $T_n(\bfa)$ with symbol the Laurent polynomial $\bfa$, is then associated, for $n > r+s$,  the circulant matrix $C_n(\bfa)$ defined by
\begin{equation} \label{def-circulant}
C_n(\bfa) = \text{circ} (a_0,a_1,\ldots,a_s, \underbrace{0,\ldots,0}_{\mathclap{n-(r+s+1) \text{ times}} }  , a_{-r},\ldots,a_{-2},a_{-1}) . 
\end{equation}

Our study of outliers of noisy banded Toeplitz matrices 
is based on the following decomposition of  $T_n(\bfa)$:
if $n >r+s$, we write
\[
T_n(\bfa)=C_n(\bfa) + B_n,
\]
where 
\[
B_n = -
\begin{pmatrix}
0_{ r \times s} & 0_{r \times (n-r-s)} & D_{r} \\
0_{(n-r-s) \times s} &0_{(n-r-s)\times (n-r-s)} & 0_{(n-r-s)\times r } \\
E_{s} & 0_{s \times (n-r-s) } & 0_{s \times r}
\end{pmatrix}
\]
with  $0_{p \times q}$ denoting the $p\times q$ zero matrix and 
\[
D_{r}= 
\begin{pmatrix}
a_{-r} & a_{-r+1} & \cdots & a_{-1} \\
0 & a_{-r} & \cdots & a_{-2} \\
\vdots & \vdots & \ddots & \vdots \\
0 & 0 & \cdots & a_{-r}
\end{pmatrix},
\quad
E_s = 
\begin{pmatrix}
a_{s} & 0 & \cdots & 0 \\
a_{s-1} & a_{s} & \cdots & 0 \\
\vdots & \vdots & \ddots & \vdots \\
a_{1} & a_{2} & \cdots & a_{s}
\end{pmatrix}.
\] 
Moreover,  we write $B_n = - P_n Q_n$, with $P_n, Q_n^\top \in M_{n,r+s} ( \dC)$ defined as
\begin{equation}\label{PnQn}
P_n= \begin{pmatrix} 
0_{r\times s}&I_r\\
0_{(n-s-r)\times s}&0_{(n-s-r)\times r} \\ 
E_s & 0_{s\times r}
\end{pmatrix},
\quad
Q_n= \begin{pmatrix} 
I_s &0_{s\times (n-s-r)}&0_{s\times r}\\
0_{r \times s} & 0_{r\times (n-s-r)}& D_r
 \end{pmatrix}.
 \end{equation}

Similarly to $M_n$, we introduce the random matrix: 
\begin{equation}\label{eq:defSn}
S_n = C_n(\bfa) + \sigma_n Y_n.
\end{equation}

For $z \in \C$ outside $\spectrum (S_n)$ and $\spectrum(C_n(\bfa))$, these resolvent matrices will play an important role: 
\begin{equation}\label{defgeneralresolvantes}
R_n ( z) = ( z  - S_n )^{-1} \quad \text{ and } \quad R'_n ( z) = ( z  - C_n(\bfa))^{-1}.
\end{equation}

The following matrices in $M_{r+s}(\dC)$ will be ubiquitous in formulas. For any $\omega  \in  \mathbb{S}^1$, we set
\begin{equation}\label{eq:defcM}
\cM(\omega)= (\omega^{p-q})_{1 \leq p,q \leq r+s} \; , 
\quad \mathcal{ D} = \begin{pmatrix}I_s&0\\0&D_r\end{pmatrix} \quad \hbox{ and } \quad \mathcal{ E}= \begin{pmatrix}E_s&0\\0&I_r\end{pmatrix}.
\end{equation}

\paragraph{The master  matrix field.} Let $\Omega \subset \dC$ be a bounded open connected set such that  $\bar \Omega \subset \dC \backslash \supp(\beta_\sigma) $, where $\bar \Omega$ is the closure of $\Omega$. We describe the limiting point processes of zeroes in $\Omega$ through a matrix field $W$ which is a random variable on the set $\cH_{r+s}(\Omega)$ of analytic functions $\Omega \to M_{r+s}(\dC)$ equipped with its usual topology of uniform convergence on compact subsets of $\Omega$, see Subsection \ref{subsec:RAF} in appendix. The set $\cP ( \cH_{r+s}(\Omega))$ of probability measures on $\cH_{r+s}(\Omega)$ is endowed with the weak topology. If $W$ is  a random variable in $\cH_{r+s}(\Omega)$ which is square integrable, we define its covariance kernel as the functions  $ \Omega \times \Omega \to M_{r+s} (\dC) \otimes M_{r+s} (\dC)$, defined for $z,z' \in \Omega$ by
$$
K(z,z') = \dE  [ W(z) \circledast \bar W(z') ] \quad \hbox{ and } \quad  K'(z,z') = \dE [ W(z) \circledast W(z') ],
$$
where $A \circledast B$ is the tensor product of the matrices unfolded as vectors. That is, in coordinates, 
$$(A\circledast B)_{ij,i'j'} = A_{ij} B_{i'j'}.$$ 
We say that $W$ is Gaussian if for all finite subsets $\{z_1,\ldots,z_k\} \subset \Omega$, the random variable $(W(z_1),\ldots, W(z_k))$ is a Gaussian vector, seen as an element of $M_{r+s}(\dC)^k \simeq \dR^{ 2 (r+s) k}$. A Gaussian law in $\cP( \cH_{r+s}(\Omega))$ is characterized by its covariance kernel.

If $\Omega \cap \spectrum (S_n) = \emptyset$, we set for $z \in \Omega$, 
\begin{equation}\label{eq:defWn}
W_n(z) = \frac{\sqrt{n}}{\sigma_n} Q_n \left( R_{n}(z)  - R'_{n}(z) \right) P_n.
\end{equation}
Otherwise, we set $W_n(z) = 0$. The law of $(W_n(z))_{z \in \Omega}$ is an element of  $\cP ( \cH_{r+s}(\Omega))$. As we will state in the next paragraph, weak limits of $W_n$ govern the point processes of eigenvalues in $\Omega$.

A last definition:  we say that the sequence of unitary matrices $U_n = (U_{ij})$ is {\em flat at the border} if for any fixed integer $i \geq 1$, we have
$$
\lim_{n \to \infty} \max_{ j }  |U_{ij}| \vee \max_{ j } | U_{n+1 - i, j} |  = 0.
$$ 
For example, $U_n = I_n$ is not flat at the border but $U_n  = F_n$ is flat. A sequence of Haar distributed elements on the unitary or orthogonal group is flat at the border with probability one.

We are ready to state our main result (in a weaker form) related to the matrix field $W_n(z)$. We recall that if $A, B \in M_{r+s}(\dC)$, their tensor product $A \otimes B  \in M_{r+s}(\dC) \otimes M_{r+s}(\dC)$ has entry $(i,j),(i',j')$ given by
$$
(A \otimes B )_{ij,i'j'} = A_{ii'}B_{jj'}.
$$
\begin{thm}\label{th:WnW}
  Assume  $(H_0)$ and $(H_k)$ for all $k>2$ and either $(H_\text{sym})$ or $U_n = F_n$ defined in \eqref{eq:defFn}. Let $\Omega \subset \dC$ be an open bounded connected set such that $\bar \Omega \cap \supp(\beta_\sigma) = \emptyset$. Then the random variables $W_n$ is tight in $\cP ( \cH_{r+s}(\Omega))$ and any accumulation point $W$ is centered and satisfies  for all $z,z' \in \Omega$, 
  $$
 \dE  [ W(z) \circledast \bar W(z') ] = \frac{\cA(z,z') \otimes \cB(z,z')^{\intercal}}{1 - \sigma^2 \theta(z,z')} ,
  $$ 
  with 
  $$
  \theta(z,z') = \oint_{\dS^1} \frac{d\omega}{\omega ( z - \bfa(\omega)) \overline{( z' - \bfa(\omega))}},
  $$
  $$
  \cA(z,z') = \cD  \oint_{\dS^1}\begin{pmatrix}I_s &0\\0 & \omega^{-(r+s)}I_r\end{pmatrix} \cM(\omega)\begin{pmatrix}I_s &0\\0 & \omega^{r+s}I_r\end{pmatrix}\frac{d\omega}{\omega ( z - \bfa(\omega)) \overline{( z' - \bfa(\omega))}} \cD^*,
  $$
  and 
  $$
  \cB(z,z') = \cE^*  \oint_{\dS^1} \cM(\omega)\frac{d\omega}{\omega ( z - \bfa(\omega)) \overline{( z' - \bfa(\omega))}} \cE.
  $$
  If $\rho = 0$, we have $\dE [ W(z) \circledast W(z')] = 0$. For any $\rho$, if $U_n = F_n$ with $\veps = 1$ or if $U_n$ is in the orthogonal group with $\veps = -1$, we have
  $$\dE [ W(z) \circledast W(z')] =  \rho \frac{\cA_{\veps }(z,z') \otimes \cB_{\veps}(z,z')^\intercal}{1 - \rho \sigma^2 \theta_{\veps}(z,z')} ,$$
with 
  $$
  \theta_\veps (z,z') = \oint_{\dS^1} \frac{d\omega}{\omega ( z - \bfa(\omega))  ( z' - \bfa(\omega^{\veps}))},
  $$
  $$
  \cA_\veps (z,z') = \cD  \oint_{\dS^1}\begin{pmatrix}I_s &0\\0 & \omega^{-(r+s)}I_r\end{pmatrix} \cM_\veps (\omega)\begin{pmatrix}I_s &0\\0 & \omega^{-\veps(r+s)}I_r\end{pmatrix} \frac{d\omega}{\omega ( z - \bfa(\omega))  ( z' - \bfa(\omega^{\veps}))} \cD^\intercal,
  $$
  where $\cM_{-1}(\omega) = \cM(\omega)$ and $\cM_1 (\omega)_{p,q} = \omega^{p+q-2}$,
  and 
  $$
  \cB_\veps (z,z') = \cE^\intercal  \oint_{\dS^1} \cM'_\veps (\omega) \frac{d\omega}{\omega ( z - \bfa(\omega))( z' - \bfa(\omega^{\veps}))} \cE,
  $$
  where $\cM'_{-1}(\omega) = \cM(\omega)$ and $\cM'_1 (\omega)_{p,q} = \omega^{2s+2-(p+q)}$.
If $U_n$ is flat at the border, then any accumulation point is Gaussian (and thus $W_n$ converges in the above examples).  If $U_n = I_n$ then   $W_n$ converges weakly toward an explicit matrix field $W$.
\end{thm}

We postpone to Theorem \ref{th:cvWn} for a more precise asymptotic equivalent of $W_n(z)$. We will state that $$ W_n(z)  \stackrel{d}{=} P_n R'_n(z) U_n X_n U_n^* R'_n(z) Q_n + \sigma_n \tilde W^{(2)}_n(z) + o(1) = \tilde W^{(1)}_n(z) + \sigma_n \tilde W^{(2)}_n(z) + o(1),$$
where $\tilde W^{(2)}_n$ is a tight sequence of Gaussian fields, independent of $X_n$ with an explicit covariance kernel. In particular if $\sigma = 0$ only $\tilde W^{(1)}_n(z)$ survives. If $U_n = I_n$, this linear term will generically be non-Gaussian (if $X$ is non-Gaussian). The condition that $U_n$ is flat at the border is a sufficient condition to guarantee that  $\tilde W^{(1)}_n$ is asymptotically Gaussian. The examples depicted in Theorem \ref{th:WnW} are the simplest to consider for computing the limiting covariance kernels $\dE [ W(z) \circledast W(z')^{\intercal}]$. Interestingly, the basis $U_n$ plays an important role in the asymptotic distribution of $W_n$. This is transparent in the expression for $\tilde W^{(1)}_n$ but it is also the case for the covariance kernel of $\tilde W_n^{(2)}$ when $\rho \ne 0$ where the matrix $U_n U_n^{\intercal}$ appears, see again Theorem \ref{th:cvWn} for more details.

\paragraph{Convergence of outlier eigenvalues.} 

We now explain the connection between the accumulations points of the matrix field $W_n$ and the point processes of zeroes. This is the place where the winding number $\delta$ plays a central role.

We first need to recall some standard definitions (see \cite{horn-johnson,mneimne89,MR2567816} for instance). . For  a $n\times n$ matrix $A$, we define  its $k$-th compound matrix  as the $\binom{n}{k}\times\binom{n}{k}$ matrix, denoted by $\bigwedge\nolimits^k(A)$, with rows and columns  indexed by subsets of $\{1,\ldots,n\}$ of cardinal $k$ ordered by lexicographic order, and where for $I,J\subset\{1,\ldots,n\}$ of cardinal $k$,  its $(I,J)$-coefficient is given by $\det A(I|J)$, where $A(I|J)$ is the submatrix formed from $A$ by retaining only those rows indexed by $I$ and those columns indexed by $J$ (that is the $(I,J)$-minor of $A$). By convention, $\bigwedge\nolimits^0(A)=I$, and we have $\bigwedge\nolimits^n(A)=\det A$. Note also that $\bigwedge\nolimits^k(A)$ is exactly the $k$-th exterior power of $A$.

Now, the $k$-th adjugate of $A$, denoted by $\adj_k(A)$ is the $\binom{n}{k}\times\binom{n}{k}$ matrix whose $(I,J)$-coefficient is given by
\[
\adj_k(A)_{I,J} = (-1)^{\sum_{i\in I} i + \sum_{i \in J} j} \det A(J^c | I^c).
\]
It can be expressed in terms of compound matrices as follows. Let $S$ be the $n\times n$ sign matrix, given by $S=\operatorname{diag}(1,-1,\ldots,(-1)^{n-1})$. Let $\Sigma$ be the $\binom{n}{k}\times\binom{n}{k}$ exchange matrix, with $1$ on the anti-diagonal, and $0$ elsewhere:
\[
\Sigma=
\begin{pmatrix} 
 & & 1 \\
 & \iddots & \\
 1 & &
\end{pmatrix}.
\]
Then one has 
\[
\adj_k(A) = \Sigma \bigwedge\nolimits^{n-k}(SA S)^\intercal \Sigma.
\]
In particular, $\adj_0(A)=\det A$ and $\adj_1(A)$ is the usual adjugate matrix of $A$, that is the transpose of its cofactor matrix. We will notably use the following remark:
 $\adj_k(A)=0$ if and only if $\text{rank}(A)\leq n-k-1$.

For $z \notin \bfa (\dS^1)$, we set 
\begin{equation}\label{defH}
\cH(z)
 = \frac{1}{2 i \pi} \mathcal{D}\oint_{\mathbb{S}^1}  
\begin{pmatrix}
\omega^sI_s & 0 \\
0 & \omega^{-r}I_r
\end{pmatrix}
 \frac{\cM(\omega)}{\omega(z-\bfa(\omega))} d\omega \;
\mathcal{ E},
\end{equation}
where $\cM(\omega)$, $\cD$ and $\cE$ are defined in \eqref{eq:defcM}.

For $\sigma \geq 0$ and $\delta$ relative integer, let 
$$
\cR_{\sigma,\delta} = \{ z \in \dC \backslash  \supp (\beta_\sigma) :   \wind(\bfa-z)=\delta \}.
$$
The following theorem is our main result on the outliers eigenvalues of $M_n$.
 \begin{thm} \label{th:W2Out}
 Under the assumptions of Theorem \ref{th:WnW},  let $\Omega \subset \dC$ be an open bounded connected subset such that  $\bar \Omega \subset \cR_{\sigma,\delta}$ for some $\delta \ne 0$. Let $W$ be an accumulation point of $W_n$ such that 
 $$
\varphi(z) =  \Tr\left( \adj_{|\delta|} \left( I_{r+s} + \mathcal H(z)  \right)  \bigwedge\nolimits^{|\delta|} \left(   W(z) \right) \right)
 $$
 is a.s.\,not the zero function on $\Omega$. Then, along this accumulation sequence, the point process of eigenvalues of $M_n$ in $\Omega$ converges weakly toward the point process of zeroes of $\varphi$.  Moreover, $\varphi$ is a.s.\,not the zero function on $\Omega$ when $W$ is Gaussian or when $U_n = I_n$ and \Hac holds.
  \end{thm}

We will prove in Proposition \ref{prop:rank} that $\adj_{|\delta|} \left( I_{r+s} + \mathcal H(z)  \right) $ is non-zero for all $z \in \cR_{\sigma,\delta}$. We emphasize the non-universality of this result, since the  random analytic function whose zeroes approximate the outliers of $M_n$ depends on the distribution of the coefficients of $X_n$.  Moreover, even when $X_n$ is a Ginibre matrix, this   random analytic function  is in general not Gaussian, unless $|\delta|=1$. These phenomena were already uncovered  in \cite{basak-zeitouni20}.

\begin{figure}[t]
    \centering
    \subfloat[]{\includegraphics[width=5cm]{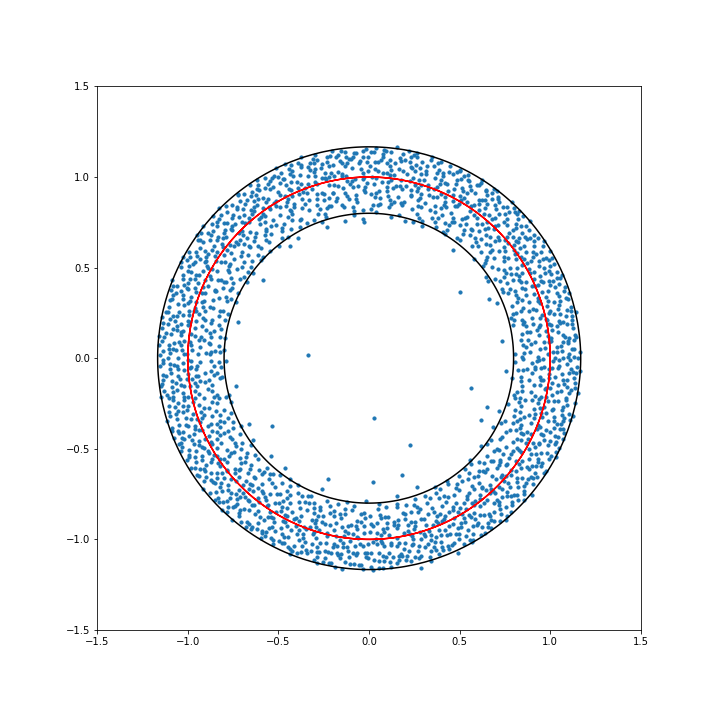} }
    \quad
    \subfloat[]{\includegraphics[width=5cm]{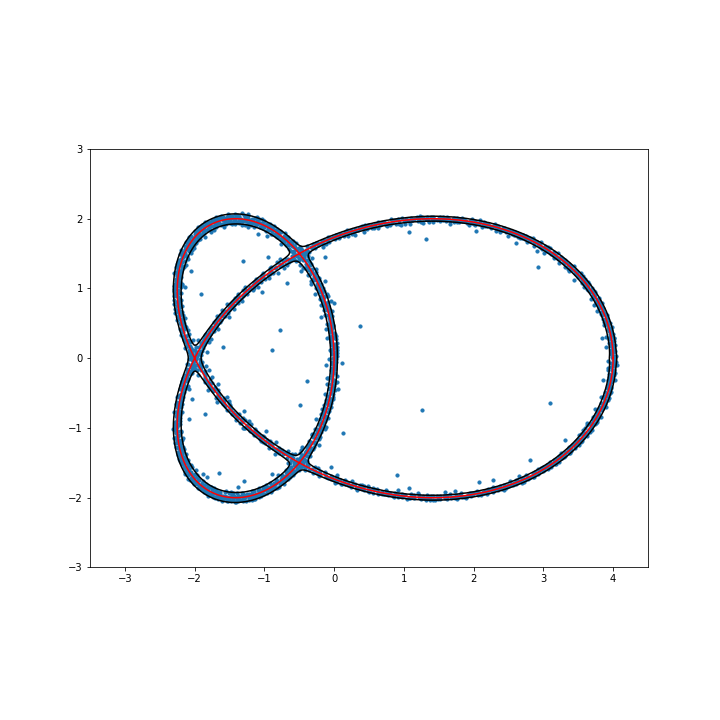} }
    \quad
    \subfloat[]{\includegraphics[width=5cm]{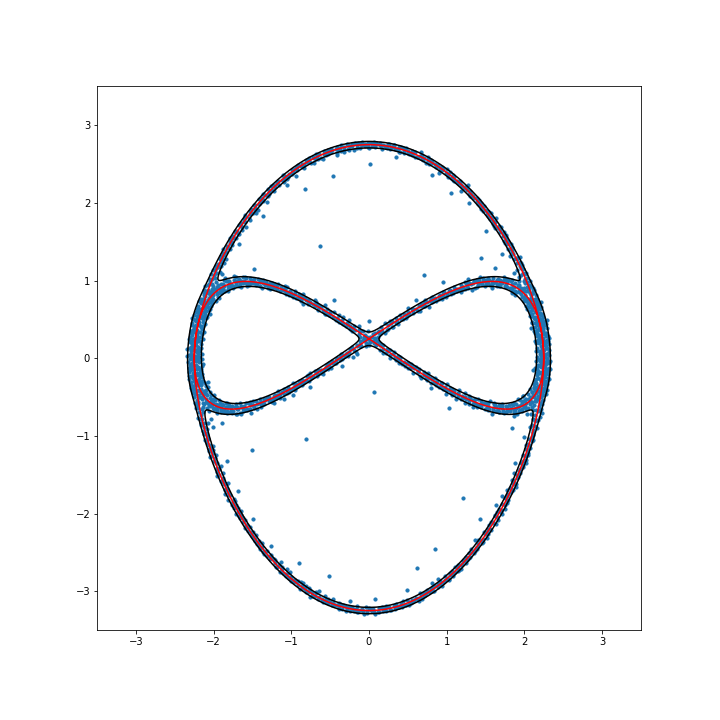} }
    \caption{Eigenvalues (blue dots) of $T_n(\bfa)+\sigma \frac{X_n}{\sqrt n}$ for $n=2000$, $\sigma=0.6$ and various symbols: (a) $\bfa(t)=t^2$, (b) $\bfa(t)=t+2t^2+t^{-1}$ and (c) $\bfa(t)=-\frac34 t + \frac54 i t^2 + \frac34 t^3 + \frac34t^{-1} - i t^{-2} -\frac34 t^{-3}$. Note that there are no outliers in the bounded region of zero winding number in subfigure (c). }
    \label{fig:simu}%
\end{figure}

Figure~\ref{fig:simu} shows simulations of eigenvalues of the model $M_n$ with $\sigma>0$ and $U_n = I_n$. In region with non-zero winding number, some eigenvalues detach from the support of the limiting measure $\beta_\sigma$. 

\subsection{Overview of the proofs}
\label{subsec:overview}

We outline here the proofs of Theorem \ref{thm-no-outliers}, Theorem \ref{th:WnW} and Theorem \ref{th:W2Out}.

\paragraph{Step 1: Sylvester's determinant identity. } We fix a compact connected set $\Gamma \subset \dC \backslash \supp(\beta_\sigma)$. As shown in \cite{bordenave-capitaine16}, normal matrices have a stable spectrum with respect to noise. As we will see with Proposition~\ref{nooutlier} in Section \ref{sec:stable}, it will follow that a.s. for all $n$ large enough, $S_n$ has no eigenvalue in $\Gamma$. On this last event, since $M_n = S_n + B_n = S_n - P_n Q_n$, we may write 
\begin{align}
\det (z -M_n)  & = \det (z -S_n) \times \det (I_n - B_n R_n(z))\nonumber \\
& = \det (z-S_n) \times \det (I_{r+s} + Q_n R_n(z) P_n),\label{proddet}
\end{align}
where the second equality follows from the classical Sylvester's determinant identity:  for $A \in M_{p,q}(\C)$ and $B \in M_{q,p}(\C)$, 
\begin{equation}\label{sylvester-identity}
\det (I_p + AB) = \det (I_q + BA). 
\end{equation}

Since $\det (z -S_n) \not = 0$ on $\Gamma$, $z$ is an eigenvalue of $M_n$ if and only if $z$ is a zero of the $(r+s) \times (r+s)$ determinant
\[
 \det (I_{r+s}+ Q_n R_n(z) P_n). 
\]

\paragraph{Step 2: Taylor expansion for the determinant. }
We thus have to understand the zeroes of the above determinant.  This will depend on whether $z$ is in $\spectrum T(\bfa)$ or not. We introduce the difference matrix 
\begin{equation} \label{dvptSigma0}
G_n(z) = \frac{\sigma_n W_n(z) }{\sqrt n} =Q_n ( R_{n}(z)  - R'_{n}(z) ) P_n. 
\end{equation}
It follows from Proposition~\ref{convunif} below,   that a.s. 
\begin{equation*}\label{eq:G20}
\lim_{n \to \infty} \sup_{ z \in \Gamma} \| G_n(z) \| = 0.
\end{equation*}
By expanding the determinant,  we have, 
\begin{equation} \label{dev-jacobi}
 \det (I_{r+s} + Q_n R_n(z) P_n) =  
   \sum_{k=0}^{r+s} \frac{1}{k!} d^k \det(I_{r+s} + Q_n R'_n(z) P_n) (  G_n(z),\ldots, G_n(z) ),
\end{equation}
where $ d^k \det(A) $ denotes the $k$-th derivative of the determinant at $A$.  
The $k$-th derivative of the determinant can be expressed in terms of compound and higher adjugate matrices (\cite{mneimne89,bhatia-jain09}). The determinant of the sum of two $n\times n$ matrices $A,B$ can be expressed as:
\[
\det(A+B) = \sum_{k=0}^n \Tr \left( \adj_k(A) \bigwedge\nolimits^k(B)  \right),
\]
hence, the above development \eqref{dev-jacobi} writes
\begin{equation} \label{dev-jacobi2}
 \det (I_{r+s} + Q_n R_n(z) P_n) = \sum_{k=0}^{r+s} \Tr \left( \adj_k\left( I_{r+s} + Q_n R'_n(z) P_n \right) \bigwedge\nolimits^k\left( G_n(z) \right)  \right).
\end{equation}

\paragraph{Step 3: Case $\delta = 0$ and Szeg\H{o}'s strong limit theorem. } 
The first term $k= 0$ of the above development \eqref{dev-jacobi2} is equal to
\[
\det ( I_{r+s} + Q_n R'_n(z) P_n ) =  \frac{\det( T_n(\bfa) - z)}{\det( C_n(\bfa) - z )},
\]
by Sylvester's identity. Since $\|G_n(z)\|$ goes to 0, we deduce that, uniformly on $\Gamma$,  
$$
\det (I_{r+s} + Q_n R_n(z) P_n) =  \frac{\det( T_n(\bfa) - z)}{\det( C_n(\bfa) - z )} + o(1).
$$
Now, in the region of spectral stability, i.e. if  $\Gamma\subset\spectrum T(\bfa)^c$, Szeg\H{o}'s strong limit theorem implies that the above quotient  is uniformly lower bounded (we shall also give an alternative proof of this fact in Section \ref{sec:unstable}). We deduce that $\det (I_{r+s} + Q_n R_n(z) P_n)$  has no zeroes in $\Gamma$ and this will lead to Theorem \ref{thm-no-outliers}.

\paragraph{Step 4: case $\delta \ne 0$, identification of the leading term of the Taylor expansion \eqref{dev-jacobi2}. } On the other hand, inside $\spectrum T(\bfa)$, one can show that $\det ( I_{r+s} + Q_n R'_n(z) P_n )$ goes to zero exponentially fast as $n$ goes to infinity, so that the outliers will be described by higher terms of the development of $\det (I_{r+s} + Q_n R_n(z) P_n)$ in \eqref{dev-jacobi2}. In the spectral region $\mathcal R_{\sigma,\delta} = \{ z \notin \supp(\beta_\sigma) : \mathrm{wind} (\bfa -z) = \delta \}$, we will prove that this leading term corresponds precisely to $k = |\delta|$ and get 
\begin{equation}\label{eq:deltalead}
\det (I_{r+s} + Q_n R_n(z) P_n) \approx   \Tr \left( \adj_{|\delta|}\left( I_{r+s} + Q_n R'_n(z) P_n \right) \bigwedge\nolimits^{|\delta|} \left( G_n(z) \right)  \right).
\end{equation}
The reason is that $Q_n R'_n(z) P_n$ converges toward $\cH(z)$ defined in \eqref{defH} exponentially fast (this is Proposition \ref{convH}) and that the dimension of the kernel of  $I + \cH(z)$ is equal to $|\delta|$. This last claim is the content of Proposition \ref{prop:rank} and its proof has some subtleties since we do not have a closed-form formula for the eigenvalues of $\cH(z)$ when $r,s >0$.

\paragraph{Step 5:  convergence of the matrix field. }  We will prove  that the $\Gamma \to M_{r+s}(\dC)$ analytic random field $W_n(z) = \sqrt n G_n(z) / \sigma_n$ defined in \eqref{eq:defWn} is closed in distribution  to some random field as $n\to \infty$. This is the content of Theorem \ref{th:cvWn} and its corollary, Theorem \ref{th:WnW}. The proof of this result will be most challenging in the case  $\sigma >0$. To  prove this, we  expand $W_n(z)$ as 
\begin{equation}
W_n(z) = \frac{\sqrt n G_n(z)}{\sigma_n}  = \frac{\sqrt n}{\sigma_n} \sum_{k\geq 1} Q_n \left( R'_n(z) \sigma_n Y_n \right)^k R'_n(z) P_n. \label{dvptSigma}
\end{equation}
We will prove the tightness of this expansion and establish a functional central limit theorem using Stein's method of exchangeable pairs for the process in $(k,z) \mapsto Q_n \left( R'_n(z) \sigma_n Y_n \right)^k R'_n(z) P_n$. The proof will follow from a general Gaussian approximation result for processes of the type: 
$$
(B_0, B_1,\ldots, B_{k-1}) \mapsto n^{-(k-1)/2} \TR( B_0 X_n B_1  X_n \cdots  B_{k-1} X_n ),
$$
where $B_t \in M_n(\dC)$ and $B_0$ has small rank. This study will be done in Section \ref{sec:CLT} and it contains the main new probabilistic ingredients of this paper. The proof of tightness for the series \eqref{dvptSigma} will notably rely on a known deterministic upper bound on graphical sum-product expressions of the form 
\begin{equation*} 
\sum_{ \bm i  } \prod_{e \in E} (M_e)_{i_{e_-}i_{e_+}},
\end{equation*}
where $G = (V,E)$ is a directed graph, $(M_e)_{e \in E}$ is a collection of matrices of size $n$ indexed by the edges of $G$ and the sum is over all $\bm i = (i_v)_{v \in V}  \in \{1,\ldots,n\}^V$, see Theorem \ref{th:sumproduct} in Appendix  (this is where the restricted assumption that  either $(H_\text{sym})$ holds or $U_n = F_n$ is used). 

\paragraph{Step 6: case $\delta \ne 0$, convergence of the outlier eigenvalues. } 
Combining \eqref{eq:deltalead} and the convergence of $W_n$ toward a limiting matrix field $W$, we will deduce that if $\Gamma \subset \mathcal R_{\sigma,\delta}$, we have the following weak convergence of random analytic functions on $\Gamma$:
\begin{equation*} 
\left( \frac{\sqrt  n}{\sigma_n} \right)^{|\delta|}   \det (I_{r+s} + Q_n R_n(z) P_n) \to \varphi (z),
\end{equation*}
where $\varphi$ is an in Theorem \ref{th:W2Out}: $$
\varphi (z) =  \Tr\left( \adj_{|\delta|} \left( I_{r+s} + \mathcal H(z)  \right)  \bigwedge\nolimits^{|\delta|} \left(   W(z) \right) \right).
 $$ 
It follows classically from Montel's Theorem that, for analytic functions: $\varphi_n \to \varphi$ and $\varphi \ne 0$ imply  the convergence of the set $\varphi_n^{-1} (\{0\})$ toward the set  $\varphi^{-1}(\{ 0\})$ (see Proposition \ref{Shirai:zeroes} for the precise statement we use). Consequently, the main  remaining difficulty to prove Theorem \ref{th:W2Out} will be to prove that the above random analytic function $\varphi$ does not vanish almost surely. This will be done by proving that the covariance kernel  $\dE  [ W(z) \circledast \bar W(z) ]$ is positive definite and using a Gaussian comparison trick to lower bound the variance of $\varphi$.

\paragraph{Comparison with other methods. } Our method of proof follows the strategy initiated in \cite{bordenave-capitaine16}. For outliers, it settled the case $\sigma >0$, $r=0$, $s=1$ and $U_n = F_n$. Here we cope with the numerous new challenges outlined above. Other methods to study random Toeplitz matrices in the case $\sigma = 0$ were introduced in \cite{SjostrandVogel} and \cite{basak-zeitouni20}. Interestingly,  \cite{basak-zeitouni20} also relies on the expansion of the determinant but, with our notation, the authors write directly: 
$$
\det ( M_n - z ) = \det ( T_n(\bfa ) - z + \sigma_n Y_n) = \sum_{k=0}^n \sigma_n^k \Tr \left( \adj_k(T_n(\bfa ) - z) \bigwedge\nolimits^k(Y_n)  \right).
$$
They prove the convergence of this expansion when $\sigma_n = n^{-\gamma}$, $\gamma > 0$ and $U_n = I_n$. In the case $\sigma > 0$, we do not expect the tightness of this expansion anymore. A remarkable advantage of  \cite{basak-zeitouni20} is that they only assume that $(H_0)$ holds and a density assumption of the law of $X$. A feature of our present method is that it disentangles clearly the algebraic deterministic part governed by $\adj_{|\delta|}\left( I_{r+s} + Q_n R'_n(z) P_n \right)$ and the random matrix part governed by the matrix field $W_n(z)$ in $M_{r+s}(\dC)$. We have not found a simple way to compare the limiting random analytical functions we have obtained (which have the same random set of zeros).

\paragraph{Some perspectives. } The method of proofs are rather general and apply for their main parts when $T_n(\bfa)$ is replaced by $C_n - P_n Q_n$,  a finite rank perturbation of a well-conditioned matrix $C_n$ (up to the explicit asymptotic computations on the covariance of the matrix field $W_n$ and the asymptotic properties of the restricted resolvent $Q_n (z - C_n)^{-1} P_n$).  

This work leaves many interesting directions to further researches such as the study of the left and right eigenvectors of the outlier eigenvalues. Also, the case  $\sigma_n = \alpha^n$, $0 < \alpha < 1$, seems to be the critical range for the noise to start to interfere with the spectrum of $T_n(\bfa)$. In another direction, this work shows that the nature of the noise has an important influence on the outliers. It would be interesting to replace $Y_n$ with a more structured noise matrix, such as an Hermitian random matrix. We note also that if the rank of the perturbation $P_n Q_n$, here $r+s$, grows with $n$ many arguments of this work do not apply directly, we may mention general Toeplitz matrices as motivating examples fitting in this category. We finally mention the general study of the outlier eigenvalues of a polynomial in deterministic and random matrices.

\paragraph{Organization of the paper. } The remainder of the paper is organized as follows. In Section \ref{sec:1}, we recall some facts on Toeplitz matrices. In Section \ref{sec:LSD}, we study $\beta_\sigma$ and prove Lemma \ref{le:supportbeta}. In Section \ref{sec:stable}, we study the stable region and prove Theorem \ref{thm-no-outliers}. Section \ref{sec:unstable} is devoted to the proofs of Theorem \ref{th:WnW} and Theorem \ref{th:W2Out}. 
In Section \ref{sec:CLT}, we establish a functional central limit theorem for traces of polynomials of random matrices. Finally Section \ref{sec:appendix} in an appendix which gather known results.

\paragraph{Acknowledgments} We would like to thank Ofer Zeitouni for introducing us to the topic of Toeplitz matrices and for some stimulating discussions.

\section{Notation and preliminaries}
\label{sec:1}
\subsection{Notation and matrix identities}

The spectral norm will be denoted $\| \cdot \|$, that is, for $A \in M_n(\C)$,
\[
\| A \| := \sup_{\|x \|_2=1 } \| Ax \|_2,
\]
where $\| \cdot \|_2$ denotes the Euclidean norm on $\C^n$.

We  will use the following perturbation inequality, which follows from Hadamard's inequality: for $A,B \in M_d(\C)$,
\begin{equation} \label{perturb-inequality}
|\det A - \det B| \leq \| A-B \| d \max(\|A\|,\|B\|)^{d-1}. 
\end{equation}


\subsection{Preliminaries on Toeplitz matrices}

We first  recall some well known facts on banded Toeplitz matrices, which can be found in the book \cite{bottcher-grudsky05}. We start with a useful remark: for any $z\in \C$, $T_n({\bf a})-z I=T_n({\bf a}-z)$, which is the Toeplitz matrix with symbol $\lambda\mapsto {\bf a}(\lambda)-z$. 

\subsubsection{Limiting empirical distribution of the singular values} 

By the Avram-Parter Theorem, see \cite[Chapter~9]{bottcher-grudsky05}, the empirical distribution of the singular values of $T_n({\bf a}-z)$ converges weakly to the push-forward by the map $t\mapsto |{\bf a}(t)-z|$ of the uniform measure on $\mathbb S^1$, that is, for any $z\in\C$,
\begin{equation}\label{nuz}
\mu_{(T_n({\bf a})-z)(T_n({\bf a})-z)^*} \overset{\text{weakly}}{\underset{n\to\infty}{\longrightarrow}} \nu_z,
\end{equation}
where $\nu_z$ is the push-forward by the map $t\mapsto |{\bf a}(t)-z|^2$ of the uniform measure on $\mathbb S^1$.

\subsubsection{Limiting empirical spectral distribution and limiting set of spectra}\label{subsec:limitESD}
Let us now recall that the empirical eigenvalue distribution of $T_n(\bfa)$ converges weakly (see \cite[Chapter~11]{bottcher-grudsky05}).
We introduce the following notation for the zeroes of the symbol ${\bf a}-z$:
\begin{definition}\label{zk}
For any $z\in \C$,
\[
{\bf a}(\lambda)-z = \sum_{k=-r}^s a_k \lambda^k -z= \lambda^{-r } \left(  \sum_{k=0}^{r+s} a_{k-r} \lambda^k -z \lambda^r    \right) = \lambda^{-r} a_s \prod_{k=1}^{r+s} \left(   \lambda  -  \lambda_k(z) \right),
\]
where the $\lambda_k(z)$ are the roots of the polynomial $Q_z(\lambda)= \sum_{k=0}^{r+s} a_{k-r} \lambda^k -z \lambda^r$, counting with multiplicities, and labelled such that:
\[
|\lambda_{1}(z)|  \leq \cdots \leq  |\lambda_{r}(z)|  \leq |\lambda_{r+1}(z)|  \leq \cdots \leq  |\lambda_{r+s}(z)|  .
\]
\end{definition}
\begin{rem} \label{rk-winding-zero}
 Since ${\bf a}-z$ has only one pole of order $r$, the winding number $\wind({\bf a}-z)$ is given by the number of roots of $Q_z$ in $\mathbb D$ minus $r$. In particular, we have $\wind({\bf a}-z)=0$ if and only if 
\[
 |\lambda_{r}(z)| < 1 < |\lambda_{r+1}(z)|.\]  Define the set
\[
\Lambda({\bf a}) = \left\{z \in \C  :  |\lambda_r(z)| = |\lambda_{r+1}(z)| \right\}.
\]
\end{rem}
By Schmidt and Spitzer, see \cite{schmidt-spitzer60} or \cite[Chapter~11]{bottcher-grudsky05},  the set $\Lambda({\bf a})$ is the limiting set of the spectra $\spectrum T_n({\bf a})$, 
and for any compact set $\Gamma \subset \C \setminus \Lambda(\bfa)$, there are no eigenvalues of $T_n(\bfa)$ in $\Gamma$ provided that $n$ is large enough.

 The set $\Lambda({\bf a})$ is easily seen to be a compact subset of $\spectrum T(\bfa)$ and  consists of a finite union of analytic arcs and a finite number of exceptional points.  It is a connected set and has two-dimensional Lebesgue measure equal to  zero. 
 If $z \not \in \Lambda(\bfa)$, there exists $\rho$ such that $|\lambda_r(z)|< \rho < |\lambda_{r+1}(z)|$. Consider the perturbed symbol $\bfa_\rho$ defined by $\bfa_\rho(t)=\sum_{k=-r}^s a_k \rho^k t^k$ and 
 denote by $g$ the function on $\C \setminus \Lambda({\bf a})$ defined by
\[
g(z) = \exp \left( \frac{1}{2\pi} \int_0^{2\pi} \log | {\bf a}_\rho(e^{i\theta}) - z| d\theta \right).
\]
Then the empirical spectral measure of $T_n({\bf a})$ converges weakly to 
\[
\frac{1}{2\pi} \nabla^2 \log g,
\]
in the distributional sense (where $\nabla^2$ denotes the Laplacian). Moreover, Hirschman, see \cite{Hirschman67} or \cite[Chapter~11]{bottcher-grudsky05}), proved that the empirical spectral measure of $T_n({\bf a})$ converges weakly to a measure supported on $\Lambda(\bfa)$, whose density is  given explicitly by a formula involving the boundary behavior of the geometric mean of $\bfa  - z$, as $z$ approaches the different analytic arcs which constitute $\Lambda(\bfa)$. 

\begin{rem}
It is interesting to point out that the set $\Lambda(\bfa)$ will not play any role in the asymptotic spectral analysis of the matrix $M_n$ (at least in the regime $\lim_n \frac{1}{n} \ln \sigma_n = 0$ that we consider here).
\end{rem}

\subsubsection{Szeg\H{o}'s strong limit theorem} Recall now the Szeg\H{o}'s strong limit theorem. Let ${\bf b}$ be the Laurent polynomial ${\bf b}(t)=\sum_{k=-r}^s b_k t^k$ and suppose that 
\[
{\bf b}(t) \not= 0 \ \text{for $t\in \mathbb S^1$ and } \wind {\bf b} =0.
\]
Since $\wind {\bf b}=0$ and ${\bf b}$ has only one pole of order $r$, we have exactly $r$ roots (counting with multiplicities) $\delta_1,\ldots, \delta_r$ of ${\bf b}$ with modulus strictly less than 1, and $s$ roots $\mu_1,\ldots,\mu_s$ with modulus strictly greater than 1.  Define:
\begin{align*}
 G({\bf b})  = b_s (-1)^s \prod_{j=1}^s \mu_j \quad \hbox{ and } \quad 
 E({\bf b})  = \prod_{i=1}^r \prod_{j=1}^s \left( 1 - \frac{\delta_i}{\mu_j}  \right)^{-1} .
\end{align*}

The Szeg\H{o}'s strong limit theorem \cite[Thm~2.11]{bottcher-grudsky05} states that:
\begin{thm}\label{szego-thm}
Let ${\bf b}$ as above with  $\wind {\bf b} =0$. 
Then, as $n\to\infty$,
\[
\det T_n({\bf b}) = G({\bf b})^n E({\bf b}) \left( 1 + O(q^n)\right),
\]
with some constant $q \in \left(\frac{\max_i |\delta_i|}{\min_j |\mu_j|} , 1\right)$.
\end{thm}
Note that $G({\bf b})$ is exactly the exponential of the $0$-th Fourier coefficient of $\ln {\bf b}$, that is
\[
G({\bf b}) = \exp (\ln {\bf b})_0 = \exp \left( \frac{1}{2\pi} \int_0^{2\pi} \ln {\bf b}(e^{i\theta}) d \theta  \right).
\]
\begin{rem}In Subsection \ref{subsec:QrP} and Subsection \ref{subsec:rk}, we will explore variations around Szeg\H{o}'s strong limit theorem. We will notably give a new proof that $\det T_n({\bf b}) / G({\bf b})^n$ converges to a  non-zero number if $\wind {\bf b} =0$.
\end{rem}

\subsubsection{Circulant matrices}

According to  \cite[(2.2)~Chapter~2]{bottcher-grudsky05} (note that with our notations we are considering the transpose instead of their matrix), $C_n(\bfa)$ is diagonalized by a discrete Fourier transform, that is,
\begin{equation}\label{eq:diagCn}
C_n(\bfa)= F_n \diag(\bfa(\omega_n^{l-1})), l=1,\ldots,n) F_n^*,
\end{equation}
where $\omega_n=e^{\frac{2\iC\pi}{n}}$ and $F_n$ was defined in \eqref{eq:defFn}.


\section{Limiting spectral distribution}

\label{sec:LSD}

\subsection{Convergence of the empirical spectral distribution}
\label{model}

As already pointed, for $\sigma = 0$, $\sigma_n = n ^{-\gamma}$ for some $\gamma > 0$, \cite{BasakPaquetteZeitouni20} proved that the empirical spectral measure  of $M_n$ converges weakly, in probability to the law of $\bfa(u)$,  with support $\bfa(\mathbb S^1)$.

For the rest of the subsection, we focus on the case $\sigma  >0$ and explain that the empirical spectral measure of $M_n$ converges weakly, in probability. Denote by  $N_n$  the nilpotent matrix with $(N_n)_{i,i+1}=1$ and zero elsewhere. Then $T_n({\bf a})$ is the following polynomial in $N_n$ and $N_n^*$:
\[
T_n({\bf a}) =  \sum_{k=0}^s a_k (N_n) ^k + \sum_{k=1}^{r} a_{-k} (N_n^*)^k.
\]
Since, $N_n$ converges, as $n\to\infty$, in $^*$-moments to a Haar unitary $u$
in some noncommutative probability space $(\mathcal A, \phi)$, where $\mathcal A$ is $C^*$-algebra and $\phi$ a normal faithful tracial state on $\mathcal A$,
 then $T_n({\bf a})$ converges in $^*$-moments to 
\[
{\bf a}(u)= \sum_{k=0}^s a_k (u) ^k + \sum_{k=1}^{r} a_{-k} (u^*)^k .
\]

By Theorem 6 in  \cite{sniady02}, if $G_n$ is a  complex Ginibre random matrix, i.e. $G_n(i,j)$ are i.i.d. standard complex Gaussian variables, then one has  that the empirical distribution of $\sigma_n \frac{G_n}{\sqrt{n}}+T_n({\bf a})$ converges weakly almost surely to the Brown measure of ${\bf a}(u)+\sigma c$ where $c$ and $u$ are free noncommutative random variables in  $(\mathcal A, \phi)$,  where $c$ is a standard circular noncommutative random variable and $u$ is Haar distributed (see for instance \cite{mingo-speicher} for a nice introduction to Brown measure in the context of free probability). 

Moreover, by the universality principle Theorem 1.5 of Tao-Vu \cite{tao-vu08}, the same holds for  $X_n$. Indeed, the Hilbert-Schmidt norm of $T_n({\bf a})$ satisfies
\[
\|T_n({\bf a})\|^2_{\text{HS}} := \Tr(T_n({\bf a})T_n({\bf a})^*) \leq n \|T_n({\bf a})\|^2 \leq n \max_{\omega \in\mathbb S^1} |{\bf a}(\omega)|^2,
\]
where $\|\cdot\|$ denotes the spectral norm. Hence,
\[
\sup_n \frac1n \|T_n({\bf a})\|^2_{\text{HS}} < \infty,
\]
so condition (1.3) of \cite[Thm~1.5]{tao-vu08} is satisfied. Moreover, 
since according to \eqref{nuz}, $\mu_{(T_n({\bf a})-z)(T_n({\bf a})-z)^*}$ converges weakly to $\nu_z$, 
  Theorem~1.5 in \cite{tao-vu08} implies that the empirical eigenvalue distribution of $M_n= U^*_n \sigma_n ( \frac{X_n}{\sqrt n} + U_n T_n({\bf a} ) U^*_n / \sigma_n ) U_n$ converges weakly  towards the Brown measure $\beta$ of ${\bf a}(u)+\sigma c$, almost surely. 
  

The support and the density of the Brown measure for the free sum of a circular element and a  normal operator can be computed using subordination techniques as in \cite{bordenave-caputo-chafai14} and this has been generalized to free sums of a circular element and an arbitrary operator in  \cite{zhong21,bercovici-zhong22}. From their results, one extracts that in our case,  the Brown measure of $\bfa(u) + \sigma c$ is supported on the closure of the open set
\[
A=\left\{z \in \mathbb{C} :  \frac{1}{2\pi}\int_0^{2\pi} \frac{d\theta}{\vert {\bf a}(e^{i\theta})-z\vert^2}  > \sigma^{-2}\right\}
\]
and is absolutely continuous with respect to Lebesgue measure on  $A$ with density given by
\begin{align*}
\rho(z)&=\frac{1}{\pi} \Bigg\{ \omega(z)^2 \E\left[ \left( |\bfa(U)-z|^2 + \omega(z)^2  \right)^{-2} \right]   \\
& \quad\quad\quad\quad\quad  + \frac{ \left| \E \left[ ( \bfa(U)-z ) \left( |\bfa(U)-z|^2 + \omega(z)^2  \right)^{-2}  \right] \right|^2 }{   \E\left[ \left( |\bfa(U)-z|^2 + \omega(z)^2  \right)^{-2} \right]  }   \Bigg\}
\end{align*}
where $U$ is uniformly distributed on $\mathbb S^1$, and $\omega\colon A \to (0,+\infty)$ is  the analytic function uniquely determined by
\[
\E\left[ \left( |\bfa(U)-z|^2 + \omega(z)^2  \right)^{-1} \right] = \sigma^{-2}, \quad \text{for any $z\in A$.}
\]

\subsection{Proof of Lemma \ref{le:supportbeta}}
For any $z \in \C$, put
\[
I(z) = \frac{1}{2 \pi} \int_0^{2 \pi }  \frac{d\theta}{|\bfa(e^{i \theta})-z|^2} . 
\]
We have that $I$ is continuous and twice differentiable on the open set $\bfa(\mathbb S^1)^c$, and $I(z)=\infty$ on $\bfa(\mathbb S^1)$. 

If $\sigma = 0$, then $\beta_0$ is the distribution of $\bfa(U)$ with $U$ uniformly distributed on $\dS^1$. In particular, we get
\begin{equation}\label{eq:suppb0}
\supp(\beta_0) = \bfa (\dS^1) = {\left\{z \in \mathbb{C} : I(z)  = \infty\right\}}.
\end{equation}
The conclusion of the lemma for $\sigma = 0$ follows.

For $\sigma >0$. as stated above, one has:
\[
\supp(\beta_\sigma)= \overline{\left\{z \in \mathbb{C} :  I(z)  > \sigma^{-2}\right\}}.
\]

Recall that complex differentiation gives $\partial_z z =  1$ and $\partial_z \overline z =0$. Hence,
\[
\partial_z \partial_{\overline z}  I(z) = \frac{1}{2 \pi} \int_0^{2 \pi} \frac{d\theta}{|\bfa(e^{i \theta})-z|^4}. 
\]
Therefore, since the Laplacian is $\Delta= 4  \partial_z \partial_{\overline z} $, we have that $\Delta I >0$ on $\bfa(\mathbb S^1)^c$, so $I$ is strictly subharmonic on $\bfa(\mathbb S^1)^c$ and $I$ has no local maximum. Indeed, 
let $U$ be some open ball in $\bfa(\mathbb S^1)^c$. Since $I$ is continuous, $I$ attains its maximum at some point  $z_0 \in \overline U$. Now suppose by contradiction that $z_0 \in U$. Then, by Taylor formula, since $\nabla I(z_0)=0$, one has
\[
I(z_0 + \xi) = I(z_0) + \frac12  \xi^T H_I(z_0) \xi  + o(\| h \|^2),
\]
where $H_I$ denotes the Hessian matrix of $I$. 
Hence, the Hessian $H_I$  is negative semidefinite at $z_0$, so its trace is nonpositive. But since $\Delta I (z_0)=\Tr(H_I(z_0))$, we arrive at a contradiction.  

Let
\[
A= \{  z \in \C :   I(z) > \sigma^{-2}\}. 
\]
Note that $A$ is open by Fatou's lemma. Note also that $A$ contains $\bfa(\mathbb S^1)$, since $I(z)=+\infty$ for $z \in \bfa(\mathbb S^1)$. 
Denote by $B$ the set
\[
B= \{  z \in \C :   I(z) \geq  \sigma^{-2}\}. 
\]
The inclusion $\overline A \subset B$ easily  follows  by continuity. 
For the reverse inclusion,  pick $z_0 \in \text{Int}(A^c)$, the interior of $A^c$ (note that $\text{Int}(A^c)$ is non empty since $I(z)< \sigma^{-2}$ for $z$ large enough).
 Hence, there is an open ball $B(z_0,\delta) \subset A^c$, so for all $z \in B(z_0,\delta)$, $I(z) \leq \sigma^{-2}$. Assume by contradiction that $z_0 \not \in B^c$, i.e.  $I(z_0) \geq \sigma^{-2}$. Hence, $I(z_0)= \sigma^{-2}$, so $I$ admits a local maximum inside  $B(z_0,\delta)$ which is not possible since $I$ is strictly subharmonic. Hence, $z_0 \in B^c$, and eventually $\overline A = B$ proving the lemma for $\sigma > 0$ \qed

We remark that the last identity in \eqref{eq:suppb0} implies that for all $\sigma \geq 0$.
\begin{equation}\label{eq:supportbeta}
\supp(\beta_\sigma) = a(\dS^1) \cup \left\{ z \in \dC :  \frac{1}{2 \pi} \int_0^{2 \pi }  \frac{d\theta}{|\bfa(e^{i \theta})-z|^2} \geq \sigma^{-2} \right\}.
\end{equation}

\section{Stable spectrum}

\label{sec:stable}

\subsection{Resolvent bounds}

Recall the definition of $S_n = C_n(\bfa) + \sigma_n Y_n$ in \eqref{eq:defSn}. The goal of this subsection is to prove the following propositions about the resolvents of $S_n$ and $C_n(\bfa)$ denoted in \eqref{defgeneralresolvantes} by $R_n(z)$ and $R'_n(z)$ respectively.

\begin{prop}\label{nooutlier}
Assume that $(H_4)$ holds. 
For any $\sigma \geq 0$ and any compact set  $\Gamma \subset \dC \backslash \supp(\beta_\sigma)$, a.s. for all $n$ large enough, neither $C_n(\bfa)$ nor  $S_n$ has eigenvalues in $\Gamma$. Moreover, there exists $C_\Gamma>0$ such that 
 \begin{equation}\label{borneresolvante} \sup_{z\in \Gamma} \|R'_n(z)\|\leq C_\Gamma \, \mbox{~and a.s. for all large $n$,\,}\sup_{z\in \Gamma} \|R_n(z)\|\leq C_\Gamma.
 \end{equation}
\end{prop}

As in the introduction, we write $B_n = - P_n Q_n$ with $P_n,Q_n$ given by \eqref{PnQn}. Outside the spectra of $S_n$ and $C_n(\bfa)$, recall that we have set $G_n(z) =Q_n  R_{n}(z) P_n -Q_n R'_{n}(z) P_n$ in \eqref{dvptSigma0}.

\begin{prop}\label{convunif} Assume that $(H_4)$ holds. 
For any $\sigma \geq 0$ and  any compact set  $\Gamma \subset \dC \backslash \supp(\beta_\sigma)$, 
a.s. $\sup_{z\in \Gamma}\| G_n(z)\|$ converges  towards zero as $n $ goes to infinity.
\end{prop}

\begin{proof}[Proofs of Proposition \ref{nooutlier} and Proposition \ref{convunif}]

We distinguish the cases $\sigma =0$ and $\sigma >0$. We start with this last case.

\noindent{\em {Case $\sigma  >0$. }} We write 
$$
S_n = \frac{\sigma_n }{ \sigma} U_n \left( \frac{\sigma}{ \sigma_n} U^*_n C_n(\bfa) U_n + \sigma \frac{X_n}{\sqrt n}\right) U_n^* := \frac{\sigma_n }{ \sigma} U_n \tilde S_n U_n^*.
$$
Since $\| U_n A U_n^* \| = \| A \|$ for any matrix $A \in M_n(\dC)$, it is enough to prove the statement for $\tilde S_n$. We may thus assume without loss of generality that $\sigma_n = \sigma$ and $U_n = I_n$ (in the sequel only unitary invariants properties of $C_n(\bfa)$ are invoked). Next, we note that
\begin{enumerate}[-]
\item  the operator norm of $T_n({\bf a})$ is bounded by $\sup_{\omega \in \mathbb{S}^1}|{\bf a}(\omega)|$ so that assumption $(A1)$ in \cite{bordenave-capitaine16} is satisfied. \item For all $z \in \dC$, $\mu_{(T_n({\bf a}) - z I_n)(T_n({\bf a}) - z I_n)^* }$ (and $\mu_{(C_n({\bf a}) - z I_n)(C_n({\bf a}) - z I_n)^* }$) converges weakly to the probability measure $\nu_z$ defined in \eqref{nuz},  so that assumption $(A2)$ in \cite{bordenave-capitaine16} is satisfied.
\item Formula \eqref{eq:supportbeta} corresponds to assumption $(A3)$ in \cite{bordenave-capitaine16} for $\sigma >0$.
\item 
The matrices $B_n,P_n,Q_n$ are uniformly bounded in $n$ so that assumption $(A1')$ in \cite{bordenave-capitaine16} is satisfied (with $A''_n = B_n$).
\item
$B_n$ has a fixed finite rank bounded by $r+s> 0$ and for any  set $\Gamma \subset \mathbb{C}\setminus \supp (\beta_\sigma)$, 
 $C_n({\bf a}) - z I_n$ is well conditioned in $\Gamma$ in the sense of (A4) of \cite{bordenave-capitaine16}. Indeed, the eigenvalues of the normal matrix $C_n({\bf a})$ are $\bfa(e^{\iC 2\pi k/n})$, $k=0,\ldots, n-1$.
Thus, since ${\bf a}(\mathbb{S}^1)\subset \supp ( \beta_\sigma)$, for any $z$ in $\Gamma$, setting $\eta_z = \mathrm{d}(z,\bfa(\dS^1))>0$, for all $n$, $C_n({\bf a})-zI_n$ has no singular value in $[0, \eta_z]$ and $\|R'_n(z) \| \leq \eta_z ^{-1}$. 
\end{enumerate}
All the assumptions of  \cite[ Theorem~1.4]{bordenave-capitaine16} are satisfied,  the first statement of Proposition  \ref{nooutlier} follows from this theorem. The second statement follows from \cite[Proposition~3.1]{bordenave-capitaine16} (since, with the notation of this paper $\{z  \in \C :  0 \notin \supp \mu_z\}= \supp (\beta_\sigma)$). As, for Proposition \ref{convunif}, it is \cite[Proposition~4.2]{bordenave-capitaine16} applied to our setting.
%

%
\noindent{\em{ Case $\sigma = 0$. }}
Using inequality (A.7.2) in \cite{BaiSilversteinbook} and the assumption that $U_n$ is unitary, we have that
$$\max_{k} \left|s_k \left( S_n -zI_n\right)-s_k(C_n(\bfa)-zI_n)\right|
\leq \sigma_n \left\| Y_n \right\| = \sigma_n \left\|\frac{ X_n}{\sqrt{n}}\right\|,$$
where for any $n\times n$ matrix $B$, $s_1(B), \ldots, s_n(B)$ denote the singular values of $B$ in decreasing order.

Now,  for any $\varepsilon>0$, a.s.  for all large $n$, we have that $\left\|\frac{X_n}{\sqrt{n}}\right\|\leq 2+\varepsilon$ (see  \cite[Theorem 5.8]{BaiSilversteinbook}).
 Let $\delta_\Gamma = \mathrm{d}(\Gamma, \bfa(\mathbb{S}^1)) >0.$
Then, a.s.,  for all large $n$, for any $z\in \Gamma$, $s_n \left(  C_n(\bfa)-zI_n\right)\geq \delta_\Gamma >0.$
We readily deduce that, setting $\eta_\Gamma = \delta_\Gamma/2>0$, we have 
a.s. for all large $n$, 
\begin{equation}\label{valsingmin}
\inf_{z\in \Gamma}s_n \left( S_n -zI_n\right)\geq \eta_\Gamma.
\end{equation} 
Since $\| R_n(z) \| = s_n \left( S_n -zI_n\right)^{-1}$, Proposition \ref{nooutlier} follows.

From the resolvent identity, we have
\[
G_n(z)=Q_n R_{n}(z) P_n -Q_n R'_{n}(z) P_n=Q_n R_{n}(z) \sigma_n Y_n R'_{n}(z) P_n,
\]
thus, for any $z \in \Gamma$, \begin{eqnarray*} \| G_n(z)\| &\leq& \sigma_n \|Q_n\|  \|R_{n}(z)\|  \left\|\frac{X_n}{\sqrt{n}}\right\| \|R'_{n}(z)\| \| P_n\|.
\end{eqnarray*}
From Proposition \ref{nooutlier}, a.s. $\|R'_{n}(z)\|$ and $\| R_n(z)\|$ are uniformly bounded by $C_\Gamma$ for all $z \in \Gamma$. Recall also that the norms of the matrices $P_n,Q_n$ are uniformly bounded in $n$. It implies Proposition \ref{convunif} in the case $\lim_n \sigma_n = \sigma = 0$.
\end{proof}

\subsection{No outliers outside \texorpdfstring{$\spectrum T(\bfa)$}{sp(T(a))}: proof of Theorem~\ref{thm-no-outliers}}

The proof follows the ideas of \cite{bordenave-capitaine16} and is based on Rouché's theorem (see \cite[Thm~8.18]{burckel}). We start by the following key  lemma:
\begin{lem} \label{lemma-nooutlier}
For any compact $\Gamma \subset \C\setminus \spectrum T({\bf a})$,
\[
\liminf_n 
 \min_{z \in \Gamma} \left|  \frac{\det( T_n(\bfa)-z)}{\det( C_n(\bfa)-z)}  \right| > \frac{1}{2^{rs}2^{r+s+1}}.
\]  
\end{lem}
\begin{proof}
Let $z \not \in \spectrum T({\bf a})$, that is $z \not \in {\bf a}(\mathbb S^1)$ and  $\wind({\bf a}-z)=0$. Let 
\[
G({\bf a}-z)= a_s (-1)^s \prod_{k=r+1}^{s+r} \lambda_k(z), 
\] where the $\lambda_k$'s  are defined in Definition \ref{zk}.
Write:
\[
\frac{\det( T_n({\bf a})-z)}{\det( C_n({\bf a})-z)} = \frac{\det( T_n({\bf a})-z)}{G({\bf a}-z)^n} \frac{G({\bf a}-z)^n}{\det( C_n({\bf a})-z)} .
\]
The determinant of the circulant matrix is given by (see \cite[Prop~2.2]{bottcher-grudsky05}): for all $n\geq r+s+1$, 
\begin{equation} \label{det-circulant}
\det \left( C_n({\bf a})-z \right)  = a_s^n (-1)^{s(n-1)} \prod_{k=1}^{s+r} \left( 1- \lambda_k(z)^n \right).
\end{equation}
Hence, 
we have:
\begin{align*}
 \frac{G({\bf a}-z)^n}{\det( C_n({\bf a})-z)}  & = (-1)^{s} \frac{\prod_{k=r+1}^{s+r} \lambda_k(z)^n}{\prod_{k=1}^{s+r} \left( 1- \lambda_k(z)^n \right)}
 =(-1)^{s} \prod_{k=1}^{r} \frac{1}{ 1- \lambda_k(z)^n }
 \prod_{k=r+1}^{s+r} \frac{1}{\frac{1}{\lambda_k(z)^n}-1}  
\end{align*}
According to remark~\ref{rk-winding-zero}, for $k=1,\ldots,r$, $|\lambda_k(z)|<1$, hence $|1-\lambda_k(z)^n|<2$, and  for $k=r+1,\ldots,s+r$, $|\lambda_k(z)|>1$, hence $\left|  1-\frac{1}{\lambda_k(z)^n} \right|<2$.  Thus, for all $n\geq r+s+1$, we have
\[
\left|    \frac{G({\bf a}-z)^n}{\det( C_n({\bf a})-z)}   \right| > \frac{1}{2^{r+s}}. 
\]
Since $z \not \in a(\mathbb S^1)$ and $\wind({\bf a}-z)=0$, one can apply   Szeg\H{o}'s strong limit Theorem~\ref{szego-thm}, hence, as $n\to\infty$, 
\[
 \frac{\det( T_n({\bf a})-z)}{G({\bf a}-z)^n} = E({\bf a}-z) \left( 1 + O(q^n) \right)
\]
where
\[
E({\bf a}-z) = \prod_{k=1}^r \prod_{j=r+1}^{s+r} \left( 1- \frac{\lambda_k(z)}{\lambda_j(z)}  \right)^{-1}.
\]
Since $\frac{|\lambda_r(z)|}{|\lambda_{r+1}(z)|}<q<1$, and $q$ depends continuously on $z$ (see \cite[p.268]{bottcher-grudsky05}),  the convergence holds uniformly on compact sets. 

Moreover, for all $k=1,\ldots,r$,  and for all $j=r+1,\ldots,r+s$,  one has
\[
\left|  1- \frac{\lambda_k(z)}{\lambda_j(z)}  \right|  \leq 1 + \frac{|\lambda_k(z)|}{|\lambda_{j}(z)|}<2,
\]
hence,
\[
|E({\bf a}-z)| > \frac{1}{2^{rs}}. 
\]
Eventually,  one gets
that for all compact $\Gamma \subset \C \setminus \spectrum T({\bf a})$, and $n$ large enough,
\[
 \min_{z \in \Gamma} \left|  \frac{\det( T_n({\bf a})-z)}{\det( C_n({\bf a})-z)}  \right| > \frac{1}{2^{rs}2^{r+s+1}}. \qedhere
\]
\end{proof}

We are ready to prove Theorem~\ref{thm-no-outliers}.
\begin{proof}[Proof of Theorem~\ref{thm-no-outliers}]
Let $\Gamma \subset \dC \backslash (\supp( \beta_\sigma) \cup  \spectrum T(\bfa))$ be a compact set. From \eqref{proddet} and Proposition~\ref{nooutlier}, a.s. for all $n$ large enough, the eigenvalues in $\Gamma$ of $M_n$ are zeroes of the holomorphic function $f_n(z)  =  \det (I_{r+s} + Q_n R_n(z) P_n ) $. 

Let $g_n(z)  = \det (I_{r+s} + Q_n R'_n(z) P_n )$. We want to apply Rouché's theorem to the holomorphic functions $f_n$ and $g_n$, see \cite[Thm~8.18]{burckel}. On the one hand, we have, recalling that
$
G_n(z) = Q_n (R_n(z) -  R'_n(z)) P_n,
$
that
\begin{align*}
|f_n(z)-g_n(z)| & = | \det [I_{r+s} + Q_n R_n(z) P_n ]  -  \det [I_{r+s} + Q_n R'_n(z) P_n ] | \\
& =  | \det [I_{r+s} + Q_n R'_n(z) P_n + G_n(z)]  -  \det [I_{r+s} + Q_n R'_n(z) P_n ] | \\
& \leq \| G_n(z) \| (r+s) \max ( \| I_{r+s} + Q_n R'_n(z) P_n \| , \| I_{r+s} + Q_n R'_n(z) P_n + G_n(z) \| )^{r+s-1} ,
\end{align*}
using the  inequality \eqref{perturb-inequality}.
Since, the norms of $P_n $ and $Q_n$ are bounded  uniformly in $n$, and using Prop.~\ref{nooutlier},
we have, 
\[
\sup_{z \in \Gamma} \| I_{r+s} + Q_n R'_n(z) P_n \| \leq    C_\Gamma,  
\]
for some constant $C_\Gamma$ depending only on $\Gamma$.  Since, by Proposition~\ref{convunif}, a.s.,  $ \sup_{z \in \Gamma} \| G_n(z) \| \to 0$ as $n \to \infty$, one has that a.s., for all $\varepsilon>0$, there exists $n_0$ such that for all $n\geq n_0$, for all $z\in \Gamma$, 
\[
|f_n(z)-g_n(z)| < \varepsilon. 
\]
On the other hand, using again  Sylvester's identity,
\begin{equation}\label{SylvesterRprime}
g_n(z) =   \det (I_{r+s} + Q_n R'_n(z) P_n )   =   \frac{\det (z- T_n(\bfa))}{\det (z-C_n(\bfa))}  .
\end{equation}
Hence, since $\Gamma \subset \C \setminus \spectrum T(\bfa)$, we have that for $n$ large enough, 
\[
\inf_{z \in \Gamma}  \left|  \frac{\det (z- T_n(\bfa))}{\det (z-C_n(\bfa))}   \right| > \frac{1}{2^{r+s+rs+1}}
\]
 by Lemma~\ref{lemma-nooutlier}, so we get that, a.s. there exists $n_0$ such that for all $n\geq n_0$, for all $z\in \Gamma$,
\[
|f_n(z)-g_n(z)|  < |g_n(z)| . 
\]
By Rouché's theorem (we only need the above inequality on the boundary $\partial \Gamma$ of $\Gamma$), for $n \geq n_0$, $f_n$ and $g_n$ have the same number of zeroes in $\Gamma$.  Since $g_n$ has no zeroes in $\Gamma$, we deduce that, a.s., there are no eigenvalues of $M_n$ in $\Gamma$ for $n$ large enough.
\end{proof} 
 

\section{Unstable spectrum}
\label{sec:unstable}

\subsection{Convergence of the master matrix field}
\label{subsec:cvWn}

In this subsection, we give an asymptotic Gaussian equivalent of 
$$
W_n(z) = \frac{\sqrt{n}}{\sigma_n} Q_n \left( R_{n}(z)  - R'_{n}(z) \right) P_n
$$
defined in Equation \eqref{eq:defWn}.  This result involves a Gaussian process coming from the CLT proved in Section~\ref{sec:CLT} below. We show that $W_n$ is close to the  matrix field $\tilde W_n$ that we define now.

Let $\Gamma$ be a connected compact set of $\dC \backslash \supp(\beta_\sigma)$ with interior $\overset{\circ}{\Gamma}$. We consider $\tilde W_n, \tilde W^{(1)}_n$ and $\tilde W^{(2)}_n$ as the random elements of $\cH_{r+s} (\Omega)$ defined as follows. We set for $z \in \Gamma$,
$$
\tilde W^{(1)}_n (z) = Q_n R'_n (z) U_n X_n U_n^* R'_n (z) P_n.
$$
The field $\tilde W^{(2)}_n$ is Gaussian and independent of $\tilde W^{(1)}_n $. Recalling the fancy product $\circledast$ introduced for Theorem \ref{th:WnW}, the covariance kernel of $\tilde W^{(2)}_n$ is, for  $z , z' \in \Gamma$,
$$
\dE [ \tilde W^{(2)}_n (z) \circledast  \overline{\tilde W^{(2)}_n (z') } ] = \frac{\theta_n(z,z')} { 1- \sigma^2_n \theta_n(z,z')} \cA_n(z,z') \otimes \cB_n(z,z')^\intercal, 
$$
with 
$$
\theta_n(z,z') = \frac{1}{n}\Tr R'_n(z) R_n'(z')^*,
$$
$$
\cA_n(z,z') = Q_n  R'_n(z) R_n'(z')^* Q_n^* \quad \hbox{ and } \quad \cB_n(z,z') = P^*_n  R'_n(z')^* R_n'(z) P_n,
$$
and 
$$
\dE [ \tilde W^{(2)}_n (z) \circledast   \tilde W^{(2)}_n (z')  ] = \frac{\varrho^2 \theta'_n(z,z')} { 1- \varrho \sigma^2_n \theta'_n(z,z')} \cA'_n(z,z') \otimes \cB'_n(z,z')^\intercal, 
$$
with 
$$
\theta'_n(z,z') = \frac{1}{n}\Tr U_n^* R'_n(z) U_n U_n^\intercal R_n'(z')^\intercal \bar U_n,
$$
$$
\cA'_n(z,z') = Q_n  R'_n(z) U_n U_n ^\intercal R_n'(z')^\intercal  Q_n^\intercal \quad \hbox{ and } \quad \cB'_n(z,z') = P^\intercal_n  R'_n(z')^\intercal \bar U_n U_n^*  R_n'(z) P_n.
$$
The existence of the Gaussian process $\tilde W_n^{(2)}$ will be checked in the proof of Theorem \ref{th:cvWn} below. It is easy to check that $\dE [ \tilde W^{(1)}_n (z) \circledast  \overline{\tilde W^{(1)}_n (z') } ] = \cA_n(z,z') \otimes \cB_n(z,z')^\intercal$ and $\dE [ \tilde W^{(1)}_n (z) \circledast   \tilde W^{(1)}_n (z')  ] = \varrho \cA'_n(z,z') \otimes \cB'_n(z,z')^\intercal $. 
Finally, we set 
\begin{equation}\label{eq:deftildeW}
\tilde W_n = \tilde W_n^{(1)} + \sigma_n \tilde W_n^{(2)}.
\end{equation}
From the independence of $\tilde W_n^{(1)}$ and $\tilde W_n^{(2)}$, we get for  $z , z' \in \Gamma$,
\begin{equation}
\label{eq:kerneln}
\dE [ \tilde W_n (z) \circledast  \overline{\tilde W_n (z') } ] =  \frac{\cA_n(z,z') \otimes \cB_n(z,z')^\intercal} { 1- \sigma^2_n \theta_n(z,z')}  \quad \hbox{ and } \quad  \dE [ \tilde W_n (z) \circledast   \tilde W_n (z')  ] = \varrho \frac{ \cA'_n(z,z') \otimes \cB'_n(z,z')^\intercal} { 1- \varrho \sigma^2_n \theta'_n(z,z')}.
\end{equation}

 Note  that in the case $\sigma = \lim_n \sigma_n = 0$, $\tilde W_n^{(1)}$ and $\tilde W_n$ are close in  Lévy-Prokhorov distance. We say in the statement below that $\tilde W_n^{(1)}$  is {\em asymptotically Gaussian} if the Lévy-Prokhorov distance between the laws of  $\tilde W_n^{(1)}$ and of the Gaussian process in $\cH_{r+s}(\overset{\circ}{\Gamma})$ with the same covariance kernel goes to $0$. By independence if $\tilde W_n^{(1)}$ is asymptotically Gaussian then $\tilde W_n$ is also asymptotically Gaussian (with respect to the Gaussian process with covariance kernel \eqref{eq:kerneln}). Our main result on the matrix field $W_n$ is the following. 

\begin{thm}\label{th:cvWn}
  Assume  $(H_0)$ and $(H_k)$ for all $k>2$ and either $(H_\text{sym})$ or $U_n = F_n$ defined in \eqref{eq:defFn}. Let $\Gamma$ be a compact connected set in $\C \setminus \supp(\beta_\sigma)$. The laws $W_n$ and $\tilde W_n$ are tight in $\cP( \cH_{r+s}(\overset{\circ}{\Gamma}))$ and the  Lévy-Prokhorov distance between the laws of $W_n$ and $\tilde W_n$ in $\cP(\cH_{r+s}(\overset{\circ}{\Gamma}))$ goes to zero as $n$ goes to infinity.  Moreover, if $U_n$ is flat at the border, then $\tilde W_n^{(1)}$ is asymptotically Gaussian.
\end{thm}

\begin{proof}[Proof of Theorem \ref{th:cvWn}]
We start by proving the existence of $\tilde W^{(2)}_n$ as the realization of a sum of Gaussian processes. Then, we prove the theorem  for $\sigma > 0$ and $\sigma = 0$ separately.

\vspace{5pt}

{\em Step 1: construction of $\tilde W_n$. } 
Let $n \geq r+s$. We write $Q^*_n=( u_1,  \ldots , u_{r+s})$
and $P_n=(v_1,\ldots,v_{r+s})$,  the $u_i$'s and the $v_j$'s being vectors in $\C^n$. We define for any $1\leq p,q\leq r+s$ and $z \in \C \setminus \supp(\beta_\sigma)$ the sequence of matrices in $M_n(\dC)$:
\begin{equation}\label{defB}
B_{p,q,z}= \left(  U_n^* R_n'(z) v_q u_p^* R_n'(z) U_n , \sigma_n U_n^* R_n'(z) U_n, \sigma_n U_n^* R_n'(z) U_n,\cdots \right)
\end{equation}
and set
\[
{\bf B}=(B_{p,q,z}, 1\leq p,q\leq r+s, z\in \C \setminus \mbox{supp}(\beta_\sigma)).
\]

For integer $k \geq 1$, let $P_k$ be the set of $\pi=(j_1,i_1,\ldots,j_k,i_k)$ with $i_t,j_t$ in $\{1,\ldots,n\}$.
Let $\{W_\pi, \pi \in P_k, k\geq1\} $ be an infinite set of  independent and identically distributed centered Gaussian random variables defined on the same probability space with $\mathbb{E}|W_\pi|^2=1$ and $\mathbb{E}W_\pi^2=\varrho^k$ and independent of $X_n$.
We consider $W = W_k(B)$ as the centered Gaussian process, indexed by $k \geq 1$ and $B=(B_0,B_1,\ldots) \in {\bf B}$ defined by
\begin{equation}\label{defWkB}
W_k (B)=  n^{-(k-1)/2} \sum_{\pi \in P_k} W_\pi  \prod_{t=0}^{k-1}  (B_t)_{i_{t} j_{t+1}},
\end{equation}
where $i_0=i_k$
. The random variables $W_k(B), k\geq 1$ are independent  centered random analytic functions defined on the same probability space.
We have 
\[
\mathbb{E}\left(W_k(B_{p,q,z})\overline{W_l}(B_{p',q',z'})\right)= \mathbb{E}\left(W_k(B_{p,q,z}){W_l}(B_{p',q',z'})\right)=0
\] 
for $k \neq  l$, and for $k=l$,
\begin{align*}\label{varianceWtilde}
  \mathbb{E}\left(W_k(B_{p,q,z})\overline{W}_k (B_{p',q',z'})\right)  & = \left( \sigma_n^2 \theta_n(z,z')\right)^{k-1} \cA_n(z,z')_{pp'} \cB_n(z,z')_{q'q} , \\ \nonumber
  \mathbb{E}\left(W_k(B_{p,q,z}){W_k}(B_{p',p',z'})\right) &  =\rho^k \left( \sigma_n^2 \theta'_n(z,z') \right)^{k-1}  \cA'_n(z,z')_{pp'} \cB'_n(z,z')_{q'q} . \nonumber
\end{align*}

  Now, for any $z$ in $\mathbb{C}\setminus \mbox{supp}(\beta_\sigma)$,  when $n$ goes to infinity, 
 \[
f_n(z) =  \sigma^2_n \theta_n(z,z) =\sigma_n^2\frac{1}{n}\sum_{l=1}^n \frac{1}{|z-\bfa(e^{2i\pi l/n})|^2}
 \]
   converges towards $\sigma^{2}\int_{0}^1 |z- \bfa(e^{2i\pi\theta})|^{-2} d\theta <1 $.
  Since $(f_n)_n$ is a bounded sequence for the uniform convergence on compact sets of $\mathbb{C}\setminus \mbox{supp}(\beta_\sigma)$, 
by Montel's theorem,  
$(f_n)_n$ uniformly converges on any compact set in a connected component of $\C \setminus \mbox{supp}(\beta_\sigma)$.
 Thus, for any connected compact set $\Gamma$ in $\C \setminus \mbox{supp}(\beta_\sigma)$, there exists $0 < \varepsilon_\Gamma <1$ such that for large $n$,
 \begin{equation}\label{majtrR'}
 \sup_{z\in \Gamma} |\sigma^2_n \theta_n(z,z)|<\varepsilon_\Gamma.
 \end{equation}
 Then, for any $1\leq p,q\leq r+s$, 
 \begin{eqnarray}
 \sup_{z\in \Gamma}\mathbb{E} \left( \left|W_k(B_{p,q,z})\right|^2\right)&=&  \sup_{z\in \Gamma}  v_{q}^*R_n'(z)^*R_n'(z)v_q u_p^* R_n'(z)      R_n'(z)^*u_{p}\left( \sigma_n^2 \theta_n(z,z) \right)^{k-1}\nonumber\\&\leq& C\varepsilon_\Gamma^{k-1} . \label{majtension} 
 \end{eqnarray}
We deduce  from Proposition \ref{prop2.1Shirai} that  on $\Omega = \overset{\circ}{\Gamma}$, 
$$
\tilde W_n^{(2)} (z)_{pq} = \sum_{k=2}^\infty W_k(B_{p,q,z})
$$
defines a Gaussian process on $\cH(\Omega)$. This concludes the proof of the existence of the Gaussian process $\tilde W_n^{(2)} (z) = (\tilde W_n^{(2)} (z)_{p,q})$ on $\cH_{r+s}(\Omega)$. Also from Lemma~\ref{criterion}, this sequence in tight.

\vspace{5pt}

{\em Step 2:  Case $\sigma >0$. } We denote by $d_{LP}$ the Lévy-Prokhorov distance between probability measures. Let $I$ be a finite set and $(z_i)_{i \in I}$ be a collection of complex numbers. Our first goal is to prove the following
\begin{equation}\label{eq:fidicvWn}
\lim_{n \to \infty} d_{LP} ( (W_n(z_i))_{i \in I} , (\tilde W_n(z_i))_{i \in I} ) = 0,
\end{equation}
in the case $\sigma > 0$.
For any $(p,q)\in\{1,\ldots,r+s\}$, any $z\in \C \setminus  \mbox{supp}(\beta_\sigma)$,  any $k\geq 1$,  set 
\[
a_k^{(p,q)}(z)=\frac{\sqrt{n}}{\sigma_n}[Q_n \left( R'_n(z) \sigma_n Y_n \right)^k R'_n(z) P_n]_{p,q},
\]
and for any $k\geq 2$,
\[
c_k^{(p,q)}(z)=W_k(B_{p,q,z}), \quad \text{and} \quad  c_1^{(p,q)}(z)=u_p^*R'_n(z)  Y_n R'_n(z) v_q=[Q_n  R'_n(z) \sigma_n Y_n R'_n(z) P_n]_{p,q},
\]
where  $B_{p,q,z}$  is defined in \eqref{defB}.

The proof of Proposition~4.3 in \cite{bordenave-capitaine16} (with $A'_n=U^*_n C_n(\bfa) U_n$) implies that  there exists $0<\delta<1$ and $C>0$ such that the sequence of events $(\Omega_n)_n$ defined by
\begin{equation} \label{omega_event}
\Omega_n:=\left\{  \forall  N\geq n, \forall k \geq 1,  \  \sup_{z\in \Gamma}\| \left(R'_{N}(z) \sigma_N  Y_N  \right)^k \| \leq C( 1- \delta)^k   \right\}
\end{equation}
 has a probability converging to one   when $n$ goes to infinity.

On $\Omega_n$, we have the series expansion \eqref{dvptSigma} and for  all $z\in \Gamma$, for any $k\geq 1$ and any $(p,q)\in \{1,\ldots,r+s\}^2$, 
\begin{equation}\label{maja}
| a_k^{(p,q)} (z) | \leq \frac{\sqrt n}{\sigma_n} C ( 1- \delta)^k, 
\end{equation}
using Proposition~\ref{nooutlier}.
We set 
\[
k_n = \lfloor n^{c} \rfloor,  
\]
where $c$ is defined by Proposition \ref{prop:tight2}.
We consider a finite set $I$, complex numbers $\alpha_i$, complex numbers $z_i$ in $\Gamma$ and $(p_i,q_i)\in \{1,\ldots,r+s\}$, all indexed by $I$, and study the convergence of the Lévy distance between the law of  
 $\sum_{i\in I} \Re(\alpha_i (W_n(z_i))_{p_iq_i} )\1_{\Omega_n}=\sum_{k\geq 1} b_k$, with distribution function $F_n$,
 where $b_k=\sum_{i\in I} \Re(\alpha_i a_k^{(p_i,q_i)}(z_i))\1_{\Omega_n}$,    and  the law of  $\sum_{i\in I} \Re(\alpha_i \tilde W_{n}(z_i)_{p_iq_i})\1_{\Omega_n}=\sum_{k\geq 1} d_k$, with distribution function $\tilde F_n$, 
 where $d_k=\sum_{i\in I} \Re(\alpha_i c_k^{(p_i,q_i)}(z_i))\1_{\Omega_n}$. As $\Omega_n$ has a probability converging to one, we would deduce that \eqref{eq:fidicvWn} holds.
 
 Recall that if $P_1$ and $P_2$ are two laws on $\R$ with distribution functions $F_1$ and $F_2$ respectively, the L\'evy distance between $P_1$ and $P_2$ is defined by:
 $$d_L(P_1,P_2)= \inf\{\epsilon>0: F_1(x-\epsilon)-\epsilon \leq F_2(x)\leq F_1(x+\epsilon)+\epsilon \; \mbox{for all} \; x\}.$$ This distance metrizes the convergence in distribution on $\R$.

For any fixed $x\in \R$, $K\geq 2$ and $\epsilon>0$, 
 denoting by $F^{(n)}_K$ the distribution function of  $\sum_{k= 1}^K b_k$, we have that 
\begin{equation}\label{inegalite} 1-F^{(n)}_K(x+\epsilon)-\mathbb{P}\left( |\sum_{k>K}b_k|\geq\epsilon\right)\leq \mathbb{P}\left(\sum_{k\geq 1} b_k >x\right)\leq 1-F^{(n)}_K(x-\epsilon)+\mathbb{P}\left( |\sum_{k>K}b_k|\geq\epsilon\right)\end{equation}
Denote by $\tilde F^{(n)}_K$ the distribution function of  $\sum_{k= 1}^K d_k$,
we have similarly that 
\begin{equation}\label{inegalitetilde} 1-\tilde F^{(n)}_K(x+\epsilon)-\mathbb{P}\left( |\sum_{k>K}d_k|\geq\epsilon\right)\leq \mathbb{P}\left(\sum_{k\geq 1} d_k >x\right)\leq 1-\tilde F^{(n)}_K(x-\epsilon)+\mathbb{P}\left( |\sum_{k>K}d_k|\geq\epsilon\right).\end{equation}

Now,
\begin{eqnarray*}\mathbb{P}\left( |\sum_{k>K}b_k|\geq\epsilon\right)\leq 
\frac{1}{\epsilon}\mathbb{E}\left( |\sum_{k>K}b_k|\right)\end{eqnarray*}
with $|\sum_{k>K}b_k|\leq \sum_{i\in I} |\alpha_i|\sum_{k>K}|a_k^{(p_i,q_i)}(z_i)|\1_{\Omega_n}$.
Using Proposition \ref{prop:tight2} and \eqref{majtrR'},  for large $n$,  if $k_n>K$, 
\begin{eqnarray}\sum_{K<k \leq  k_n}\mathbb{E}\left( | a_k^{(p_i,q_i)} (z_i)|\1_{\Omega_n}\right)&\leq & 
\sum_{K<k \leq  k_n}\left\{C k^{12} \left\{ \sigma_n^2 \theta_n(z_i,z_i) \right\}^{k-1}\right\}^{1/2}\nonumber \\
&\leq&
\sum_{K<k \leq  k_n}\left\{C k^{12}  \varepsilon_\Gamma^{k-1}\right\}^{1/2}\nonumber\\
&\leq&
\sum_{k >K} C k^{6}  \varepsilon_\Gamma^{(k-1)/2}=:\zeta_K \rightarrow_{K\rightarrow+\infty}0.\label{sommeinfkn}\end{eqnarray}
Now, according to \eqref{maja}, \begin{equation}\label{majksupkn}\sum_{k> k_n}\mathbb{E}\left( | a_k^{(p_i,q_i)} (z_i)|\1_{\Omega_n}\right)\leq C \sqrt n ( 1 - \delta_0)^{k_n}
 \end{equation}
Thus,  \begin{equation}\label{sumb}\mathbb{P}\left(|\sum_{k>K}b_k|>\epsilon \right)\leq 
\frac{1}{\epsilon}\left(\sum_{i\in I} |\alpha_i|\right)\left(\zeta_K+C \sqrt n ( 1 - \delta_0)^{k_n}\right).\end{equation}
Now,
\begin{eqnarray*}\mathbb{P}\left( |\sum_{k>K}d_k|\geq\epsilon\right)\leq 
\frac{1}{\epsilon}\mathbb{E}\left( |\sum_{k>K}d_k|\right)\end{eqnarray*}
with $|\sum_{k>K}d_k|\leq \sum_{i\in I} |\alpha_i|\sum_{k>K}|c_k^{(p_i,q_i)}(z_i)|$.
Using 
\eqref{majtension},
 we obtain that   for large $n$, for any $i \in I$,  $$\sum_{k>K}\mathbb{E}\left( | c_k^{(p_i,q_i)} (z_i)|\right)\leq  
C \sum_{k >K}  \varepsilon_\Gamma^{(k-1)/2}=:\psi_K\rightarrow_{K\rightarrow+\infty}0.$$
Thus for large $n$, for any $K$,  \begin{equation}\label{sumd}\mathbb{P}\left( |\sum_{k>K}d_k|\geq\epsilon\right)\leq \frac{1}{\epsilon}\sum_{i\in I} |\alpha_i|\psi_K.\end{equation}
Therefore, using \eqref{sumb} and \eqref{sumd}, for any $K$ and $n$ large enough (independently), we have 
$$\mathbb{P}\left( |\sum_{k>K}b_k|\geq\epsilon\right)\leq \epsilon, \; \mbox{and}\; \mathbb{P}\left( |\sum_{k>K}d_k|\geq\epsilon\right)\leq \epsilon.$$
Then, we can deduce from \eqref{inegalite} and \eqref{inegalitetilde} that 
for any $\epsilon>0$, there exists $n_0(\epsilon)$ and $K_0(\epsilon)$ such that $\forall x \in \R$, $\forall n\geq n_0(\epsilon)$,  $\forall K\geq K_0(\epsilon)$,
$$ F^{(n)}_K(x-\epsilon)-\epsilon\leq  F_n(x)\leq  F^{(n)}_K(x+\epsilon)+\epsilon,$$
$$\tilde F^{(n)}_K(x-\epsilon)-\epsilon\leq \tilde F_n(x)\leq \tilde F^{(n)}_K(x+\epsilon)+\epsilon.$$
Therefore $\forall n\geq n_0(\epsilon)$,  $\forall K\geq K_0(\epsilon)$, 
\[
d_L(\sum_{k=1}^K b_k, \sum_{k\geq 1}b_k) \leq \epsilon \quad  \text{and} \quad d_L(\sum_{k=1}^K d_k, \sum_{k\geq 1}d_k) \leq \epsilon,
\]
and then \begin{eqnarray*} d_L(\sum_{k\geq 1}b_k, \sum_{k\geq 1}d_k)&\leq& d_L(\sum_{k\geq 1}b_k, \sum_{k=1}^K b_k)+d_L(\sum_{k = 1}^K b_k, \sum_{k= 1}^K d_k)+d_L(\sum_{k=1}^Kd_k, \sum_{k\geq 1}d_k)\\&\leq& 2\epsilon +  d_L(\sum_{k = 1}^K b_k, \sum_{k= 1}^K d_k)\end{eqnarray*}
Now, fixing $K\geq K_0(\epsilon)$, we deduce from Theorem \ref{th:steinTCL} that \eqref{eq:fidicvWn} holds (using that $d_L\leq d_{LP}$).

Moreover, following the lines of \eqref{sommeinfkn}, we obtain that 
for large $n$, for any $z\in \Gamma$,  
\begin{equation}\label{majinfkn}\sum_{k \leq  k_n}\mathbb{E}\left( | a_k^{(i,j)} (z)| \1_{\Omega_n}\right)\leq
\sum_{k \geq 0} C k^{6}  \varepsilon_\Gamma^{(k-1)/2}<+\infty
\end{equation}
since $\varepsilon_\Gamma <1$.
It readily follows from \eqref{dvptSigma}, \eqref{majksupkn} and \eqref{majinfkn} that 
 there exists some $C>0$ such that, for any $1\leq p,q\leq r+s$, 
\begin{equation}\label{tightSigma}
\sup_{z \in \Gamma}\mathbb{E} \left[ \Big| W_n(z)_{pq} \Big| \1_{\Omega_n} \right] \leq C.
\end{equation}
From Lemma~\ref{criterion}, this implies that $W_n$ is a tight sequence in $\cH_{r+s} (\overset{\circ}{\Gamma})$. It concludes the proof of Theorem \ref{th:cvWn} in the case $\sigma > 0$.

\vspace{5pt}

{\em Step 3:  Case $\sigma = 0$. }  We define 
\begin{equation*}\label{omegatilde}
\tilde \Omega_n=\left\{\frac{\|X_n\|}{\sqrt{n}}\leq 3\right\}.
\end{equation*} 
 According to Theorem 5.8 in \cite{BaiSilversteinbook},   this event has a probability converging to one   when $n$ goes to infinity. If $C_\Gamma$ is the constant defined in Proposition \ref{nooutlier}, we consider the event
 $$
 \Omega_n = \tilde \Omega_n \cap \left\{ \sup_{z \in \Gamma} \| R_n(z) \| \leq C_\Gamma \right\}. 
 $$
This event has a probability converging to one.  Our goal is to prove that 
\begin{equation}\label{eq:kokgt}
\lim_{n \to \infty} \sup_{z \in \Gamma} \dE [ \| W_n(z)- \tilde W_n^{(1)}(z) \| \1_{\Omega_n} ] =  0.
\end{equation}
From  Lemma~\ref{criterion}, this implies that the  Lévy-Prokhorov distance between the laws of $W_n$ and $\tilde W_n$ goes to $0$ as $n$ goes to infinity.

From the resolvent identity, for any integer $k_n \geq 2$ and $z \in \Gamma$,  \begin{equation*}\label{eq:devtaylor}
Q_n R_n(z) P_n  -  Q_n R'_n(z) P_n  = \sum_{k=1}^{k_n-1} Q_n  \left(R'_n(z) \sigma_n Y_n \right)^k R'_n(z) P_n + Q_n  \left(R'_n(z) \sigma_n Y_n \right)^{k_n} R_n(z) P_n  .
\end{equation*}
Therefore,  for any $z \in \Gamma$, we have 
\begin{equation}\label{dvpsigmagamma}
W_n(z)- \tilde W_n^{(1)}(z)=  \sum_{k = 2}^{k_n} w_k(z) +  \sqrt n  \sigma_n^{k_n-1}  Q_n  \left(R'_n(z)   Y_n \right)^{k_n} R_n(z) P_n ,
\end{equation} 
where
$$
w_k (z) = \sqrt n \sigma_n^{k-1} Q_n  \left(R'_n(z)  Y_n \right)^k R'_n(z) P_n
$$
We set 
\[
k_n = \lceil \log n \rceil +1.
\] 
We first check that the last term on the right-hand side of \eqref{dvpsigmagamma} goes to $0$ on the event $\Omega_n$. Indeed, if $\Omega_n$ holds, we have, for some constant $C = C(\Gamma) >0$, 
$$
 \|  \sqrt n  \sigma_n^{k_n-1}  Q_n  \left(R'_n(z)   Y_n \right)^{k_n} R_n(z) P_n  \|  \leq C  \sqrt n  (C \sigma_n) ^{k_n -1}  \leq C \sqrt n (C \sigma_n)^{\log n}.
$$
This goes to $0$ as $n$ goes to infinity uniformly in $z \in \Gamma$.

We next check that the first term on the right-hand side of \eqref{dvpsigmagamma} goes to $0$. Using Proposition \ref{prop:tight2},  for large $n$, for any $z\in \Gamma$ and $1 \leq p , q \leq r +s$, for some constant $C = C(\Gamma)$ changing from line to line,
\begin{eqnarray*}
\sum_{k = 2}^{k_n}\mathbb{E}\left( | w_k (z)_{pq}|\right) &\leq &  \sum_{k = 2}^{k_n} \sigma_n^{k-1}\left(C k^{12}  \theta_n(z,z)^{k-1} \right)^{1/2}\\
&\leq & \sum_{k = 2}^{k_n} C k^{6} \left( C \sigma_n\right)^{k-1}\nonumber\\
&\leq & \sum_{k =1}^\infty \left(C \sigma_n\right)^{k},
\end{eqnarray*}
where we have used that for all $k \geq 2$, $k^6 \leq C^k$ for some $C >0$. For $n$ large enough, $\sigma_n C < 1$, we get 
\begin{eqnarray}\label{pourtensioninfini}
\sum_{k = 2}^{k_n}\mathbb{E}\left( | w_k (z)_{pq}|\right) &\leq & \frac{ C \sigma_n} {1 - C \sigma_n},
\end{eqnarray}
which goes to $0$ as $n$ goes to infinity. It concludes the proof of \eqref{eq:kokgt}.

\vspace{5pt}

{\em Step 4 :  $U_n$ flat at the border. } It remains to check that if $U_n$ is flat at the corner then $\tilde W_n^{(1)}(z)$ is close to Gaussian for any $z \in \Gamma$. We fix $z \in \Gamma$ and for $1 \leq p \leq r +s$, we introduce the vectors in $\dC^n$,
$$
x_p =  U_n^* R_n'(z)^* u_p \quad \hbox{ and } \quad y_p = U_n^* R_n'(z) v_p .
$$
From Theorem \ref{th:steinTCL}, it suffices to check that for all $1 \leq p \leq r +s$, $\|x_p\|_{\infty} \to 0$ and $\|y_p \|_{\infty} \to 0$ as $n \to \infty$. For $1 \leq i \leq n$, we write
$$
(x_p)_i = \sum_{j,k} \bar U_{ji} \bar R'_n(z) _{kj} (u_p)_k.
$$
We have $(u_p)_k = 0$ unless $k \in C_n = \{1,\ldots,s\} \cup \{ n-r+1,\ldots,n\}$ and $(u_p)_k$ is uniformly bounded. Therefore from Lemma \ref{le:bdres} below, for some $0 < \kappa < 1$ and $C > 0$,
$$
\| x_p \|_{\infty} \leq C  \sum_{k \in C_n} \sum_{j = 1}^n \kappa^{d_n(k,j)} \max_{1 \leq i \leq n} |U_{ji}|,
$$
where $d_n(i,j) = \min_{k \in \mathbb Z} |i - j - k n |$ is the distance between $i$ and $j$ modulo $n$. For integer $h \geq 1$, let $\veps_h = \max_{j : d_n(j,C_n) \leq h }  \max_i |U_{ji}|$. Since $|U_{ij}| \leq 1$ and since, for any $l \geq 0$, there are at most two $j$ such that $d_n(k,j) = l$, we find 
$$
\| x_p \|_{\infty} \leq   C \sum_{k \in C_n}  \sum_{l > h} 2 \kappa^l + C \sum_{k \in C_n} \sum_{l = 0}^h 2 \kappa^l \veps_h \leq \frac{ 2 C (r+s) }{1 - \kappa} \left ( \kappa^h + \veps_h\right).
$$
The assumption that $U_n$ is flat at the border implies that $\veps_{h_n} \to 0$ for some sequence $h_n \to \infty$. We deduce that $\| x_p \|_{\infty} \to 0$. The proof of $\| y_p \|_{\infty} \to 0$ is identical.
\end{proof}

We  have used the following probably well-known lemma. 

\begin{lem}\label{le:bdres}
Let $z \in \dC \backslash \bfa (\dS^1)$. There exist positive reals $C > 0$ and  $0 < \kappa  < 1$ such that for all integers $n \geq 1$, $1 \leq i,j \leq n$
$$
| R'_n(z)_{ij} | \leq C  \kappa^{ d_n(i,j)},
$$ 
where $d_n(i,j) = \min_{k \in \mathbb Z} |i - j - k n |$ is the distance between $i$ and $j$ modulo $n$.
\end{lem}
\begin{proof} We may reproduce a classical argument in \cite{EIJKHOUT1988247,MR1708693} for symmetric band matrices. 
It is standard that there exist $C > 0$ and $0 < \kappa < 1$ such that for all integer $l \geq 1$, there exists a polynomial $P_l$ (depending on $\bfa , z$) of degree $l$ such that for all $\omega \in \dS^1$,
\begin{equation}\label{eq:berns}
| (z-\bfa(\omega)) ^{-1} - P_l(\omega) | \leq C \kappa^l ,
\end{equation}
(see below). Then we deduce from \eqref{eq:berns} and the spectral theorem that 
$$
| R'_n(z)_{ij}  - P_l(C_n(\bfa))_{ij} | \leq  C \kappa^l,
$$
where $C_n(\bfa)$ is the circulant matrix. Observe that $P_l(C_n(\bfa))_{i,j} = 0$  if $d_n(i,j) > (r + s)l$. We take $l = \lceil d_n(i,j)/(r+s) \rceil +1$. Adjusting the values of $\kappa$ to $\kappa^{1/(r+s)}$ and $C$, we deduce the lemma.

It remains to check that \eqref{eq:berns} holds. Up to replace $\bfa$ by $\bfa - z$, without loss of generality, assume $z = 0$.  For $\omega \in \dS^1$, we write 
$$
 \bfa(\omega) ^{-1} = \frac{\omega^{r}}{ \sum_{\ell= -r}^s a_{\ell} \omega^{\ell + r} } = \frac{ \omega^{r}}{ Q (\omega)} .
$$
By assumption $Q (\omega) = \prod_i ( \omega - \lambda_i)$ with $\lambda_i \ne 0$. We may decompose $1/Q$ into simple elements. The proof is then contained in the proof of Proposition \ref{prop:convratfunc} in Appendix.
\end{proof}

\subsection{Limit of \texorpdfstring{$Q_n R'_n(z) P_n$}{QnR'n(z)Pn}} 
\label{subsec:QrP}
We now study the limit of $Q_n R'_n(z) P_n$ when $n$ goes to infinity for $z\in \C\setminus \bfa(\mathbb S^1)$. Recall the definitions of $\cM(\omega)$, $\cD$ and $\cE$ in \eqref{eq:defcM} and the definition of $\cH(z)$ in \eqref{defH}.

\begin{prop}\label{convH}  Let $\Gamma$ be a compact subset of $\C \setminus \bfa (\mathbb{S}^1)$. There exist $C>0$ and $0 < \kappa < 1$ (depending on $\bfa$ and $\Gamma$) such that for any $z \in \Gamma$, for all $n \geq 1$,
\[
\| Q_n R'_n(z) P_n  -   \cH(z) \| \leq C   \kappa ^n ,
\]
\end{prop}
\begin{proof}
Recall that the circulant matrix $C_n(\bfa)$ is diagonalized as
\[
C_n(\bfa)= F_n \diag(\bfa(\omega_n^{i-1})), i=1,\ldots,n) F_n^*,
\]
where $\omega_n=e^{\frac{i2\pi}{n}}$ and $F_n$ is the eigenvector matrix given by $F_{i,j} = \frac{1}{\sqrt n} \omega_n^{(i-1)(j-1)}$, for $1\leq i,j \leq n$.  We decompose the matrix $F_n$ into block matrices as follows:
\begin{equation}\label{F}
F_n=
\begin{pmatrix}
B_1 \\
B_2 \\
B_3 \\
\end{pmatrix} 
=
\begin{pmatrix}
\tilde B_1 \\
\tilde B_2 \\
\tilde B_3 \\
\end{pmatrix} ,
\end{equation}
where $B_1$ and $B_3$ are respectively the $s\times n$ matrix given by the $s$ first rows of $F_n$ and the    $r\times n$ matrix given by the last $r$ rows of $F_n$, that is
\[
B_1 = \frac{1}{\sqrt n}  \left(  \omega_n^{(i-1)(j-1)}  \right)_{\substack{1 \leq i \leq s \\ 1 \leq j \leq n}} ; 
\quad B_3 =  \frac{1}{\sqrt n}  \left(  \omega_n^{(n-r+i-1)(j-1)}  \right)_{\substack{1 \leq i \leq r \\ 1 \leq j \leq n}}  ,
\]
and similarly $\tilde B_1$ is the $r$ first rows of $F_n$ and $\tilde B_3$ the $s$ last rows:
\[
\tilde B_1 = \frac{1}{\sqrt n}  \left(  \omega_n^{(i-1)(j-1)}  \right)_{\substack{1 \leq i \leq r \\ 1 \leq j \leq n}} ; 
\quad \tilde B_3 =  \frac{1}{\sqrt n}  \left(  \omega_n^{(n-s+i-1)(j-1)}  \right)_{\substack{1 \leq i \leq s \\ 1 \leq j \leq n}}  ,
\]
($B_2$ and $\tilde B_2$ will be of no use).

Setting
\[
\Delta_n(z)=\diag((z-\bfa(\omega_n^{i-1}))^{-1}), i=1,\ldots,n),
\]
it follows from \eqref{eq:diagCn} that
\begin{align*}
Q_n R'_n(z) P_n 
& = Q_n F_n \Delta_n(z) F_n^* P_n \\
& = 
\begin{pmatrix}
I_s & 0 \\
0 & D_r
\end{pmatrix}
\begin{pmatrix}
B_1 \Delta_n(z) \tilde B_3^* & B_1 \Delta_n(z) \tilde B_1^* \\
B_3 \Delta_n(z) \tilde B_3^* & B_3 \Delta_n(z) \tilde B_1^*
\end{pmatrix}
\begin{pmatrix}
 E_s & 0 \\
0 & I_r 
\end{pmatrix}. 
\end{align*}
Hence, computing the limit of $Q_n R'_n(z) P_n$ consists of computing the limit of each $B_i \Delta_n(z) \tilde B_j^*$, for $i,j=1,3$. One has, for $1 \leq p,q \leq s$, 
\begin{align*}
(B_1 \Delta_n(z) \tilde B_3^* ) _{p,q}
& =\frac1n \sum_{k=1}^n \omega_n^{(p-1)(k-1)} \frac{1}{z-\bfa(\omega_n^{k-1})} \overline{\omega_n^{(n-s+q-1)(k-1)}}\\
& = \frac1n \sum_{k=1}^n e^{i 2 \pi (p-q+s)(k-1)/n}  \frac{1}{z-\bfa(e^{i 2 \pi (k-1)/n})} \\
& \underset{n \to \infty}{\longrightarrow} \frac{1}{2 \pi} \int_0^{2 \pi} \frac{e^{i (p-q+s) \theta}}{z-\bfa(e^{i\theta})} d\theta,
\end{align*}
where one recognizes the limit of a  Riemann sum. Proposition \ref{prop:convratfunc} in Appendix gives an exponential rate of convergence for such Riemann sum. Similarly, one has
\begin{align*}
(B_1 \Delta_n(z) \tilde B_1^* ) _{p,q}  &  \underset{n \to \infty}{\longrightarrow}  \frac{1}{2 \pi} \int_0^{2 \pi} \frac{e^{i (p-q) \theta}}{z-\bfa(e^{i\theta})} d\theta , \quad \text{for $1\leq p\leq s, \ 1 \leq q \leq r,$}\\
(B_3 \Delta_n(z) \tilde B_3^* ) _{p,q}  &  \underset{n \to \infty}{\longrightarrow}  \frac{1}{2 \pi} \int_0^{2 \pi} \frac{e^{i (p-q +s-r) \theta}}{z-\bfa(e^{i\theta})} d\theta , \quad \text{for $1\leq p\leq r, \ 1 \leq q \leq s,$} \\
(B_3 \Delta_n(z) \tilde B_1^* ) _{p,q}  &  \underset{n \to \infty}{\longrightarrow}  \frac{1}{2 \pi} \int_0^{2 \pi} \frac{e^{i (p-q-r) \theta}}{z-\bfa(e^{i\theta})} d\theta , \quad \text{for $1\leq p\leq r, \ 1 \leq q \leq r$} .
\end{align*}
The result follows. 
\end{proof}

\subsection{Rank of \texorpdfstring{$I  + \cH(z)$}{I+H(z)}} 
\label{subsec:rk}

Our next proposition connects the rank of $I  + \cH(z)$ to the winding number of $\bfa - z$. 
\begin{prop} \label{prop:rank}
If $z \in  \C \setminus \bfa (\mathbb{S}^1)$ and $\delta (z) = \wind(\bfa - z)$ then
$\dim(\ker( I  + \cH(z) )) = |\delta(z)|.
$\end{prop}

This subsection is devoted to the proof of Proposition \ref{prop:rank}. Up to replace $\bfa $ by $\bfa -z$, we may assume without loss of generality that $z =0$. For ease of notation, we remove the dependency in $z = 0$ of all considered quantities (we write $\cH$, $\delta$ in place of $\cH(z)$, $\delta(z)$ for example). We set 
$$
Q (\omega)  = \omega^r \bfa(\omega) = \sum_{\ell = -r}^{s} a_{\ell}  \omega^{\ell+r}.
$$
We denote by $\lambda_1 , \ldots, \lambda_{r+s}$ the roots of $Q$ (with multiplicities) and  order them by non-decreasing modulus, $|\lambda_1| \leq |\lambda_2| \leq \cdots \leq |\lambda_{r+s}|$. The number of roots in $\mathbb{D}$ is denoted by $J$: $|\lambda_{J} | < 1 < |\lambda_{J+1}|$. By definition, we have 
$$
\delta = J - r.
$$

We start the proof by a two technical lemmas when all roots of $Q$ are simple  (i.e. all $\lambda_i$'s are distinct). We introduce the following vectors and matrices. For $j \geq 1$, we set 
$$
f_j = \begin{pmatrix}
\phi_j \\ D_r \psi_j
\end{pmatrix} \in \dC^{r+s},
$$
where $(\phi_j)_p = \lambda_j^{r+p-1}$, $1 \leq p \leq s$, $(\psi_j)_p = \lambda_j^{p-1}$, $1 \leq p \leq r$. We also define the vector in $\dC^{r+s}$, $(g_j)_p = \lambda_j^{s-p}$. Finally for $r \geq 1$, the matrix $\cT \in M_{r+s}(\dC)$ is the upper-triangular matrix 
$$
\cT = \begin{pmatrix} 0 & 0 \\
0 & D_r T
\end{pmatrix}
$$
with $T_{pq} = 0$ for $1 \leq q < p \leq r$ and $T_{pq} = 1/ (q-p)! \cdot (1/Q)^{(q-p)} (0)$ for $1 \leq p \leq q \leq r$ where $(1/Q)^{(0)} = 1/Q$ and for $i \geq 1$, $(1/Q)^{(i)}$ is the $i$-th derivative of $1/Q$. Note that since $a_{-r} \ne 0$, $0$ is not a pole of $1/Q$. If $r = 0$, then we simply set $\cT = 0$.

\begin{lem}\label{le:resH}
If all roots of $Q$ are simple then 
$$
\cH = - \sum_{j = 1}^{J}\frac{1}{Q'(\lambda_j)} f_j  g_j^\intercal \mathcal E- \cT.
$$
\end{lem}
\begin{proof}
We have $$
\frac{1}{\omega \bfa(\omega)} 
\begin{pmatrix}
\omega^sI_s & 0 \\
0 & \omega^{-r}I_r
\end{pmatrix}
\cM(\omega) =  \frac{1}{Q(\omega)} {
\begin{pmatrix}
\omega^{s+r-1}& \cdots & 1 \\
\omega^{s+r} & \cdots & \omega\\
\vdots &   &  \vdots \\
\omega^{2 s + r -2} & \cdots & \omega^{s-1} \\
\omega^{s-1} & \ldots & \omega^{-r} \\
\vdots & & \vdots  \\
\omega^{s+r-2} & \ldots & \omega^{-1}
\end{pmatrix}}.
$$
Observe that the only coefficients of the above matrix with elements of the form $\omega^{p}$ with $p < 0$ are in the upper triangular part of the bottom left $r \times r$ corner. 
On the other end, since all roots of $Q$ are simple, from Cauchy's Residue Theorem, we have for integer $p \geq 0$,
$$ \frac{1}{2 i \pi}  \oint_{\mathbb{S}^1}   \frac{\omega^p }{Q(\omega)} d\omega = \sum_{j=1}^{J} \frac{\lambda_j^{p}}{Q'(\lambda_j)} 
$$
while for $p < 0$, 
$$ \frac{1}{2 i \pi}  \oint_{\mathbb{S}^1}   \frac{\omega^p }{Q(\omega)} d\omega = \sum_{j=1}^{J} \frac{\lambda_j^{p}}{Q'(\lambda_j)} + \frac{1}{(-p-1)!} \left(\frac{1}{Q}\right)^{(-p-1)}(0).
$$
It follows that 
$$
\frac{1}{2 i \pi}  \oint_{\mathbb{S}^1}   \frac{1}{\omega \bfa(\omega)} 
\begin{pmatrix}
\omega^sI_s & 0 \\
0 & \omega^{-r}I_r
\end{pmatrix}
\cM(\omega) = \sum_{j=1}^J \frac{\lambda_j^{r-1}}{Q'(\lambda_j)} \begin{pmatrix}
\lambda_j^s I_s & 0 \\
0 & \lambda_j^{-r}I_r
\end{pmatrix}
\cM(\lambda_j) + \begin{pmatrix} 0 & 0 \\ 0 & T 
\end{pmatrix}.
$$
Consequently, 
$$
\cH = - \sum_{j=1}^J \frac{\lambda_j^{r-1}}{Q'(\lambda_j)} \cD \begin{pmatrix}
\lambda_j^s I_s & 0 \\
0 & \lambda_j^{-r}I_r
\end{pmatrix}
\cM(\lambda_j) \mathcal E - \mathcal T.
$$

We next use that $\cM(\omega)$ is a rank one matrix: $\cM(\omega)= a(\omega) b^\intercal (\omega)$ with $a(\omega)_p = \omega^p$ and $b(\omega)_p = \omega^{-p}$. It is then easy to check that  
$$
\lambda_j^{r-1}  \cD  \begin{pmatrix}
\lambda_j^s I_s & 0 \\
0 & \lambda_j^{-r}I_r
\end{pmatrix}
\cM(\lambda_j) = f_j g_j^{\intercal}.
$$
The conclusion follows.
\end{proof}

We next define for $1 \leq j \leq r+s$,
$$
\tilde f_j = \begin{pmatrix}
\phi_j \\ - D_r \psi_j
\end{pmatrix} \in \dC^{r+s},
$$
Since all $\lambda_j$'s are distinct, note that the vectors $\begin{pmatrix}
 \phi_j  \\ \psi_j 
\end{pmatrix}$, $1 \leq j \leq r+s$, classically form a basis of $\dC^{r+s}$. Hence since $\begin{pmatrix} I_s & 0 \\ 0 & \pm D_r \end{pmatrix}$ is invertible, $(f_j)_j$ and $(\tilde f_j)_j$ are two other basis of $\dC^{r+s}$. 

The next lemma asserts that the family of vectors $(g_j)$ and $\cE \tilde f_j$ are bi-orthogonal. 

\begin{lem}\label{le:biorth}
If all roots of $Q$ are simple then for all $1 \leq i , j \leq r+s$,
$$
g_j^{\intercal} \cE \tilde f_i = \delta_{ij} Q'(\lambda_j).
$$ 
\end{lem}
\begin{proof}
For $1 \leq  q \leq r$, we have $$(D_r \psi_i)_q = \sum_{p \geq q}  a_{-r+p -q} (\lambda_i)^{p-1} = \lambda_i^r \sum_{\ell = -r}^{-q} a_{\ell} \lambda_i ^{\ell + q -1}.$$ 
Similarly, for $1 \leq q \leq s$, 
$$
(E_s\phi_i)_q = \lambda_i ^r \sum_{\ell = s - q +1}^s a_\ell \lambda_i ^{\ell - s + q  -1}.
$$
We get
\begin{eqnarray*}
g_j^{\intercal} \cE \tilde f_i & = & \sum_{q=1}^s \lambda_j^{-q + s} ( E_s\phi_i)_q - \sum_{q=1}^r \lambda_j^{-q} (D_r\psi_i)_q \\
 & = & \lambda_i^r \left(  \sum_{q=1}^s \lambda_j^{s-q} \sum_{\ell = s - q +1}^s a_\ell \lambda_i ^{\ell - s + q  -1} - \sum_{q=1}^r \lambda_j^{-q} \sum_{\ell = -r}^{-q} a_{\ell} \lambda_i ^{\ell + q  -1} \right) \\
 & = &\lambda_i^r  \left(\sum_{\ell =1}^s a_{\ell} \lambda_i^{\ell-1} \sum_{q = s - \ell +1}^s \left( \frac{\lambda_j}{\lambda_i} \right)^{s-q} - \sum_{\ell = -r}^{-1} a_\ell \lambda_i^{\ell -1} \sum_{q=1}^{-\ell} \left( \frac{\lambda_i}{\lambda_j} \right)^{q}  \right)\\
& = &  \lambda_i^r  \left(\sum_{\ell =1}^s a_{\ell} \lambda_i^{\ell-1} \sum_{k =0}^{\ell -1} \left( \frac{\lambda_j}{\lambda_i} \right)^{k} - \sum_{\ell = -r}^{-1} a_\ell \lambda_i^{\ell -1}  \sum_{k=1}^{-\ell} \left( \frac{\lambda_i}{\lambda_j} \right)^{k}  \right).
\end{eqnarray*}
If $i \ne j$, we thus find
\begin{eqnarray*}
g_j^{\intercal} \cE \tilde f_i & = &  \lambda_i^r  \left(\sum_{\ell =1}^s a_{\ell} \lambda_i^{\ell-1}  \frac{1 - \left( \frac{\lambda_j}{\lambda_i} \right)^{\ell} }{1 - \frac{\lambda_j}{\lambda_i} }  - \sum_{\ell = -r}^{-1} a_\ell \lambda_i^{\ell -1}   \frac{\frac{\lambda_i}{\lambda_j}  - \left( \frac{\lambda_i}{\lambda_j} \right)^{-\ell+1} }{1 - \frac{\lambda_i}{\lambda_j} } \right) \\
 & = & \frac{ \lambda_i^r }{\lambda_i - \lambda_j} \left( \sum_{\ell =1}^s a_\ell  (\lambda_i^\ell - \lambda_j^{\ell}) +  \sum_{\ell =-r}^{-1} a_\ell  (\lambda_i^\ell - \lambda_j^{\ell})\right)\\
& = &  \lambda_i^r \frac{\bfa(\lambda_i) - \bfa(\lambda_j)}{\lambda_i - \lambda_j}.
\end{eqnarray*}
Since $Q(\lambda_i) = Q (\lambda_j) = 0$, the claim follows. Similarly, for $i = j$, we get from what precedes
\begin{eqnarray*}
g_i^{\intercal} \cE \tilde f_i & = &  \lambda_i^r  \left(\sum_{\ell =1}^s \ell a_{\ell} \lambda_i^{\ell-1}    - \sum_{\ell = -r}^{-1} (-\ell) a_\ell \lambda_i^{\ell -1}   \right) \\
 & = & \sum_{\ell=-r}^{s} \ell  a_{\ell}  \lambda_i^{\ell + r -1} \\
& = &  Q'(\lambda_i) ,
\end{eqnarray*}
where at the third line, we have used that $Q(\lambda_i) = \sum_{\ell} a_{\ell} \lambda_i^{\ell + r} = 0$ so that $Q'(\lambda_i) = \sum_{\ell}  (\ell + r) a_{\ell}  \lambda_i ^{\ell + r -1} =  \sum_{\ell} \ell a_{\ell}  \lambda_i ^{\ell + r -1}$. It concludes the proof.
\end{proof}

We are ready to prove Proposition \ref{prop:rank}.

\begin{proof}[Proof of Proposition \ref{prop:rank}]

We shall first prove Proposition \ref{prop:rank} when $\delta \geq 0$ and all roots simple. Then extend  to the case of roots with arbitrary multiplicities and finally to $\delta < 0$. 

\vspace{5pt}

\noindent {\em Case 1: simple roots, non-negative winding number. } We assume here that all roots of $Q$ are simple and $J = \delta +r \geq 0$.

Let $(e_p)_{1 \leq p \leq r+s}$ be the canonical basis of $\dC^{r+s}$, $\Pi_+$ be the orthogonal projection onto the vector space $(e_1\ldots,e_s)$ and $\Pi_- = I - \Pi_+$ the orthogonal projection onto  the vector space $(e_{r+1}\ldots,e_{r+s})$ (if $r >0$). 

For $x \in \dC^{r+s}$, we write 
$$
x = \sum_{j=1}^{r+s} x_j \tilde f_j. 
$$
If $x \in \ker (I + \cH)$, we get from Lemma \ref{le:resH} and Lemma \ref{le:biorth},  that
\begin{equation}\label{eq:kerIH}
0 = (I + \cH) x = \sum_{j=1}^{J} x_j (\tilde f_j - f_j) + \sum_{j = J+1}^{r+s} x_j  \tilde f_j - \sum_{j=1}^{r+s} x_j \cT \tilde f_j. 
\end{equation}
We first compose \eqref{eq:kerIH}  by $\Pi_+$. Since the image of $\Pi_+ \cT = 0$ and $\Pi_+ (\tilde f_j - f_j) = 0$, we find in $\dC^s$ that,
$$
0 = \sum_{j = J+1}^{r+s} x_j \Pi_+ \tilde f_j = \sum_{j = J+1}^{r+s} x_j \phi_j.
$$ 
Remark that $(\phi_1,\ldots,\phi_{s})$ forms a basis of $\dC^{s}$ and that, since $J \geq r$, we have $r+s -J \leq s$. It follows that if $x \in \ker (I + \cH)$ then $x_j  = 0$ for all $j > J$.

If $r >0$, we next compose \eqref{eq:kerIH} by $\Pi_-$. Since $\Pi_- (\tilde f_j - f_j) = - 2 D_r \psi_j$, we find in $\dC^r$,
$$
0 = - 2 \sum_{j=1}^{J} x_j D_r \psi_j + \sum_{j=1}^{J} x_j D_r T D_r \psi_j.
$$
We use that $D_r$ is invertible ($D_r$ is upper triangular and $D_{pp} = a_{-r} \ne 0$), we deduce that the above equation is equivalent to
$$
0 = \sum_{j=1}^{J} x_j ( 2 I  - T D_r) \psi_j.
$$
We observe that $TD_r$ is upper triangular with $(TD_r)_{pp} = 1$. In particular $2 I  - TD_r$ is invertible and the above equation is equivalent to
$$0 = \sum_{j=1}^{J} x_j \psi_j.$$
Finally, since any subset of the $\psi_j$'s of cardinality $r$ forms a basis of $\dC^r$, we deduce that $\dim(\ker (I + \cH)) = J - r$ as requested. If $r=0$, we also find $\dim(\ker (I + \cH)) = J$ since   $f_j = \tilde f_j = \phi_j$, $1 \leq j \leq s$ forms a basis of eigenvectors of $I + \cH$. 

\vspace{5pt}

\noindent {\em Case 2: general multiplicities, non-negative winding number. } We still assume that $\delta = J - r \geq 0$. If $Q$ has roots with arbitrary multiplicities, it should be possible to adapt the above Lemma \ref{le:resH} and Lemma \ref{le:biorth} to a general polynomial $Q$ but the computations with Cauchy's residue formula seems daunting. Alternatively, the strategy of the proof in the general case consists in using that the set of polynomials with simple roots is dense among all polynomials. It however requires some care since the dimension of the kernel is not a continuous function of the entries of a matrix.

Before entering in the technical details, we describe the strategy. Assume that we are given polynomials $Q_{\veps} = a_s \prod_{i=1}^{r +s} (\omega - \lambda_i(\veps))$ with simple roots such that $(\lambda_1(\veps),\ldots,\lambda_{r+s}(\veps))$ converges to  $(\lambda_1 ,\ldots,\lambda_{r+s})$ as $\veps \to 0$. Then the associated matrices $E(\veps)$, $D(\veps)$, $\cH(\veps)$, converges to $E$, $D$, $\cH$. Also, if $r > 0$, $\cT(\veps)$ converges toward the matrix $\cT$, since $\lambda_i \ne 0$ in this case. Assume further that there are vectors, $1 \leq i \leq r+s$,
$$
f'_i(\veps) = \begin{pmatrix}
\phi'_i(\veps) \\
D (\veps) \psi'_i(\veps)
\end{pmatrix}, \quad  f'_i  = \begin{pmatrix}
\phi'_i  \\
D \psi'_i 
\end{pmatrix}, \quad \tilde f'_i(\veps) =  \begin{pmatrix}
\phi'_i(\veps) \\
- D (\veps) \psi'_i(\veps)
\end{pmatrix} ,  \quad f'_i  = \begin{pmatrix}
\phi'_i  \\
- D  \psi'_i 
\end{pmatrix},
$$
such that   $(\tilde f_i)_i$ is a basis of $\dC^{r+s}$, $\phi'_i(\veps)$ converges to $\phi'_i$, $\psi'_i (\veps)$ converges to $\psi'_i $, $(\phi'_i)_{i > J}$ free in $\dC^s$, any subset of the $\psi'_i$'s of cardinality $r$ forms a basis of $\dC^r$ and finally
\begin{equation}\label{eq:Hftilde}
\cH (\veps) \tilde f'_i (\veps) = \IND_{i \leq J} f'_i  (\veps) - \cT(\veps) \tilde f'_i(\veps).
\end{equation}
Then we claim that all these assumptions imply that 
\begin{equation*}
\dim  (\ker (I + \cH)) = J - r. 
\end{equation*}
Indeed, since all matrices and vectors converge, our assumptions imply that 
$
\cH  \tilde f'_i = \IND_{i \leq J} f'_i   - \cT \tilde f'_i,
$
We may then reproduce the argument given in case 1.

Our task is thus to construct this polynomial $Q_\veps$ and these converging vectors $\phi'_i(\veps)$ and $\psi'_i(\veps)$.  For $\lambda \in \dC$ and $\veps > 0$, we define $\phi_{\lambda}(\veps) \in \dC^s$ by $(\phi_{\lambda}(\veps))_{p} = (\lambda + \veps ) ^{r+p -1}$. Similarly $\psi_{\lambda}(\veps) \in \dC^r$ is defined by $\psi_{\lambda}(\veps)_p = (\lambda + \veps )^{p-1}$. By construction, we have $\psi_{\lambda} = \psi_{\lambda}( 0)$ and $\phi_{\lambda} = \phi_{\lambda} (0)$. 
 For integer $m \geq 1$, we set $\phi^{1,m}_{\lambda} (\veps) = \phi_{\lambda} ( (m-1) \veps)$ and $\psi^{1,m}_{\lambda} (\veps) = \psi_{\lambda} ( (m-1) \veps)$. By recursion on $\ell \geq 2$, we define 
 $$
 \psi^{\ell,m}_{\lambda} (\veps) = \frac{ \psi^{\ell-1,m+1}_{\lambda}(\veps)-  \psi^{\ell-1,m}_{\lambda} (\veps)}{\veps} \quad \hbox{ and } \quad  \phi^{\ell,m}_{\lambda} (\veps) = \frac{ \phi^{\ell-1,m+1}_{\lambda}(\veps) -  \phi^{\ell-1,m}_{\lambda} (\veps)}{\veps}
 $$
We denote by $\Phi_{\lambda}^{m,1}(\veps)$, $\Phi_{\lambda}^{1,m}(\veps)$, $\Psi_{\lambda}^{m,1}(\veps)$ and $\Psi_{\lambda}^{1,m}(\veps)$ the matrices whose columns are the vectors, for $1 \leq k \leq m$, $\phi_{\lambda}^{k,1}(\veps), \phi_{\lambda}^{1,k}(\veps)$, $\psi_{\lambda}^{k,1}(\veps)$ and $\psi_{\lambda}^{1,k}(\veps)$ respectively. By construction, 
$$
\Phi_{\lambda}^{m,1} (\veps)  = \Phi_{\lambda}^{1,m} T_m(\veps)\quad  \hbox{ and } \quad \Psi_{\lambda}^{m,1} (\veps)  =\Psi_{\lambda}^{1,m}  T_m(\veps),
$$
where $T_m \in M_{m} (\dR)$ is an upper triangular matrix with coefficients of the diagonal $(1,\veps^{-1},\ldots, \veps^{-m+1})$. In particular $T_{m}(\veps)$ is invertible and the vector space spanned by $(\phi_{\lambda}^{1,1}(\veps),\ldots,\phi_{\lambda}^{m,1}(\veps))$ is equal to the vector space spanned by $(\phi_{\lambda}^{1,1}(\veps),\ldots,\phi_{\lambda}^{1,m}(\veps))$. The same comment holds for the $\psi_{\lambda}$'s.

We also have, by recursion on $m$, that
$$
\lim_{\veps \to 0} \phi_{\lambda}^{m,1} (\veps) =   \phi_{\lambda}^{m} \quad  \hbox{ and } \quad \lim_{\veps \to 0} \psi_{\lambda}^{m,1} (\veps) =   \psi_{\lambda}^{m},
$$
where for $1 \leq p \leq s$, $(\phi_{\lambda}^{m})_p = (r+p-1)_{m-1} \lambda^{r+p-m}$,  for $1 \leq p \leq r$, $(\psi_{\lambda}^{m})_p = (p-1)_{m-1} \lambda^{p-m}$  and we have used the Pochhammer symbol: $(k)_0 = 1$, $(k)_l = k (k-1)\ldots (k- l+1)$ for integer $l \geq 1$. The matrix $\Psi_{\lambda}^{m}$ whose column are $(\psi_{\lambda}^1,\ldots,\psi_{\lambda}^m)$ is lower triangular. Hence, we get the key property, for any $\veps > 0$,
\begin{equation*}\label{eq:diffphi}
\dim (\SPAN (\psi^1_{\lambda},\ldots, \psi^m_{\lambda})) = \min(m,s) = \dim (\SPAN (\psi^{1,m}_{\lambda}(\veps),\ldots, \psi^{1,m}_{\lambda}(\veps))),
\end{equation*}
and similarly for the $\phi_{\lambda}'$s.

We are ready to define the vectors $\phi'_i(\veps)$ and $\psi'_i(\veps)$. Let $d \geq 1$ be the number of distinct roots of $Q$, $(\lambda_1,\ldots,\lambda_d)$ be those roots and denote by $m_j$ the multiplicity of $\lambda_j$. For $\veps  >0$, we consider the polynomial $Q_\veps$ of degree $r+s$ with leading coefficient $a_s$ and whose  roots are $(\lambda_j + (k-1)\veps)$, $1 \leq j \leq d, 1 \leq k \leq m_j$. We reorder these roots as $(\lambda_i(\veps))_i$, $1 \leq i \leq r+s$ by non-decreasing magnitude. For $\veps$ large enough, all roots $\lambda_i(\veps)$ are distinct and there are precisely $J$ roots in the unit disc $\mathbb{D}$.  We consider the vectors $\phi'_{j,k}(\veps) = \phi_{\lambda_j}^k(\veps)$, $\psi'_{j,k}(\veps) = \psi_{\lambda_j}^k(\veps)$, $1 \leq k \leq m_j$. We re-order them to get vectors $\phi'_i(\veps)$ and $\psi'_i(\veps)$, $1 \leq i \leq r+s$. We can check that all required properties are satisfied. Indeed, \eqref{eq:Hftilde} follows from Lemma \ref{le:resH} and Lemma \ref{le:biorth}, and all other properties have been discussed above.

\vspace{5pt}

\noindent {\em Case 3: negative winding number. } We now assume that $J < r$. In this case, we  may reverse the sense of integration: more precisely, we write 
$$
\oint_{\dS^1} \frac{\omega^p}{Q(\omega)} d \omega = \oint_{\dS^1} \frac{\omega^{-p}}{Q(\omega^{-1})} d \omega =  \oint_{\dS^1} \frac{ \omega^{r+s -p}}{\widehat Q(\omega)} d\omega,
$$
where $\widehat{Q}(\omega) = \omega^{r+s} Q(\omega^{-1})$ is the reciprocal polynomial of $Q$ whose roots are $(\lambda_j^{-1})_{j}$. We may thus perform the same analysis of $I + \cH$ with $Q$ replaced by $\widehat{Q}$, $J$ replaced by $\widehat{J} = r+s - J$, $(r,s)$ replaced by $(\widehat r,\widehat{s} ) = (s,r)$. In particular, since $J < r$, we have $\hat J > \widehat{r}$. It thus suffices to reproduces the same argument to conclude. (At the level of matrices it amounts to replace $\cH$ by its transpose $\cH^\intercal$).
\end{proof}

\subsection{Convergence of the covariance kernel}
In this subsection, we perform some computations on the kernel of $\tilde W_n$ which appears in Theorem \ref{th:cvWn}. Recall the definitions of $\cA_n$, $\cB_n$, $\theta_n$ and their analogs with a prime defined in Subsection \ref{subsec:cvWn}. Recall finally, the definitions of $\cA$, $\cB$, $\theta$ in Theorem \ref{th:WnW}. We gather these computations in the following proposition.

\begin{prop}\label{prop:kernel}
Let $\Gamma$ be a compact subset of $\dC \backslash \bfa (\dS^1)$. Uniformly over all $z,z' \in \Gamma$, we have the following convergence:
$$
\lim_n \theta_n(z,z') = \theta ( z , z') \; , \quad 
\lim_n \cA_n(z,z') = \cA(z,z')
\quad \hbox{ and } \quad \lim_n \cB_n(z,z') = \cB(z,z).
$$
If $U_n = F_n$ with $\veps = 1$ or if $U_n$ is in the orthogonal group with $\veps = -1$, then 
$$
\lim_n \theta'_n(z,z') = \theta_\veps ( z , z') \; , \quad 
\lim_n \cA'_n(z,z') = \cA_\veps(z,z')
\quad \hbox{ and } \quad \lim_n \cB'_n(z,z') = \cB_\veps (z,z).
$$
\end{prop}
\begin{proof}
The proofs are entirely similar to the proof of Proposition \ref{convH}. We only give the proofs of the convergence to $\cA(z,z')$ and $\cA_{-1}(z,z')$. We start with the convergence of $\cA_n(z,z')$. Using the notation of Proposition \ref{convH}, we write $R'_n(z) = F_n \Delta_n(z) F_n^*$, we find 
$$
\cA_n(z,z') = Q_n F_n \Delta_n(z) \overline \Delta_n(z') F_n^* Q_n^* =   \cD \begin{pmatrix}
B_1  \Delta_n(z) \overline \Delta_n(z') B_1^* & B_1    \Delta_n(z) \overline \Delta_n(z') B_3^* \\
B_3  \Delta_n(z) \overline \Delta_n(z') B_1^* & B_3  \Delta_n(z) \overline \Delta_n(z') B_3^*
\end{pmatrix} \cD^*,
$$
where $B_1,B_3$ are as in \eqref{F}. We may then argue as in Proposition \ref{convH}. For $1 \leq p,q \leq s$, 
\begin{align*}
(B_1 \Delta_n(z)\overline  \Delta_n(z')B_1^* ) _{p,q}
& =\frac1n \sum_{k=1}^n \omega_n^{(p-1)(k-1)} \frac{1}{(z-\bfa(\omega_n^{k-1}))\overline{ (z'-\bfa(\omega_n^{k-1}) )}}  \omega_n^{-(q-1)(k-1)}\\
& = \frac1n \sum_{k=1}^n e^{i 2 \pi (p-q)(k-1)/n}  \frac{1}{(z-\bfa(e^{i 2 \pi (k-1)/n})) \overline{(z'-\bfa(e^{i 2 \pi (k-1)/n}) )}} \\
& \underset{n \to \infty}{\longrightarrow} \frac{1}{2 \pi} \int_0^{2 \pi} \frac{e^{i (p-q) \theta}}{(z-\bfa(e^{i\theta}))\overline{ (z'-\bfa(e^{i\theta}))} } d\theta.
\end{align*}
Similarly, for $1 \leq p \leq s$, $1 \leq q \leq r$,
\begin{align*}
(B_1 \Delta_n(z) \overline \Delta_n(z') B_3^* ) _{p,q}
& =\frac1n \sum_{k=1}^n \omega_n^{(p-1)(k-1)} \frac{1}{(z-\bfa(\omega_n^{k-1}))\overline{ (z'-\bfa(\omega_n^{k-1}) )}}  \omega_n^{-(n-r+q-1)(k-1)}\\
& \underset{n \to \infty}{\longrightarrow} \frac{1}{2 \pi} \int_0^{2 \pi} \frac{e^{i (p-q+r) \theta}}{(z-\bfa(e^{i\theta}))\overline{ (z'-\bfa(e^{i\theta}))} } d\theta.
\end{align*}
For $1 \leq p \leq r$, $1 \leq q \leq s$,
\begin{align*}
(B_3 \Delta_n(z) \overline  \Delta_n(z') B_1^* ) _{p,q}
& =\frac1n \sum_{k=1}^n \omega_n^{(n - r + p-1)(k-1)} \frac{1}{(z-\bfa(\omega_n^{k-1}))\overline{ (z'-\bfa(\omega_n^{k-1}) )}}  \omega_n^{-(q-1)(k-1)}\\
& \underset{n \to \infty}{\longrightarrow} \frac{1}{2 \pi} \int_0^{2 \pi} \frac{e^{i (p-q-r) \theta}}{(z-\bfa(e^{i\theta}))\overline{ (z'-\bfa(e^{i\theta}))} } d\theta.
\end{align*}
Finally, for $1 \leq p,q \leq r$, 
\begin{align*}
(B_3 \Delta_n(z) \overline  \Delta_n(z')B_3^* ) _{p,q}
& =\frac1n \sum_{k=1}^n \omega_n^{(n-r+p-1)(k-1)} \frac{1}{(z-\bfa(\omega_n^{k-1}))\overline{ (z'-\bfa(\omega_n^{k-1}) )}}  \omega_n^{-(n-r+ q-1)(k-1)}\\
& \underset{n \to \infty}{\longrightarrow} \frac{1}{2 \pi} \int_0^{2 \pi} \frac{e^{i (p-q) \theta}}{(z-\bfa(e^{i\theta}))\overline{ (z'-\bfa(e^{i\theta}))} } d\theta.
\end{align*}
We deduce easily that $\cA_n(z,z') \to \cA(z,z')$.

We next suppose that $U_n$ is orthogonal, that is $U_n U_n^\intercal = I_n$. We prove the convergence of $\cA'_n(z,z)$ in this case. We observe that  
$$
F_n F_n^{\intercal}  = F_n^\intercal F_n   = K_n = \begin{pmatrix} 1 & && \\
 & & & 1 \\
 & & \iddots & \\
 & 1 & &
\end{pmatrix}.
$$
That is, $ (F_n F_n^{\intercal} )_{pq} = \1(p+q-2 = 0 \; [n])$. Since $U_n U_n^\intercal = I_n$ and $F_n F_n^{\intercal}$ is a real matrix, we get
$$
\cA'_n(z,z') = Q_n F_n \Delta_n(z) K_n \Delta_n(z')  F^\intercal_n Q_n^\intercal  =  \cD \begin{pmatrix}
B_1  \Delta_n(z) K_n \Delta_n(z') B_1^\intercal & B_1  \Delta_n(z) K_n \Delta_n(z') B_3^\intercal \\
B_3  \Delta_n(z) K_n \Delta_n(z') B_1^\intercal & B_3  \Delta_n(z) K_n \Delta_n(z') B_3^\intercal
\end{pmatrix} \cD^\intercal. 
$$
We note that 
\begin{align*}
(B_1 \Delta_n(z) K_n \Delta_n(z') B_1^\intercal  )_{p,q} & =\sum_{k=1}^n \omega_n^{(p-1)(k-1)} \frac{1}{(z-\bfa(\omega_n^{k-1})){ (z'-\bfa(\omega_n^{n-k +2 -1}) )}}  \omega_n^{(q-1)(n-k +2 -1)}\\
& \underset{n \to \infty}{\longrightarrow} \frac{1}{2 \pi} \int_0^{2 \pi} \frac{e^{i (p-q) \theta}}{(z-\bfa(e^{i\theta}))  (z'-\bfa(e^{-i\theta})) } d\theta.
\end{align*}
The rest of the proof repeats similar computations.
\end{proof}

We conclude this subsection with a lemma on the rank of the covariance kernel. 

\begin{lem}\label{le:rankK}
For any $z \in \dC \backslash \supp(\beta_\sigma)$, the matrix $\cA(z,z) \otimes \cB(z,z)^\intercal / (1 - \sigma^2 \theta(z,z))$ is positive-definite.
\end{lem}
\begin{proof}
Since $z \notin \supp(\beta_\sigma)$, we have $ \sigma^2 \theta(z,z) < 1$. We need to prove that $\cA(z,z) \otimes \cB(z,z)^\intercal$ is invertible. Since the rank is multiplicative under tensor product, we need to prove  that $\cA(z,z) $ and $\cB(z,z)$ are invertible. We prove the claim for $\cA(z,z)$.  Since $\cD$ and $\cE$  are invertible is suffices to prove that 
$$
\mathcal K  = \mathcal K(z) =   \oint_{\dS^1}\begin{pmatrix}I_s &0\\0 & \omega^{-(r+s)}I_r\end{pmatrix} \cM(\omega)\begin{pmatrix}I_s &0\\0 & \omega^{r+s}I_r\end{pmatrix}\frac{d\omega}{\omega | z - \bfa(\omega)|^2}
$$
is invertible.
Let $f \in \dC^{r+s}$, we write
$$
f = \begin{pmatrix}
f_1 \\ 
f_2
\end{pmatrix} 
\quad \hbox{ and } \quad f(\omega) = \begin{pmatrix}
f_1 \\ 
 \omega^{r+s} f_2
\end{pmatrix} 
$$
with $f_1 \in \dC^s$ and $f_2 \in \dC^r$. Observe that $\cM(\omega) = \varphi(\omega)  \varphi(\omega)^*$ with $\varphi(\omega)_p = \omega^p$. Let $C = 1 /   \mathrm{dist}(z,\bfa(\dS^1))^2 $, we get
\begin{align*}
\langle f , \mathcal K f \rangle &= \frac{1}{2 \pi} \int_0 ^{2\pi} \frac{ |\langle \varphi(e^{i\theta}) ,  f(e^{i \theta}) \rangle |^2 }{| z - \bfa(e^{i\theta})|^2 } d \theta \\
 & \geq \frac{C}{2 \pi}   \int_0 ^{2\pi}   |\langle \varphi(e^{i\theta}) ,  f(e^{i \theta}) \rangle |^2  d \theta \\
 & = \frac{C}{2 \pi}    \int_0 ^{2\pi}   \left|\sum_{p=1}^s f_{1}(p)  e^{-i\theta p} + \sum_{q=1}^r f_{2}(q)  e^{-i\theta (q-r)}  \right|^2  d \theta \\
 & = \frac{C}{2 \pi}  \int_0 ^{2\pi}   \sum_{ 1 \leq p,p' \leq s}   \bar f_{1}(p) f_{1}(p') e^{i\theta (p-p')} +  \sum_{ 1 \leq q,q' \leq r}  \bar f_{2}(q) f_{2}(q') e^{i\theta (q-q')}  \\ 
 & \quad \quad  \quad \quad  \quad \quad  + 2 \Re \left(\sum_{1 \leq p \leq s , 1 \leq q \leq r}  \bar f_{1}(p)f_2 (q) e^{i \theta (p - q + r)} \right)  d \theta \\
 & = C \left( \| f_1 \|^2  + \|f_2 \|^2 \right) = C \| f\|^2,
\end{align*}
where at the last line, we have used the linearity of the trace and the orthonormality of the functions $e_p  : \theta \to e^{i p \theta}$ in $L^2 ( [0,2\pi], d\theta/(2\pi))$.  
We deduce that $\cA(z,z)$ is positive definite. The proof for $\cB(z,z)$ is identical (it is even simpler). The conclusion follows.
\end{proof}

\subsection{Proof of Theorem \ref{th:WnW}}
We first claim that $
\max_{p,q} \dE | W^{(1)} _n(z)_{p,q} |^4  $ is bounded uniformly in $n$ and $z \in \Gamma = \bar \Omega$. Indeed, since the operator norms of $P$, $Q$, $U_n$ and $R'_n(z)$ are bounded uniformly, it suffices to check that there exists a constant $C >0$ such that for any deterministic vectors $\phi, \psi \in \dC^n$ of unit norms, 
\begin{equation}\label{eq:tight4}
\dE |\langle \phi , X \psi \rangle|^4    \leq C.
\end{equation}
To this end, we write
\begin{align*}
\dE |\langle \phi , X \psi \rangle| ^4 & = \dE \langle \phi , X \psi \rangle^2 \overline{\langle \phi , X \psi \rangle}^2  \\
& = \sum_{i,j} \bar \phi_{i_1} \bar \phi_{i_2} \phi_{i_3}  \phi_{i_4} \dE [ X_{i_1j_1} X_{i_2j_2} \bar X_{i_3j_3} \bar X_{i_4j_4} ]    \psi_{j_1}  \psi_{j_2} \bar \psi_{j_3} \bar  \psi_{j_4},
\end{align*}
where the sum is over all $1 \leq i_l ,j_l \leq N$, $1 \leq l \leq 4$. Using assumptions $(H_0)$ and $(H_4)$, we deduce that, with $m = \dE |X_{ij}|^4$, we have
$$
\dE |\langle \phi , X \psi \rangle| ^4  \leq m \sum_p \sum_{i,j} \delta_p (i,j) | \phi_{i_1}  \phi_{i_2} \phi_{i_3}  \phi_{i_4}    \psi_{j_1}  \psi_{j_2} \psi_{j_3}   \psi_{j_4} |,
$$
where the sum is over all pairings $p$ of $\{1,2,3, 4\}$ (that is $p$ is one of the three pair partitions of $\{1, \ldots , 4\}$) and $\delta_p(i,j) = 1$ if and only if for all pairs $(a,b)$ of $p$, $(i_a , j_a) = (i_{b},j_b)$ and $\delta_{p} (i,j) = 0$ otherwise. In particular, we find
$$
\dE |\langle \phi , X \psi \rangle| ^4 \leq 3 m \sum_{i_1,i_2,j_1,j_2} | \phi_{i_1}  \phi_{i_2}  \psi_{j_1}  \psi_{j_2} |^2 = 3m.
$$
This proves \eqref{eq:tight4}.

Theorem \ref{th:WnW} is then an immediate consequence of Theorem \ref{th:cvWn}, Proposition \ref{prop:kernel} and the uniform integrability of $W^{(1)}_n(z)_{pq} W^{(1)}_n(z')_{p'q'}$. The only remaining statement that we need to check is that when $U_n = I_n$ then $\tilde W_n^{(1)} = Q_n R'_n(z) X R'_n(z) P_n$ converges weakly. Write 
$$
R'_n(z) = \begin{pmatrix}
G_{1}\\
G_{2} \\
G_{3} \\
\end{pmatrix}  = \begin{pmatrix}
\tilde G_{1}^\intercal &
\tilde G_{2}^\intercal  &
\tilde G_{3}^\intercal 
\end{pmatrix},
$$
where $G_{1}$, $\tilde G_{3}$ have size $s \times n$ and $G_{3}, \tilde G_{1}$ have size $r \times n$, 
 We have
\begin{align*}
\tilde W_n^{(1)}  = 
\cD
\begin{pmatrix}
G_1 X \tilde G_3^\intercal & G_1 X \tilde G_1^\intercal \\
G_3 X \tilde G_3^\intercal & G_3 X \tilde G_1^\intercal
\end{pmatrix}
 \cE.
\end{align*}
We should therefore check that the above matrix in the middle converges. To this end, we observe $(R'_n(z))_{pq} = \gamma_n(p-q)$ for some $n$-periodic function $\gamma_n : \mathbb Z \to \dC$.  If $p-q = \ell \; \mathrm{ mod} (n)$ we have
\begin{align}
\gamma_n(\ell) = (R'_{n}(z))_{pq}  &= (F_n \Delta_n(z) F_n^*)_{pq} \nonumber \\
& = \frac 1 n \sum_{k=1}^n \omega_n^{(p-1)(k-1)}\frac{1}{z-\bfa(\omega_n^{k-1})}\omega_n^{-(q-1)(k-1)}   \nonumber \\
&  \underset{n \to \infty}{\longrightarrow} \gamma(\ell)  =  \frac{1}{2 \pi} \int_0^{2\pi} \frac{e^{i\theta \ell }}{z - \bfa(e^{i\theta})} d\theta. \label{eq:gammaldhedhe} 
\end{align}

Next, we consider the infinite array $(\tilde X^{(n)}_{ij})_{i,j \in \mathbb Z}$ defined by $\tilde X^{(n)}_{ij} = X_{i_n,j_n}$ where $i_n = i$, $j_n = j$ $\mathrm{mod}(n)$ and $1\leq i_n , j_n \leq n $. The infinite array $\tilde X^{(n)}$ converges in distribution (for the product topology on infinite arrays) toward $\tilde X  = ( X_{ij})_{i,j \in \mathbb Z}$ an infinite array of iid copies of $X_{11}$.

We may now check that $\tilde W_n^{(1)}$
converges weakly to an explicit matrix expressed in terms of $\gamma$ and $\tilde X$.  Consider for example, the entry $G_1 X \tilde G_3^\intercal$. For $1 \leq p , q \leq s$, we have 
\begin{align*}
(G_1 X \tilde G_3^\intercal)_{p,q} & =  \sum_{a = 1}^n \sum_{b=1}^{n} (R'_{n}(z))_{p a } X_{a ,b} (R'_n(z) )_{b, n-s + q } \\  
& =  \sum_{a , b \in D_n }  \gamma_n(p-a)  \tilde X^{(n)}_{a ,b} \gamma_n (b + s-q)  ,
\end{align*}
where $D_n = \{ -\lceil n / 2 \rceil +1 , \ldots \lfloor n /2 \rfloor \}$ and we have performed the change of variable from $\{1,\ldots , n \} \to D_n$, $a \to a \IND (a \leq n/2) + (a - n) \IND( a > n/2) $ and used the $n$-periodicity of $\tilde X^{(n)}$ and $\gamma_n$.
Similarly, for $1 \leq p \leq s$, $1 \leq q \leq r$,
\begin{align*}
(G_1 X \tilde G_1^\intercal)_{p,q} & =  \sum_{a , b \in D_n}  \gamma_n(p-a)  \tilde X^{(n)}_{a ,b} \gamma_n (b-q)  .
\end{align*}
For $1 \leq p \leq r$, $1 \leq q \leq s$,
\begin{align*} 
(G_3 X \tilde G_3^\intercal)_{p,q} & =  \sum_{a , b \in D_n}  \gamma_n(p-r-a)  \tilde X^{(n)}_{a ,b} \gamma_n (b + s-q)  .
\end{align*}
For $1 \leq p \leq r$, $1 \leq q \leq r$,
\begin{align*} 
(G_3 X \tilde G_1^\intercal)_{p,q} & =  \sum_{a, b \in D_n}  \gamma_n(p-r-a)  \tilde X^{(n)}_{a ,b} \gamma_n (b-q)  .
\end{align*}
Note also that assumption $(H_0)$ and Markov inequality imply that the event \{\hbox{for all} $a,b$, $| \tilde X^{(n)}_{a ,b} | \leq C ( {|a| + |b|})^2\}$ has probability tending to one as $C \to \infty$. Using Lemma \ref{le:bdres}, we deduce by dominated convergence that $\tilde W_n^{(1)}$ converges weakly toward 
\begin{equation}\label{eq:W1UI}
    W^{(1)} = \cD \cA \cE,
\end{equation} 
where $\cA$ is obtained by taking the limit $n\to \infty$ in the above formulas. More precisely, for $1 \leq p , q \leq r+s$, we have
$$
\cA_{pq} =  \sum_{ (a, b) \in \mathbb{Z}^2}  \gamma(p-a + \veps_p)  X_{a ,b} \gamma (b + s-q),
$$
with $\veps_p = -(r+s) \IND ( p > s)$
(recall that $\gamma$ depends implicitly on $z$). It concludes the proof of the convergence.\qed

\subsection{Convergence of the outlier eigenvalues : proof of Theorem \ref{th:W2Out}}

\noindent{\em Proof of the first claim: convergence of the outlier eigenvalues if $\varphi$ is a.s. not  the zero function. } Let $\Omega$ be a bounded open connected subset of $\cR_{\sigma,\delta}$, $\delta \ne 0$, as in the statement of Theorem \ref{th:W2Out}.  Set $\Gamma = \bar \Omega \subset \cR_{\sigma,\delta}$ be the closure of $\Omega$. On the event $\{\spectrum(S_n) \cap \Gamma = \emptyset\}$, we set 
\[
\varphi_n\colon z \mapsto \left( \frac{\sqrt n}{\sigma_n} \right)^{|\delta|}\det \left( I_{r+s} + Q_n R_{n}(z)P_n  \right),
\]
 and $\varphi_n  = 0$ otherwise. The function $\varphi_n$ is a random analytic function on $\Omega$. Next, the event $\{\sup_{z \in \Gamma} \| G_n(z) \| \leq 1 \}$ has probability converging to $1$ by Proposition~\ref{convunif}.
Multiplying  \eqref{dev-jacobi2} by $\left(\frac{\sqrt n}{\sigma_n} \right)^{|\delta|}$, we have, on the  event $\{\sup_{z \in \Gamma} \| G_n(z) \| \leq 1\}$, that for any $z\in \Gamma$,
 \begin{align}
\varphi_n(z)
 & =  \left(\frac{\sqrt n}{\sigma_n} \right)^{|\delta|} \sum_{k=0}^{r+s} \Tr\left( \adj_k( I_{r+s} + Q_n R'_n(z) P_n ) \bigwedge\nolimits^k (
 G_n(z)) \right) \nonumber \\
 & =  \sum_{k=0}^{|\delta|}  \left(\frac{\sqrt n}{\sigma_n} \right)^{|\delta|-k} \Tr\left( \adj_k( I_{r+s} + \mathcal H(z) ) \bigwedge\nolimits^k \left( W_n(z) \right) \right) + \varepsilon_n(z) \nonumber,
 \end{align}
  where, for some constant $C = C_\Gamma$, 
 \begin{align*} \label{eq:maj-epsilon}
\| \varepsilon_n \|_\Gamma  =      \sup_{z \in \Gamma}| \varepsilon_n(z) | &  \leq C \left(\frac{\sqrt n}{\sigma_n} \right)^{|\delta|} \{ \sup_{z \in \Gamma} \| Q_n R'_n(z) P_n - \mathcal H(z)\| \} \vee  \{ \sup_{z \in \Gamma} \| Q_n R'_n(z) P_n - \mathcal H(z)\| \}^{|\delta|} \\
& \quad \quad \quad \quad \quad \quad + C \left\{ \left(\frac{\sigma_n}{\sqrt n} \right) \sup_{z \in \Gamma} \| W_n(z) \| \right\} \vee \left\{\left(\frac{\sigma_n}{\sqrt n} \right) \sup_{z \in \Gamma} \| W_n(z) \| \right\}^{r+s - |\delta|} .
 \end{align*}
Since $\frac 1 n \ln \sigma_n \to 0$, by Proposition~\ref{convH} and Theorem \ref{th:cvWn}, we have that, in probability,
$$\lim_{n \to \infty} \| \varepsilon_n \|_\Gamma  = 0.$$
 
 Now, from   Proposition~\ref{prop:rank}, $ \adj_k( I_{r+s} + \mathcal H(z) )=0$ for $k<|\delta|$. We find 
\begin{equation}
 \label{eq:dev-final}
 \varphi_n(z) =  \Tr\left( \adj_{|\delta|}( I_{r+s} + \mathcal H(z) ) \bigwedge\nolimits^{|\delta|} \left( W_n(z)\right) \right)  + \varepsilon_n(z).
\end{equation}
From the continuous mapping theorem, along any accumulation point $W$ of $W_n$, $\varphi_n$ converges weakly in $\cH(\Omega)$ toward 
$$
\varphi (z) = \Tr\left( \adj_{|\delta|}( I_{r+s} + \mathcal H(z) ) \bigwedge\nolimits^{|\delta|} \left( W(z)\right) \right) 
$$
 defined in the statement of Theorem \ref{th:W2Out}. If $\varphi$ is a.s. not the constant function, the first claim of the theorem is then a consequence of Proposition \ref{Shirai:zeroes}  in Appendix.

 \vspace{5pt}
 
 \noindent{\em Proof of the second claim: if $W$ is Gaussian  $\varphi$ is a.s. not the zero function. }  It remains to prove that if $W$ is Gaussian then  $\varphi$ is a.s. not the constant function. We start with a general lemma on Gaussian vectors. Recall that we say that a random vector $X$ on $\dC^d$ is Gaussian if $(\Re(X), \Im(X))$ is a real Gaussian vector on $\dR^{2d}$. 
 
 \begin{lem}\label{le:acG}
 Let $d \geq 1$ and $X$ be a centered Gaussian vector on $\dC^d$ such that
 $ \dE X X^*$ is positive definite. Then there exists $Z = (Z_1,\ldots,Z_d)$ centered Gaussian on $\dC^d$ with $\dE Z Z^* = I_d$ such that the laws of $X$  and $Z$ are equivalent measures.  
 \end{lem}
\begin{proof}
Recall that the support of the law of a real Gaussian vector on $\dR^k$ is a linear subspace of $\dR^k$. Then, the proof is based on the observation that two real Gaussian vectors $Y_1$ and $Y_2$ on $\dR^k$ have equivalent laws if their supports are equal. This is immediate since Gaussian distributions have positive densities on their supports (with respect to the Lebesgue measure on their support).  Now, we identify $\dC^d$ with the real vector space, say $C_d$, of dimension $2d$, spanned by $( e_k,ie_k), 1 \leq k \leq d,$ where $(e_1,\ldots,e_d)$ is the canonical basis of $\dC^d$. We see $X = \Re(X) + i \Im(X)$ as a real Gaussian vector on $C_d$ and denote by $V$ its support. Let $E_k$, $1 \leq k \leq d$, be the vector subspace of $C_d$ spanned by $(e_k,i e_k)$, we may identify $E_k$ with a copy of $\dC$. We finally set $V_k = E_k \cap V$. Our assumption implies that $\dim (V_k) \in \{1,2\}$ and 
$$
V = V_1 \oplus V_2 \ldots \oplus V_d.
$$ It remains to define $Z$ as the Gaussian vector whose coordinates $(Z_1,\ldots,Z_d)$ in $\dC^d$ are centered independent, $Z_k$ has support $V_k$ and $\dE |Z_k|^2 = 1$ (if $\dim(V_k) = 2$, we further assume that $\dE Z_k^2 = 0$ to fully characterize the law of $Z_k$). \end{proof}

We may prove the second claim of the theorem. We fix $z \in \Gamma$ and set $A =\adj_{|\delta|}( I_{r+s} + \mathcal H(z) ) $.  It suffices to prove that $\varphi(z) \ne 0$ with probability one. By Lemma \ref{le:acG} and Lemma \ref{le:rankK} it suffices to prove that $T \ne 0$ with probability one, where  
$$
T =  \Tr\left( A \bigwedge\nolimits^{|\delta|} \left( Z\right) \right) 
$$
and $Z \in M_{r+s}(\dC)$ is a centered Gaussian matrix with independent indices of unit variance. We recall that, indexing the entries of the matrix  $\bigwedge\nolimits^{|\delta|} \left( Z\right)$ by pairs of subsets $(I,J)$ of $\{1,\ldots, r+s\}$ of cardinal $|\delta|$, the entry $(I,J)$  of $\bigwedge\nolimits^{|\delta|} \left( Z\right)$ is $ \det Z(I|J)$. Next, we write 
$$
\det  Z (I|J) = \sum_{\sigma} (-1)^{\sigma} \prod_{i \in I} Z_{i,\sigma(i)},
$$
where the sum is over all bijections $\sigma$ from $I$ to $J$ and $(-1)^{\sigma}$ is the corresponding signature. Since $Z_{i,j}$ are independent, centered of unit variance, we have the classical orthogonality relations: if $(I,J) \ne (I',J')$ 
$$
\dE [ \det Z (I|J) \overline{ \det Z (I'|J')} ] = 0  
$$
and 
$$
\dE [ | \det Z (I|J) |^2 ] = \sum_{\sigma}   \prod_{i \in J} \dE | Z_{i,\sigma(j)}|^2 = |\delta|!,
$$
(these relations are central in \cite{basak-zeitouni20}).  
Thus, we get 
$$
\dE [ | T |^2 ]= \dE \left| \sum_{I,J} A_{J,I} \det ( Z (I|J) )  \right|^2 = \sum_{I,J} |A_{J,I}|^2 |\delta|!.
$$  
Proposition \ref{prop:rank} implies that  $A \ne 0$ and thus
$$
\dE [ | T |^2 ] > 0.
$$
Finally, since $T$ is a polynomial in the variables $(Z_{ij})$, the Carbery-Wright inequality \cite{zbMATH01659346} implies that $T \ne 0$ with probability one.

 \vspace{5pt}

\noindent{\em Proof of the third claim: if $U_n = I_n$ and \Hac holds then $\varphi$ is a.s. not the zero function. }  
We finally prove that in the case $U_n = I_n$ then $\varphi$ is a.s.\ 
non-zero provided that the law of $X_{11}$ is absolutely continuous with respect to the Lebesgue measure on a real vector subspace, say $K$, of $\dC \simeq \dR^2$, $\dim _{\dR}(V) \in \{1,2\}$. Without loss of generality, up to taking a rotation, we can assume that $K = \mathbb R$ or $K = \dC$. We fix some $z \in \Gamma$.

From what precedes, it is enough to prove that the law of $W(z)$ is absolutely continuous with respect to a Gaussian law. Since $W(z) = W^{(1)}(z) + \sigma W^{(2)}(z)$ and $(W^{(1)},W^{(2)})$ independent, it suffices to check that the law of $W^{(1)}(z)$ is absolutely continuous with respect to a Gaussian law. This amounts to check that the law of $W^{(1)}(z)$ is absolutely continuous with respect to the Lebesgue measure on a vector subspace of $M_{r+s}(\dC) \simeq \dR^{2 ( r+s)^2}$.

Recall the definition of $W^{(1)}(z)$ in  \eqref{eq:W1UI}. Since $\cD$ and $\cE$ are invertible, it suffices to check that the marginal laws of  $\cA(z)$ is absolutely continuous with respect to the Lebesgue measure on real vector subspaces of $M_{r+s}(\dC) \simeq \dR^{2 (r+s)^2}$. To this end, recall that, for $1 \leq p, q \leq r+s$, we have 
$$
\cA_{p,q} =  \sum_{(a,b) \in \mathbb Z^2}  \gamma(p-a + \veps_p)  X_{a ,b} \gamma (b+s-q),
$$
where $\tilde X = (X_{a,b})_{(a,b) \in \mathbb Z^2}$ is an infinite array of iid copies of $X_{11}$, $\veps_p = -(r +s) \IND (p > s)$ and $\gamma (\ell)$ is given in \eqref{eq:gammaldhedhe}. By Lemma \ref{le:bdres}, for some $0 < \kappa < 1$ and $C > 0$,
$$
\gamma(\ell) \leq C \kappa^{|\ell|}.
$$

The proof is ultimately a consequence of the fact that the Lebesgue measure on $\dR^n $ is invariant by orthogonal transformations and that $\cA$ can be regarded as the image of the infinite array $\tilde X = (X_{a,b})$ by a linear map of finite rank. The first part of the proof clarifies this last point. We fix some $\beta$ such that  $\kappa^2 < \beta < 1$, we consider the Hilbert space $H = \ell^2 ( \mathbb Z^2 , w) $ of infinite arrays $(x_{a,b})_{a,b \in \mathbb Z^2}$, $x_{a,b} \in K$, equipped with the scalar product 
$$
\langle x , y \rangle_H = \sum_{a,b} \bar x_{ab} y_{ab} \beta^{|a| + |b|}.
$$
We may then consider the linear map $T$ from $H$ to $ M_{r+s}(\dC) \simeq \dR^{2 (r+s)^2}$ defined by, for $1 \leq p,q \leq r+s$,  $$
T(x)_{pq} =   \sum_{(a,b) \in \mathbb Z^2}  \gamma(p-a + \veps_p)  x_{a ,b} \gamma (b+s-q).
$$
It is indeed a bounded linear map since, by Cauchy-Schwarz inequality,
\begin{align*}
|T(x)_{pq} |  & \leq C^2  \kappa^{2(r+s)} \sum_{a ,b }  \kappa^{|a| + |b| }   |x_{a b} | \\\
& \leq C^2  \kappa^{2(r+s)} \sum_{a , b }  \left(\frac{\kappa}{ \sqrt \beta}\right) ^{|a| + |b|}   |x_{a b} | \beta^{(|a| + |b|)/2} \\
 &  \leq C^2  \kappa^{2(r+s)} \sqrt{ \sum_{a , b }  \left(\frac{\kappa^2}{  \beta}\right) ^{|a| + |b|}  } \| x \|_H.
\end{align*}
We equip $M_{r+s}(\dC) \simeq \dR^{2(r+s)^2}$  with the Euclidean scalar product.  Let $T^*$ be the adjoint of $T$. Let $H_0 \subset H$ be  the image of $T^*$ and $H^*_0 \subset M_{r+s}(\dC)$ the image of $T$. We have $\dim_{\dR}(H_0)  = \dim_{\dR}(H^*_0) =d \in \{ 1 , \ldots, 2(r+s)^2 \}$ ($d \ne 0$ follows from Lemma \ref{le:rankK}). From the spectral theorem applied to $T T^*$ and $T^* T$, we obtain the singular value decomposition of $T$: for any $x \in H$,
$$
T (x ) = \sum_{i=1}^d \sigma_i u_i \langle f_i , x \rangle_H,
$$
where $\sigma_i > 0$, $(f_1,\ldots,f_d)$ and $(u_1,\ldots, u_d)$ are orthonormal basis of $H_0$ and $H^*_0$ respectively (for their respective scalar products).

Now, let $(E_{a,b}), (a,b) \in \mathbb Z^2$, be the canonical basis of $H$. For integer $n \geq 1$, let $\pi_n$ be the orthogonal projection from $H$ onto $H^n$ the subspace spanned by $\{ x_{a,b} E_{a,b} :  \max (|a|,|b|) \leq n , x_{a,b} \in K \}$.  For integer $n \geq 1$ large enough, we have $\pi_n(f_i) \ne 0$, for all $1 \leq i  \leq d$. In particular, for all $n$ large enough, the rank of $T_n = T \circ \pi_n$ is also $d$.

 We fix such integer $n \geq 1$. Let $T_n^*$ be the adjoint of $T_n$ now for the usual Euclidean scalar product on $H^n$,
$$
\langle x ,  y \rangle = \sum_{a , b} \bar x_{ab} y_{ab}.
$$
Let  $H_{0}^n \subset H^n$ be the image of $T_n^*$. The singular value decomposition of $T_n$ reads: for all $x \in H^n$: 
\begin{equation}\label{eq:Tnffrfr}
T_n(x) = \sum_{i=1}^d \sigma'_i u'_i \langle f'_i , x \rangle,
\end{equation}
where $\sigma'_i > 0$, $(f'_1,\ldots,f'_d)$  and  $(u'_1,\ldots, u'_d)$ are  orthonormal basis of $H_0^n$ and $H^*_0$.

All ingredients are gathered to conclude. Recall that $\tilde X = (X_{a,b})$. From assumption $(H_0)$ and $\beta < 1$, 
$$
\dE \| \tilde X \|_H^2  = \sum_{(a,b)} \dE |X_{ab}|^2 \beta^{|a|+|b|} =  \sum_{(a,b)}   \beta^{|a|+|b|} < \infty.
$$
Hence, almost surely $\tilde X \in H$. It follows that 
$$
\cA \stackrel{d}{=} T(\tilde X) = T_n(\tilde X) + T \circ ( 1- \pi_n) (\tilde X).
$$
By construction $ T_n(\tilde X)$ and  $T \circ ( 1- \pi_n) (\tilde X)$ are independent variables whose supports are both included in $H_0^*$. Therefore the claim that the law of $\cA$ is absolutely continuous with respect to the Lebesgue measure on $H_0^*$ is implied by the claim that the law of $ T_n(X)$ is absolutely continuous with respect to the Lebesgue measure on $H_0^*$.

This is now easy to check. The real vector space $H^n $ has finite dimension $N =\dim_{\dR}(K) (2n+1)^2$. By assumption \Hac $X_n = \pi_n(\tilde X)$ is  absolutely continuous with respect to the Lebesgue measure on $H^n \simeq \dR^N$. We may complete $(f'_1,\ldots,f'_d)$ in an orthonormal basis  $(f'_1,\ldots, f'_N)$ of $H^n$. Let $U$ be the orthogonal matrix which maps the canonical basis of $H^n$ to $(f'_1,\ldots , f'_N)$.  By the change of variable formula, the law of the vector $U X_n =  (\langle f'_1,X_n \rangle, \ldots,  \langle f'_N,X_n  \rangle)$ is also absolutely continuous with respect to the Lebesgue measure on $H^n \simeq \dR^N$. In particular the law of $Y = (\langle f'_1,X_n \rangle, \ldots,  \langle f'_d,X_n \rangle)$ is absolutely continuous with respect to the Lebesgue measure on $H_0^n \simeq \dR^d$. Finally, from the singular value decomposition \eqref{eq:Tnffrfr}, the law of $T_n (\tilde X) = \sum_i \sigma'_i u_i Y_i$ is a absolutely continuous with respect to the Lebesgue measure on $H_0^*$ (since $\sigma'_i >0$ and $(u'_1,\ldots,u'_d)$ orthonormal basis). \qed


\section{Central limit theorem for random matrix products}

\label{sec:CLT}

 In this section, for ease of notation, we drop the explicit dependency in $n$ and write $X$ in place of $X_n$.

\subsection{A central limit theorem for random matrix products}

%

Let $n \geq 1$ be an integer and let $B = (B_0,B_1,\ldots ) \in M_n(\dC)^{\dN}$ be a sequence of deterministic matrices in $M_n(\dC)$ such that, for some vectors $u,v \in \dC^n$,
\begin{equation}\label{eq:defB0}
B_0 = u \otimes \bar v, 
\end{equation}
so that $B_0$ is a rank one projector. Writing $X$ in place of $X_n$, for integer $k \geq 1$, we are interested in the complex random variable
\begin{equation}\label{eq:defZ}
Z = Z_k(B) = n^{-(k-1)/2}  \TR( B_0 X B_1  X \cdots  B_{k-1} X ) =   n^{-(k-1)/2}   \langle v , X B_1  X \cdots  B_{k-1} X ,  u \rangle.
\end{equation}
The variable $Z$ is a generalized entry of the matrix product $X B_1  X \cdots  B_{k-1} X $. The goal of this subsection is prove that under mild assumptions, the process $(Z_k(B))$, indexed by $k\geq 1$ and $B \in M_n(\dC)^\dN$, is close in distribution to Gaussian process that we now define.

%
%
%

We introduce the following quantities, for $B,B' \in M_n(\dC)^{\dN}$, 
\begin{equation}\label{eq:defSigma}
\Sigma_{k}(B,B') = n^{-k+1} \prod_{t=0}^{k-1} \TR (B_t (B'_t)^*)    \quad \hbox{ and } \quad  \Sigma'_{k}(B,B') = \varrho^k  n^{-k+1} \prod_{t=0}^{k-1} \TR (B_t (B'_t)^\intercal).
\end{equation}
We note that for $B_0 = u \otimes \bar v$ and $B'_0 = u'\otimes \bar v'$ as in \eqref{eq:defB0}, we have 
$$
\TR (B_0 (B'_0)^*) = \langle v , u' \rangle \langle v' , u \rangle \quad \hbox{ and } \quad     \TR (B_0 (B'_0)^\intercal) = \langle \bar v , u' \rangle \langle \bar v' , u \rangle.
$$

We consider $W = W_k(B)$ as the centered Gaussian process, indexed by $k \geq 1$ and $B \in M_n(\dC)^{\dN}$ such that 
$ \dE W_k (B) \bar W_l (B')  = \dE W_k  W_l= 0 $ for $k \ne l$ and for $k = l$, 
$$
\dE W_k (B) \bar W_k (B') = \Sigma_k(B,B') \quad \hbox{ and } \quad     \dE W_k (B)  W_k (B') = \Sigma'_k(B,B').
$$
This process $W$ can be realized as follows: 
\begin{equation}\label{eq:defWk}
W_k (B)=  n^{-(k-1)/2} \sum_{\pi \in P_k} W_\pi  \prod_{t=0}^{k-1}  (B_t)_{i_{t} j_{t+1}},
\end{equation}
where $P_k$ is the set of $\pi = (j_1,i_1,\ldots, j_k, i_k)$ with $i_t,j_t$ in $\{1,\ldots ,n\}$, $i_0 = i_k$ and $(W_\pi)_{\pi \in P_k}$ are independent and identically distributed centered Gaussian random variables with $\dE | W_\pi |^2 = 1$ and $\dE W_\pi ^2 = \varrho^k$.

We define a new process $\tilde W = \tilde W_k(B)$ by setting
$$
\tilde W_1 (B) = Z_1 (B) \quad \hbox{ and } \quad \tilde W_k (B) = W_k (B) \hbox{ for } k \geq 2,
$$
where $W$ and $X$ are chosen to be independent. Note that $Z_1 (B) = \langle v ,  X u\rangle$ by definition, in particular, it is linear in $X$.


We denote by $\| \cdot \|$ the operator norm. We consider the following assumption on the sequence of matrices $B = (B_0,B_2,\ldots)$, for any integer $k \geq 0$,
\begin{equation}\label{eq:hypTCL}
\| B_k \| \lesssim 1, 
\end{equation}
where here and below the notation $\lesssim$ coincides with the $O(\cdot)$ notation and the constants hidden  may depend on  $k$ and the upper bound $C_k$ on the moments $\dE |X_{ij }|^k$ in Assumption $(H_k)$. Note that for $k= 0$, we have $\| B_0 \| = \| u\|_2 \| v \|_2$ where $\| \cdot \|_2$ is the Euclidean norm.

Since $\TR (M) \leq \mathrm{rank}(M) \| M \|$  for $M \in M_n(\dC)$, under assumption \eqref{eq:hypTCL}, we have 
$$
| \Sigma_k(B,B) | \lesssim 1 \quad  \hbox{ and } \quad | \Sigma'_k(B,B) | \lesssim 1 . 
$$
This explains the choice of the scaling for $Z$ and $W$.

The main result in this subsection is the following multivariate  central limit theorem theorem.

\begin{thm}\label{th:steinTCL}
Assume that $(H_k)$ holds for all $k \geq 2$. Let $T$ be a finite set and for $t \in T$, let $k_t \geq 1$ be an integer and $B_t = (B_{t,0},B_{t,1},\ldots) \in M_n(\dC)^\dN$ with $B_{t,0} =  u_{t} \otimes \bar v_{t}$ as in \eqref{eq:defB0} and such that,  for all $ t\in T$, \eqref{eq:hypTCL} holds for $B_t$, uniformly in $n$.

Then, the Lévy-Prokhorov distance between the laws of $(Z_{k_t}(B_t))_{t \in T}$ and $(\tilde W_{k_t}(B_t))_{t \in T}$ goes to $0$ as the $n$ goes to infinity.

Moreover, if  for all $t \in T$, $\| u_t\|_{\infty}$ and $\|v_t \|_{\infty}$ go to $0$ as $n$ goes to infinity, then the Lévy-Prokhorov distance between the laws of $(Z_{k_t}(B_t))_{t \in T}$ and $(W_{k_t}(B_t))_{t \in T}$ goes to $0$ as the $n$ goes to infinity. 
\end{thm}

The special role played by $k=1$ is easy to notice. For example, if $u = v = e_1$ is a vector the canonical basis then $Z_1 = X_{11}$ and we might not expect a central limit theorem (unless the law of $X_{ij}$ is Gaussian). However, as soon as $k \geq 2$, there is a central limit theorem. This central limit phenomenon can be related to the second-order freenees, see \cite{MR2216446}.

The proof of Theorem \ref{th:steinTCL} will rely on Stein's method of exchangeable pairs. In particular, a quantitative version of Theorem \ref{th:steinTCL} can be extracted from the proof (even if we made no effort to optimize the bounds on the various asymptotically negligible terms).

The remainder of this subsection is devoted to the proof of Theorem \ref{th:steinTCL}. Expanding the matrix product, in \eqref{eq:defZ} we find
\begin{equation}\label{eq:devZ}
Z   = n^{-(k-1)/2} \sum_{\pi \in P_k}  \prod_{t=0}^{k-1}  (B_t)_{i_{t} j_{t+1}}   \prod_{t=1}^{k} X_{j_t i_{t}} = \sum_{\pi \in P_k} M_\pi X_\pi ,
\end{equation}
where, as above, $P_k$ is the set of $\pi = (j_1,i_1,\ldots, j_k, i_k)$ with $i_t,j_t$ in $\{1,\ldots ,n\}$, $i_0 =i_k$ and we have set
$$
X_\pi = \prod_{t=1}^{k} X_{j_t i_{t}} \; \quad \hbox{   and  } \quad \;  M_\pi = n^{-(k-1)/2}    \prod_{t=0}^{k-1}  (B_t)_{i_{t} j_{t+1}}.
$$
The vector $\pi$ can be thought as a path on the set of indices. The complex number $M_\pi$ is the deterministic weight associated to the path $\pi$ while the complex number $X_\pi$ is the random  weight associated to $\pi$.  When we want to make the dependence in the variables explicit in $k,B$ we will write 
$
Z$ as $Z = Z_k (B)$ and $M_\pi$ as $M_\pi = M_\pi (B)$ for example.

We set 
\begin{equation}\label{eq:defZ1}
L = L_k(B) = \sum_{\pi \in Q_{k}} M_\pi(B) X_\pi ,
\end{equation}
where $Q_k \subset P_k$ is the subset of paths $\pi = (j_1,i_1, \ldots, j_k,i_k)$ such that for all $s \ne t$  $(j_t,i_t) \ne (j_s,i_s)$.  Our first lemma asserts that $Z$ and $T$ are close in $L^2(\dP)$. The main interest of $L$ is that it is a linear function of the matrix entry $X_{ij}$ for any $i,j$. Obviously $L_1 = Z_1$. For $k \geq 2$, the next lemma asserts that $Z_k$ and $L_k$ remain close.

\begin{lem}\label{le:ZZ1}
For any $k \geq 1$ and $B = (B_0,B_1,B_2,\ldots)$ satisfying \eqref{eq:defB0} and \eqref{eq:hypTCL}, we have
$$
\dE |Z - L|^2 \lesssim n^{-1}.  
$$
\end{lem}

\begin{proof} 
Set $$\hat Z = Z - L = \sum_{\pi \in R_k} M_\pi X_\pi,$$ with $R_k = P_k \backslash Q_k$. The statement of the lemma follows from the two claims 
\begin{equation}\label{eq:ZZ1ES}
| \dE \hat Z | \lesssim n^{-1/2}
\end{equation}
and 
\begin{equation}\label{eq:ZZ1VARS}
\dE | \hat Z - \dE \hat Z |^2  \lesssim n^{-1}.
\end{equation}

We start with \eqref{eq:ZZ1ES}.  Note  that by assumption $(H_k)$, we have $ \dE |  X_\pi| \lesssim 1$. We write 
$$
\dE \hat Z = \sum_{\pi \in R_k}  M_\pi  \dE   X_\pi.
$$

To each $\pi = (j_1,\ldots,i_k)$ in $P_k$, we associate a bipartite directed (multi)-graph $G(\pi)$ with bipartite vertex sets $I = \{i_t : 1 \leq t \leq k\}$ and $J = \{ j_t : 1 \leq t \leq k\}$ and edge sets $E_{JI} = \{ (j_t,i_t) : 1 \leq t  \leq k\}$ and $E_{IJ} = \{ (i_{t-1},j_{t}) : 1 \leq t  \leq k\}$ with the convention $i_0 = i_k$. The multiplicity of an edge $(j,i) \in E_{JI}$ is the number of $t$ such that $(j,i) = (j_t,i_t)$. Since $\dE X_{ji} = 0$, $\dE X_\pi = 0$ unless all edges in $E_{JI}$ have multiplicity at least two. We denote by $\bar P_k$ the subset of such paths.

We say that two paths $\pi = (j_1,\ldots,i_k)$, $\pi' = (j'_1, \ldots,j'_k)$ are isomorphic and write $\pi \sim \pi'$  if there exist two bijections $\phi, \psi$ on $\{1,\ldots,n\}$ such that $\phi(j_t) = j'_t$ and $\psi(i_t) = i'_t$. This is an equivalence relation.  By construction $G(\pi)$ and $G(\pi')$ are isomorphic if $\pi \sim \pi'$. Also, by assumption $\dE X_\pi = \dE X_{\pi'}$ if $\pi \sim \pi'$. Note that there is a finite number of equivalence classes ($k^{2k}$ is a very rough upper bound). To prove \eqref{eq:ZZ1ES}, it follows that it is enough to prove that for any $\pi \in \bar P_k$,
\begin{equation}\label{ZZ1ES1}
m_1(\pi) =  \ABS{ \sum_{\pi' : \pi' \sim \pi} M_{\pi'} }   \lesssim n^{-1/2}.
\end{equation}
We fix such $\pi = (j_{1},i_{1},\ldots,j_{k},i_{k}) \in \bar P_k$ and denote by $I,J, E_{IJ}, E_{JI}$ the vertex and edge sets of $G(\pi)$. Writing $\pi' = (j'_1,i'_1,\ldots,j'_k, i'_k)$, we have 
\begin{equation}\label{ZZ1ES2}
m_1(\pi) =   n^{-(k-1)/2}  \ABS{   
\sum_{\pi' : \pi' \sim \pi}  \prod_{t=0}^{k-1}  (B_t)_{i'_{t} j'_{t+1}} }.
\end{equation}

Since each edge, say $(j,i)$, in $E_{JI}$ has multiplicity at least $2$, $i$ and $j$ are visited at least twice (that is $i_{t} = i$ for at least two distinct $1 \leq t \leq k$ and similarly for $j$). In particular, $|I| \leq k/2$ and $|J| \leq k/2$. We define the in-degree of $j \in J$ as the number of distinct $0 \leq t \leq k-1$ such that $j_{t+1} = j$ (that is $(i_{t},j_{t+1}) \in E_{IJ}$ is an incoming edge). From what precedes, all $j \in J$ have in-degree at least two. 

We start with a rough upper bound. We claim that 
\begin{equation}\label{eq:m1bd1}
m_1 (\pi)  \lesssim n^{-(k-1)/2} n^{|I| \wedge |J|-1/2}.
\end{equation}

Let us prove this claim.  We start with the following observation. The involution map on $P_k$, $\theta :   (j_1, i_1,\ldots, j_k,i_k) \to (i_k, j_k,\ldots, i_1,j_1)$ inverses the role played by $I$ and $J$. Moreover, we have $m_{1,B} (\theta(\pi) ) = m_{1,\check B} ( \pi)$ where we have written the dependence of $m_1$ on $B$ explicitly and set for $B = (B_0,B_1,B_2 \ldots)\in M_n(\dC)^{\dN}$, $\check B = (B_0^\intercal, B_{k-1}^\intercal,\ldots,B_1^\intercal, B_k, \ldots)$.  Since that $M$ and $M^\intercal$ have the same norm, it is sufficient to prove that \eqref{eq:m1bd1} holds with only $|I|$  in the exponent.

We use the convention that a product over an emptyset is equal to $1$. For any $\iota_1,\ldots,\iota_p$ in $\{1,\ldots, n\}$ with $p \geq 2$ and $M_1,\ldots, M_p$ in $M_n(\dC)$, from Cauchy-Scwharz inequality, we have the upper bound,
\begin{align}
 \sum_{j'=1}^n  \prod_{q=1}^p\ABS{ (M_q)_{\iota_q j'}}&  \leq  
 \prod_{q=3}^p \| M_q \| \sqrt{ \sum_{j'=1}^n  | (M_{1})_{\iota_1 j'} |^2 } \sqrt{\sum_{j'=1}^n  | (M_{2})_{\iota_2 j'} |^2 } \nonumber \\
 & = \prod_{q=3}^p \| M_q \|  \sqrt { (M_1 M_1)^*_{\iota_1 \iota_1} } \sqrt { (M_2 M_2)^*_{\iota_2 \iota_2} }
  \leq    \prod_{q=1}^p \| M_q \| .\label{eq:boundCp}
\end{align}

To estimate $m_1(\pi)$ in \eqref{ZZ1ES2},  we sum over all possible bijections, that is over all possible values of $(j'_1,\ldots,i'_k)$.   As already pointed, all $J$-vertices have in-degree at least $2$. Then, we put the absolute value in \eqref{ZZ1ES2} inside the sum. We  use \eqref{eq:boundCp}  for each $j \in J$ when summing over all possible images say $j'$. Here, in \eqref{eq:boundCp}, $p \geq 2$ is the in-degree of $j \in J$ in $G(\pi)$, the matrices $M_q$ are equal to $B_{t_q}$ for some $0 \leq t_q \leq k-1$ and $\iota_q = i'_{t_q}$ (recall $i'_0= i'_k$). There is at most $n$ possible images, say $i'$, of any $i \in I$, we deduce that 
$$
m_1 (\pi) \leq \PAR{ \prod_{t=0}^{k-1} \|B_{t}\| } n^{-(k-1)/2} n^{|I|}.
$$
It does not quite prove \eqref{eq:m1bd1} due to a missing factor $n^{-1/2}$. To obtain this improvement, we use that $B_0$ is a rank one projector. More precisely, we use the first expression in \eqref{eq:boundCp} when $j = j_1\in J$ and pick $\iota_1 = i'_{0} = i'_k$,  $M_1 = B_0$. Then, since $\TR( B_0 B_0^* ) = \| u \|_2 \| v \|_2 = \| B_0 \|^2$, we deduce from Cauchy-Schwarz inequality,
$$
\sum_{i'=1}^n \sqrt { (B_0 B^*_0)_{i'i'} } \leq \sqrt n \sqrt { \TR (B_0 B^*_0) } \leq \| B_0\| \sqrt n.
$$
Hence, when summing over the possible images of $I$-vertices, we have improved the bound from $n^{|I|}$ to $n^{|I|-1/2}$.  This proves \eqref{eq:m1bd1}.

The bound \eqref{eq:m1bd1} proves \eqref{ZZ1ES1} as soon as $ |I| \wedge |J| \leq k/2-1$. Since $|I| \vee |J| \leq k/2$, the remaining case to consider is: $k$ even and $|I| = |J|= k/2$ 

We first remark that when summing over all $\pi' \sim \pi$ that is over all possible values of $i'_{t}$, $j'_{t}$, we may remove the constraints that $i'_{t} \ne i'_{t'}$ if $i_{t} \ne i_{t'}$. Indeed, the case where this constraint is not satisfied corresponds to $i'_{t} = i'_{t'}$ and can thus be written in absolute value as $m_1(\tilde \pi)$ where $\tilde \pi$ is obtained from $\pi$ by identifying  $i_{t}$ and $i_{t'}$. In particular, there are at most $|I| - 1$ vertices in $I(\tilde \pi)$ and we  can apply \eqref{eq:m1bd1} to neglect this contribution. The same argument works for $J$-vertices. We denote by $\hat m_1(\pi)$, the expression of $m_1(\pi)$ when the constraints $i_{t} \ne i_{\veps',t'}$ if $i_{t} \ne i_{\veps',t'}$ have been removed, and similarly for $J$-vertices. From what precedes, to prove \eqref{ZZ1ES1}, it suffices to check $\hat m_1(\pi) \lesssim n^{-1/2}$. We have the following expression for $\hat m_1(\pi)$,
$$
\hat m_1 (\pi) = n^{-(k-1)/2} \ABS{ \sum_{i'_1,\ldots,i'_{k/2} , j'_1,\ldots,j'_{k/2}} \prod_{t=0}^{k-1} (B_{t})_{i'_{i_{t-1}},j'_{j_{t}}} },
$$
where the sum is over all $1 \leq i'_s , j'_s \leq n$, for all $1 \leq s \leq k/2$.

Consider $G_{IJ}$ as the bipartite (multi)-graph on the disjoint union $I \sqcup J$ with edge (multi)-set $E_{IJ}$ (this is a slight abuse of notation $E_{IJ}$ is treated here as a multiset with $k$ elements). The assumption $2|I| = 2|J| = k$ implies that all  vertices are of degree  $2$ in $G_{IJ}$. In particular, $G_{IJ}$ is a disjoint union of alternating cycles (where a cycle a length $2$ is composed of two distinct edges $(i,j)$ from $i$ to $j$).

Pick an element $i \in I$. It appears in exactly two factors in the above expression for $\hat m_1(\pi)$, say $t=t_1$ and $t=t_2$. We then use the elementary identity 
$$
\sum_{i'} (M_{1})_{i'k_1} (M_2)_{i'k_2} = (M^\intercal_1 M_2)_{k_1k_2}.
$$
with $M_\veps = B_{t_\veps}$ and $k_{\veps} = j'_{j_{t_\veps}}$, $\veps = 1,2$. Repeating this for all $I$ and $J$-vertices, we deduce that 
$$
\hat m_1 (\pi) = n^{-(k-1)/2} \prod_{l=1}^c \ABS{ \TR ( A_l)},
$$
where $c$ is the number of connected components of $G_{IJ}$ and $A_l$ is a product of the matrices $(B^{\sigma_t}_t)$, $t \in p_l$, where $p = (p_1,\ldots,p_c)$ is a partition of $\{0,\ldots,k-1\}$ into $c$ blocks and $B_t^{\sigma_t}\in \{ B_t , B^\intercal_{t}\}$. Since $\| B_t\| \lesssim 1 $, we have $\| A_l \| \lesssim 1$ and thus $| \TR(A_l)| \lesssim n$. However, if $0 \in p_l$, then, since $B_0 = u \otimes \bar v$, $\TR ( M_1 B_0 M_2 ) = \langle v , M_2 M_1 u \rangle$ and we find 
$$
\TR ( A_l)= \langle v , A'_l  u \rangle
$$
where $A'_l$ is a product of the matrices $(B^{\veps_t}_t)$, $t \in p_l \backslash \{0\}$. Hence,  $|\TR ( A_l)| \lesssim 1$ if $0 \in p_l$.  We have thus checked that 
$$
\hat m_1 (\pi) \lesssim n^{-(k-1)/2}  n ^{c -1} \lesssim n^{-(k-1)/2}  n ^{|I| -1}. 
$$
This concludes the proof of \eqref{ZZ1ES1} and \eqref{eq:ZZ1ES}.

We now prove \eqref{eq:ZZ1VARS}. The proof follows exactly the same lines. It ultimately relies on \eqref{eq:boundCp} and some topological facts. We write
\begin{eqnarray}
\dE |\hat Z - \dE \hat Z |^2   & = & \sum_{\pi_1,\pi_2 \in R_k}  M_{\pi_1} \bar M_{\pi_2} \dE [ (X_{\pi_1} - \dE X_{\pi_1} )\overline{(X_{\pi_2} - \dE X_{\pi_2} )}].\label{eq:vardedef}
\end{eqnarray}
With the above notation, to each $\pi_1 = (j_{1,1},\ldots,i_{1,k})$, $\pi_2 = (j_{2,1},\ldots, i_{2,k})$ in $P_k$, we associate the bipartite directed graph $G(\pi_1,\pi_2) = G(\pi_1) \cup G(\pi_2)$ with vertex sets $I = I(\pi_1) \cup I(\pi_2)$, $J = J(\pi_1) \cup J(\pi_2)$ and edge sets $E_{IJ} = E_{IJ} (\pi_1) \cup E_{IJ}(\pi_2)$, $E_{JI} = E_{JI} (\pi_1) \cup E_{JI}(\pi_2)$.  Due to the independence of the entries of $X$, the  expectation in \eqref{eq:vardedef} is zero unless $E_{JI}(\pi_1)$ and $E_{JI}(\pi_2)$ have a non-empty intersection and each edge in $E_{JI}$ has multiplicity at least two. We denote by $\bar P_{k,k} \subset P_k \times P_k$ the collection of pairs of paths $\pi = (\pi_1,\pi_2)$ such that each edge in $E_{JI}$ is visited twice. We set $\bar R_{k,k} = (R_k \times R_k) \cap \bar P_{k,k}$.

As above, we introduce an equivalence relation on $P_k \times P_k$: $\pi = (\pi_1,\pi_2) \sim \pi' = (\pi'_1,\pi'_2)$ if there exist two bijections $\phi, \psi$ on $\{1,\ldots,n\}$ such that for $\veps = 1,2$, $t = 1, \ldots,k$, $\phi(j_{\veps,t}) = j'_{\veps,t}$,  $\psi(i_{\veps,t}) = i'_{\veps,t}$ with $\pi^\delta_\veps = ( j^\delta_{\veps,1}, \ldots, i^\delta_{\veps,k})$. There is a finite number of equivalence classes. To prove \eqref{eq:ZZ1VARS}, it is thus sufficient to prove that for any fixed $\pi =  (\pi_1,\pi_2) \in  \bar R_{k,k}$, 
\begin{equation}\label{ZZ1VARS1}
m_2 (\pi ) =  \ABS{ \sum_{ (\pi'_1,\pi'_2) \sim (\pi_1,\pi_2) }  M_{\pi'_1} \bar M_{\pi'_2} } \lesssim n^{-1},
\end{equation}
where we have used that the expectation in \eqref{eq:vardedef} is uniformly upper bounded by some constant depending on $k$ and depends only the equivalence class.  Writing, for $\veps =1,2$, $\pi_\veps = (j_{\veps,1},\ldots,i_{\veps,k})$, we get \begin{equation*}\label{ZZ1VARS2}
m_2(\pi ) = 
 n^{-k+1} \ABS{ \sum_{ \pi' \sim \pi }  \prod_{t=0}^{k-1}  (B_t)_{i'_{1,t} j'_{1,t+1}} (\bar B_t)_{i'_{2,t} j'_{2,t+1}}  }.
\end{equation*}

We first claim that 
\begin{equation}\label{eq:mpiI1}
 m_2(\pi) \leq \PAR{ \prod_{t=0}^{k-1} \| B_t \|^2 } n^{-k+1} n^{|I|\wedge |J|-1}
  \lesssim n^{-k+1} n^{|I|\wedge |J|-1}. 
\end{equation}
Indeed, we argue as in the proof of \eqref{eq:m1bd1}. Thanks to the involution map $\theta$, it suffices to prove that $m_2(\pi) \lesssim n^{-k+1} n^{|I|-1}$. Then we use \eqref{eq:boundCp} for each $j \in J$. If $j_{1,1} \ne j_{2,1}$, when summing over $j_{\veps,1}$, we set in the first expression in  \eqref{eq:boundCp}, $M_1 = B_0$, $\iota_1 =  i'_{\veps,k}$ ($i_{1,k}$ and $i_{2,k}$ could be equal or not). Similarly, if $j_{1,1} = j_{2,1}$, we set in the first expression in \eqref{eq:boundCp}, $M_1 = M_2 = B_0$, $\iota_\veps = i'_{\veps,k}$ for $\veps = 1,2$. Hence in both cases, there is factor $( (B_0B^*_0)_{i'_{1,k} i'_{1,k}} (B_0B^*_0)_{i'_{2,k} i'_{2,k}} )^{1/2} $. As in the proof of \eqref{eq:m1bd1}, we sum over all possible images of $I$-vertices and use that $B_0$ is a rank one projector, it gives \eqref{eq:mpiI1} (in both cases  $i_{1,k} = i_{2,k}$ and $i_{1,k} \ne i_{2,k}$).

We consider two cases in the proof of \eqref{ZZ1VARS1}.

{\em Case 1 : $\pi_1 = \pi_2$. } Since $\pi_1 \in R_k$, there exists at least one edge $JI$ of multiplicity at least $2$ in $E_{JI}(\pi_1)$ and thus of multiplicity at least $4$ in $E_{JI}$. All other edges in $E_{JI}$ have multiplicity at least $2$. The sum of all edge multiplicities in $E_{JI}$ is $2k$. It follows that the total number of distinct edges in $E_{JI}$ is at most $k-1$. In particular $|I| \vee |J|  \leq k-1$. Thanks to \eqref{eq:mpiI1}, this proves that \eqref{ZZ1VARS1} when $\pi_1 = \pi_2$.

{\em Case 2 : $\pi_1 \ne \pi_2$. } Vertices in $J$ have out-degree at least $2$ since each edge in $E_{JI}$ has multiplicity at least two. The sum of out-degrees of $J$-vertices  is $2k$. It follows that $|J| \leq k$. Similarly $|I| \leq k$. If $|I| \wedge |J| \leq k-1$, then thanks to \eqref{eq:mpiI1}, \eqref{ZZ1VARS1} holds. The last case to consider is thus $|I| = |J| = k$.

As in the proof of \eqref{eq:m1bd1},  we remark also that when summing over all $\pi' = (\pi'_1,\pi'_2)$ in the expression of $m_2(\pi)$, we may remove the constraints that $i'_{\veps,t} \ne i'_{\veps',t'}$ if $i_{\veps,t} \ne i_{\veps',t'}$. Indeed, the case where this constraint is not satisfied can be reduced to cases of $\tilde \pi \in \bar R_{k,k}$ with $| I(\tilde \pi) | < | I(\pi) |$ or $| J(\tilde \pi) | < |J(\pi)|$ which can be directly bounded thanks to \eqref{eq:mpiI1}.  We denote by $\hat m_2(\pi)$, the expression of $m_2(\pi)$ when the constraints $i_{\veps,t} \ne i_{\veps',t'}$ if $i_{\veps,t} \ne i_{\veps',t'}$ have been removed. From what precedes, it suffices to check $\hat m_2(\pi) \lesssim n^{-1}$.

As above, we consider the graph $G_{IJ}$  on $I \sqcup J$ spanned by the edges $E_{IJ}$. Since $|I| = |J| = k$, all vertices have degree exactly two in $G_{IJ}$. As above, by construction, $\hat m_2(\pi)$ is equal to a product of contributions on each connected component. Arguing as in the proof of \eqref{eq:m1bd1}, we get that 
$$
\hat m_2 (\pi) = n^{-k+1} \prod_{l=1}^c \ABS{ \TR ( A_l)},
$$
where $c$ is the number of connected components of $G_{IJ}$ and $A_l$ is a product of the matrices $(C_{\veps,t})$, $(\veps,t) \in p_l$, where $p = (p_1,\ldots,p_c)$ is a partition of $\{1,2\} \times \{0,\ldots,k-1\}$ into $c$ blocks and $C_{\veps,t} \in \{ B_{t} , B^\intercal_{t}\}$. Using that $B_0$ is a rank one projector, we deduce that 
\begin{equation}\label{eq:mpifinal}
\hat m_2(\pi) \lesssim n^{-k+1} n ^{c'},
\end{equation}
where $c'\leq c-1$ is the number of connected components of $G_{JI}$ not containing  a boundary vertex $j_{\veps,1}, i_{\veps_k}$.  The proof of \eqref{ZZ1VARS1} will thus be complete if we check that $c \leq  k-1$. Since $|I| = |J| = k$, it is thus sufficient to check that there exists a connected component with two distinct $IJ$-edges (since each connected component contains at least one $J$-vertex and $I$-vertex).

Now, we use the assumption that $\pi_1 \ne \pi_2$ but $E_{JI}(\pi_1) \cap E_{JI}(\pi_2)$ non-empty. It implies that there exists $v_* \in I \cup J$  which is in $\pi_1 \cap \pi_2$ but the neighboring edges in $G(\pi_1)$ and $G(\pi_2)$ are not all equal. Then, since in-degrees and out-degrees are exactly two, it leaves only two possibilities.  If $v_* = i_* \in I$ then for some $t_1,t_2$: $j_{1,t_1} = j_{2,t_2}$, $i_* = i_{1,t_1}  = i_{2,t_2}$ and $j_{1,t_1+1} \ne  j_{2,t_2+1}$. If $v_* = j_* \in J$, $i_{1,t_1-1} \ne i_{2,t_2-1}$, $j_* = j_{1,t_1}  = j_{2,t_2}$ and $i_{1,t_1} =  i_{2,t_2}$. Indeed, only the edge $IJ$ may differ: if two neighboring $JI$-edges of $v_\star$ differ, then the out and in-degree of $v_*$ would be at least $4$ (since each $JI$-edge has multiplicity at least $2$). We have thus found a vertex with two distinct adjacent $IJ$-edges. It  concludes the proof of \eqref{eq:ZZ1VARS}.
\end{proof}

Our next lemma computes the asymptotic covariance of $L = L_k (B)$. For integer $k \geq 1$. For $B,B'$ in $M_{N}(\dC)^\dN$, we set 
\begin{equation}\label{eq:defhatSigma}
\hat \Sigma_{k}(B,B') = \sum_{\pi \in Q_k} M_\pi (B)  \bar M_\pi (B') \quad \hbox{ and } \quad  \hat \Sigma_{k}'(B,B') = \varrho^k \sum_{\pi \in Q_k} M_\pi (B)  M_\pi (B'),
\end{equation}
where we have written explicitly the dependence of $M_\pi$ in $B$.
The following lemma asserts that $\hat \Sigma$ and $\hat \Sigma'$ is the asymptotic covariance of $L$. 
\begin{lem}
Let $k,l \geq 1$ be integers  and $B,B'$ in $M_{N}(\dC)^{\dN}$ satisfying assumptions \eqref{eq:defB0}-\eqref{eq:hypTCL}. Set $ L = L_k(B)$ and $ L' = L_l (B')$. We have $\dE L = 0$. Moreover, if $k \ne l$, $\dE L  \bar L' = \dE L   L'=  0$ while if $k = l$,
$$
\dE L   \bar L' = \hat \Sigma_k(B,B') + o(1) \quad \hbox{ and } \quad \dE L   L' = \hat \Sigma'_k (B,B') + o(1), 
$$
where $o(1)$ denotes a vanishing sequence  as $n\to \infty$ depending on $k,l$ and the upper bound on the operator norms of $B_t$ and $B'_t$, $0 \leq t \leq k-1$.
\end{lem}

\begin{proof}
The claim $\dE L = 0$ follows from $\dE X_\pi = 0$ for all $\pi \in Q_k$ which is an immediate consequence of $\dE X_{ij} = 0$ and the definition of $Q_k$. For the second claim, we write
$$
\dE L   \bar L' =  \sum_{\pi \in Q_{k}, \pi' \in Q_l} M_\pi (B) \bar M_{\pi'} (B') \dE [ X_\pi \bar X_{\pi'}].
$$
Assume without loss of generality $k > l$. Recall $\dE X_{ij} = 0$ and each $(j_t, i_t)$ appears once in $\pi = (j_1,i_1, \ldots, j_k,i_k) \in Q_k$. It follows $\dE [ X_\pi \bar X_{\pi'}] $ will be zero unless the $k$ distinct $(j_t,i_t)$'s appear also $\pi'$. This is impossible since $k > l$.  The same argument applies to $\dE L L'$.  We skip the proof of the claim for $k= l$: it will be a straightforward consequence of the forthcoming Lemma \ref{le:stein1}, Lemma \ref{le:stein2} and \eqref{eq:ESk}.
\end{proof}

Our next lemma checks that $\hat \Sigma_k$ and $\Sigma_k$ defined respectively by \eqref{eq:defhatSigma} and \eqref{eq:defSigma} are close. 

\begin{lem}\label{le:hatSigma}
Let $k\geq 1$ be and integer  and $B,B'$ in $M_{N}(\dC)^{\dN}$ satisfying assumptions \eqref{eq:defB0}-\eqref{eq:hypTCL}. Then 
$$
\ABS{\Sigma_k(B,B') - \hat \Sigma_k(B,B') } \lesssim n^{-1} \quad \hbox{ and } \quad \ABS{\Sigma'_k(B,B') - \hat \Sigma'_k(B,B') } \lesssim n^{-1}.
$$
\end{lem}
\begin{proof}
By polarization, it suffices to check the statement for $B = B'$. We have 
$$
\Sigma_k = \sum_{\pi \in P_k} M_\pi  \bar M_\pi  \quad \hbox{ and } \quad \Sigma_k = \rho^k \sum_{\pi \in P_k} M_\pi M_\pi .
$$
Therefore, 
$$
\ABS{ \Sigma_k - \Sigma_{k} }  = \ABS{ \sum_{\pi \in P_k \backslash Q_k } M_\pi  \bar M_\pi }.
$$
Using the notation of Lemma  \ref{le:ZZ1}, we decompose $R_k = P_k \backslash Q_k$ over equivalence classes.  Then \eqref{eq:mpiI1} applied to $m_2( \pi,\pi)$ with $\pi \in R_k$ implies $\ABS{ \Sigma_k - \hat \Sigma_{k} } \lesssim n^{-1}$. 
A similar bounds holds for $\Sigma'_k - \hat \Sigma'_k$.
\end{proof}


We now use Stein's method of exchangeable pairs. Let $\tilde X = (\tilde X_{ij})$ be an independent copy of $X$. For integers $1 \leq p,q \leq n$, we define $X^{pq}$ by setting for all $1 \leq i,j \leq n$, 
$$
X^{pq}_{ij} = \left\{ \begin{array}{ll}
X_{ij} & \hbox{ if $(i,j) \ne (p,q)$} \\
\tilde X_{pq} & \hbox{ if $(i,j) = (p,q)$}
\end{array}\right. .
$$
Finally, we set $\hat X = X^{pq}$ where $(p,q)$ is uniformly distributed on $\{1,\ldots,n\}^2$. By construction $(X,\hat X)$ and $(\hat X ,X)$ have the same distribution: they form an exchangeable pair. We define $L^{pq} = L^{pq}_k (B)$ and $\hat L = \hat L_{k} (B)$ as the random variables defined as $L$ in \eqref{eq:defZ1} with $X$ replaced by $X^{pq}$ and $\hat X$ respectively. 

\begin{lem}\label{le:stein1}
For $L = L_k (B)$ and $\hat L = \hat L_k (B)$ as above, we have 
$$
\dE[ \hat L - L | X] = -\frac{k}{n^2} L.
$$
\end{lem}

\begin{proof}
By construction 
$$
\dE[ \hat L - L | X] = \frac{1}{n^2} \sum_{p,q} \dE[ L^{pq} - L | X] = \frac{1}{n^2} \sum_{\pi \in Q_k} M_\pi \sum_{p,q}  \dE [ X^{pq} _{\pi}   - X_\pi|X].
$$
Fix $\pi = (j_1,i_1, \ldots, j_k,i_k) \in Q_k$. Observe that $X^{pq} _{\pi}  =  X_\pi$ if $(p,q) \notin \{ (j_t,i_t) : 1 \leq t \leq k\}$. Otherwise, we have $X^{pq}_\pi = \tilde X_{pq} X_{\pi \backslash pq}$ where we have set $X_{\pi \backslash pq } = \prod_{ t : j_t i_t \ne pq } X_{j_t i_t}$. Hence $\dE [ X^{pq} _{\pi} | X ] =  0$ since, by the definition of $Q_k$, $(p,q)$ is then equal to exactly one $(j_t,i_t)$. We deduce that 
$$
\sum_{p,q} \dE [ X^{pq}_{\pi}   - X_\pi|X] = - \sum_{t=1}^k X_\pi = -k X_\pi. 
$$
The conclusion follows.
\end{proof}

For the next lemma, we introduce a new quantity which can be interpreted as the asymptotic quadratic variation of $\hat L - L$. More precisely, for $B,B'$ in $M_{N}(\dC)^\dN$, we set 
\begin{eqnarray*}
S_{k}(B,B') & = & \sum_{\pi \in Q_k} M_\pi (B)  \bar M_\pi (B')  \sum_{l=1}^k (|X_{\pi \backslash (j_li_l)}|^2 + |X_\pi|^2) \\
S_{k}'(B,B') & =  & \sum_{\pi \in Q_k} M_\pi (B) M_\pi (B')  \sum_{l=1}^k (\rho  X^2_{\pi \backslash (j_li_l)} + X^2_{\pi}) ,
\end{eqnarray*}
where we have set $X_{\pi \backslash j_li_l } = \prod_{ t : j_t i_t \ne j_li_l  } X_{j_t i_t} = \prod_{ t \ne l } X_{j_t i_t}$. We observe that $S_k$ and $S_k'$ are measurable with respect to $\sigma$-algebra generated by $X$ and that, by assumption $(H_0)$,
\begin{equation}\label{eq:ESk}
\dE S_k =2  \hat \Sigma_k \quad \hbox{ and } \quad \dE S'_k = 2 \hat \Sigma'_k. 
\end{equation}

\begin{lem}\label{le:stein2}
Let $k,l \geq 1$ be integers, and $B,B'$ in $M^{\dN}_{N}(\dC)$ satisfying assumptions \eqref{eq:defB0}-\eqref{eq:hypTCL}. Set $ L = L_k(B$, $\hat L = \hat L_k (B)$, $ L' = L_l (B')$, $\hat L' = \hat L_l (B')$, $S_k = S_k (B,B')$, $S'_k =S_k (B,B')$.  We have
$$
\dE[ (\hat L - L)  \overline{(\hat L' - L')} | X]  = \delta_{kl} \frac{S_k}{n^2}  +  \frac{E}{n^2} \quad \hbox{ and } \quad\dE[ (\hat L - L)  {(\hat L' - L')}  | X] = \delta_{kl} \frac{S'_k}{n^2} + \frac{ E'}{n^2},
$$
where the random variables $E, E'$ satisfy $\dE |E|^2 \lesssim n^{-1}$ and $\dE |E'|^2 \lesssim n^{-1}$.
\end{lem}
\begin{proof}We give the proof for $(\hat L - L)  {(\hat L' - L')}$, the case of $(\hat L - L)  \overline{(\hat L' - L')}$ being identical. 
For $\pi \in \sqcup_k Q_k$, we set $M_\pi= M_\pi(B)$ and $M'_{\pi} = M_\pi (B')$. We write 
\begin{eqnarray*}
\dE[ (\hat L - L)  {(\hat L' - L')}  | X] &=& \frac 1 {n ^2}  \sum_{p,q} \dE[ (\hat L^{pq}- L)   {(\hat {L'}^{pq} - L')}  | X] \\
& = & \frac 1 {n ^2}  \sum_{\pi \in Q_k , \pi' \in Q_l}  M_\pi  M'_{\pi'} \sum_{p,q} \dE [ (X^{pq} _{\pi}   - X_\pi)(X^{pq} _{\pi'}   - X_{\pi'}) |X].
\end{eqnarray*}
Observe that $X^{pq} _{\pi}  =  X_\pi$ if $(p,q) \notin \{ (j_t,i_t) : 1 \leq t \leq k\}$, say $(p,q) \notin \pi$ for short. Otherwise, $(p,q) \in \pi$ and we have $X^{pq}_\pi = \tilde X_{pq} X_{\pi \backslash pq}$ where we have set $X_{\pi \backslash pq } = \prod_{ t : j_t i_t \ne pq } X_{j_t i_t}$. By the definition of $Q_k$, $Q_l$ and Assumption $(H_k)$ we get 
$\dE [ (X^{pq} _{\pi}   - X_\pi)(X^{pq} _{\pi'}   - X_{\pi'}) |X] = \rho X_{\pi \backslash pq}X_{\pi' \backslash pq} + X_\pi X_{\pi'}$ if $(p,q) \in \pi \cap \pi'$ and otherwise $\dE [ (X^{pq} _{\pi}   - X_\pi)(X^{pq} _{\pi'}   - X_{\pi'}) |X] = 0$.

We deduce that 
\begin{eqnarray*}\label{eq:SI22}
\dE[ (\hat L - L)   {(\hat L' - L')}  | X] & = &  \frac 1 {n ^2}\sum_{\pi_1 \in Q_k , \pi_2 \in Q_l}  M_{\pi_1}   M'_{\pi_2} \sum_{(p,q) \in \pi_1 \cap \pi_2}  \left( \rho X_{\pi_1 \backslash pq}X_{\pi_2 \backslash pq}  + X_{\pi_1} X_{\pi_2}\right) \\
& =& \delta_{kl} \frac{S'_k}{n^2} + \frac{E'_1 +E'_2}{n^2},
\end{eqnarray*}
with \begin{eqnarray*}
E'_1 & =&  \sum_{\pi_1 \in Q_k, \pi_2 \in Q_l , \pi_1 \ne \pi_2}  M_{\pi_1} M'_{\pi_2} |  \pi_1 \cap \pi_2 |  X_{\pi_1} X_{\pi_2},\\ E'_2 & = & \rho \sum_{\pi_1 \in Q_k, \pi_2 \in Q_l , \pi_1 \ne \pi_2}  M_{\pi_1} M'_{\pi_2} \sum_{(p,q) \in \pi_1 \cap \pi_2}  X_{\pi_1 \backslash pq } X_{\pi_2\backslash pq }.
 \end{eqnarray*}

We should prove that for $i =1,2$, $\dE |E'_i|^2 \lesssim n^{-1}$. We prove only the case $i =1$, the case $i=2$ being   identical. We may assume $k \geq 2$ (if $k=1$, $E'_i = 0$). To that end,  we start by expanding $|E'_1|^2 $:
$$
\dE  |E'_1|^2 = \sum  M_{\pi_1} M'_{\pi_2} \bar M_{\pi_3} \bar M'_{\pi_4} |  \pi_1 \cap \pi_2| |  \pi_3 \cap \pi_4 |  \dE \left[ X_{\pi_1} X_{\pi_2}  \bar X_{\pi_3} \bar X_{\pi_4}  \right],
$$
where the sum is over $4$-tuple $\pi = (\pi_1,\pi_2,\pi_3,\pi_4 ) \in Q_{k} \times Q_l \times Q_k \times Q_k$, $\pi_1 \ne \pi_2$, $\pi_3 \ne \pi_4$.
 To $4$-tuple $\pi = (\pi_1,\pi_2,\pi_3,\pi_4 ) \in Q_{k} \times Q_l \times Q_k \times Q_k$, we may associate a bipartite directed graph $G(\pi) = ( I, J, E_{IJ}, E_{JI})$ which is the union of the graphs $G(\pi_\veps), \veps = 1,\ldots,4$ defined in Lemma \ref{le:ZZ1}. We consider the subset $\bar Q_{k,l,k,l}$ of $Q_{k} \times Q_l \times Q_k \times Q_k$ such that $\pi_1 \ne \pi_2$, $\pi_3 \ne \pi_4$, $E_{JI} (\pi_1) \cap E_{JI} (\pi_2) \ne \emptyset$, $E_{JI} (\pi_3) \cap E_{JI} (\pi_4) \ne \emptyset$ and such that each edge in $E_{JI}$ is visited at least twice. By assumption, in the above expansion for $\dE  |E'_1 |^2 $, only elements $\pi \in \bar Q_{k,l,k,l}$ might give a non-zero term in the summand. 

We consider the usual equivalence class by the action of pairs of bijections  on $\{1,\ldots,n\}$. Arguing as in Lemma \ref{le:ZZ1}, it is sufficient to prove for any fixed $\pi  = (\pi_1,\pi_2,\pi_3,\pi_4 )\in \bar Q_{k,l,k,l}$, we have
\begin{equation}\label{eq:mpi4}
m_4 (\pi) =  \ABS{ \sum_{ \pi'  \sim\pi  }  M_{\pi'_1}  M_{\pi'_2} \bar M_{\pi'_3}  \bar M_{\pi'_4}   } \lesssim n^{-1},
\end{equation}
where the sum is over all $\pi' = (\pi'_1,\pi'_2,\pi'_3,\pi'_4) \sim \pi$. Arguing as in the proof of Lemma \ref{le:ZZ1}, we have 
\begin{equation}\label{eq:boundmpi1}
m_4(\pi) \leq n^{-2k +2 } n^{(|I| + \delta_I )\wedge (|J| + \delta_J)  -2},
\end{equation}
where $\delta_I,\delta_J \in \{0,1/2,1\}$. We have $\delta_I = 1$ if $j_{\veps,1} = j_{\veps',1}$ for all $\veps,\veps'$ in $\{1,\ldots,4\}$. Indeed, arguing as in the proof of \eqref{eq:mpiI1}, in this case the factor $\prod_{\veps=1}^2  ( (B_0 B_0^*)_{i'_{\veps,k}i'_{\veps,k}})^{1/2}$ will appear when using \eqref{eq:boundCp} for $j = j_{\veps,1}$. Similarly, we have $\delta_I =1/2$, if $j_{\veps,1} = j_{\veps',1}$ for exactly $3$ different indices, say $\veps = 2,3,4$ for example: in this case, the factor $\prod_{\veps =1}^3 ((B_0 B_0^*)_{i'_{\veps,k}i'_{\veps,k}})^{1/2}$ will appear when using \eqref{eq:boundCp} for $j = j_{1,1}$ and $j = j_{2,1}$, hence the exponent $|I| - 3/2$. In all other cases,  when using \eqref{eq:boundCp}, there will be $4$ terms $\prod_{\veps=1}^4 ((B_0 B_0^*)_{i'_{\veps,k}i'_{\veps,k}})^{1/2}$ and $\delta_I = 0$. Thanks to the involution $\theta$ the same argument holds with $J$ replaced by $I$ and $\delta_J$ defined accordingly.

Hence if $(|I| + \delta_I) \wedge (|J| + \delta_J) \leq 2k-1$, \eqref{eq:boundmpi1} implies  \eqref{eq:mpi4}. Moreover, by counting degrees of $G$, we find that $|I| \leq 2k$ and $|J| \leq 2k$. We assume from now on $2k - 1  \leq |I|  \leq |J|  \leq 2k$.

We start with the case  $|I| = 2k-1$ and $\delta_I \in \{1/2,1\}$. Then there is one $JI$-edge of multiplicity $p \in \{ 3,4\}$ (this is $(j_1,i_1)$) and all other $JI$-edges are of multiplicity $2$.  It follows $|J|  = 2k-1$ as well.

We may argue as in the proof of Lemma \ref{le:ZZ1}:  to prove \eqref{eq:mpi4}, it suffices to prove that $\hat m_4(\pi) \lesssim n^{-1}$ where $\hat m_4(\pi)$ is obtained from $m_4(\pi)$ by removing, when  summing over $\pi'$, the constraints that $i'_{\veps,t} \ne i'_{\veps',t'}$ if $i_{\veps,t} \ne i_{\veps',t'}$ and similarly for the $J$-vertices.

Also, as in the proof of Lemma \ref{le:ZZ1}, we consider the graph $G_{IJ}$ on $I \sqcup J$ spanned by the edges in $E_{IJ}$. All vertices of $G_{IJ}$ have degree $2$, except one $J$-vertex with degree $p$ (this is a $j_{\veps,1}$) and one $I$-vertex with degree $p$ (this is a $i_{\veps,1}$).

{\em Case 1. } $|I| = 2k-1$ and $\delta_I =1/2$.
We have $p=3$ in this case. Then, by parity, there is one connected component of $G_{IJ}$ which the two vertices of degree $3$ and the other connected components of $G_{IJ}$ contain vertices of degree $2$ only.  The argument leading to \eqref{eq:mpifinal} implies that 
$$
\hat m_4(\pi) =n^{-2k+2}  | \sum_{i,j} \prod_{q=1}^3 (A_{1,q})_{ij} |\prod_{l=2}^{c} |\TR(A_l)| , 
$$
where $c$ is the number of connected components of $G_{IJ}$, $A_l$, $A_{1,q}$ are product of matrices $B_t$ or $B^\intercal _t$, $0 \leq t \leq k-1$, each $B_t$ appears exactly 4 times. The assumption $\delta_I = 1/2$, implies that each $A_{1,q}$ contains as a factor $B_0$ and one $A_l$, $l \geq 2$, contains a factor $B_0$. Since $B_0$ is a rank one projector, we deduce $|\prod_{l=2}^{c} |\TR(A_l)| \lesssim n^{c-2}$ and from \eqref{eq:boundCp}, that 
$$
| \sum_{i,j} \prod_{q=1}^3 (A_{1,q})_{ij} | \lesssim 1.
$$
So finally, 
$$
\hat m_4(\pi) \lesssim n^{-2k+2} n^{c-2} \leq n^{-1},
$$
where the inequality $c \leq 2k -1$ follows the fact that each  connected component contains at least one $J$-vertex and the connected  component of the $j_{\veps,1}$ of degree $3$ contains at least $2$. It concludes the proof of \eqref{eq:mpi4} in this case.

{\em Case 2. } $|I| = |J| =  2k-1$ and $\delta_I = \delta_J = 1/2$. In this case, there are two subcases to consider. Set $j_1 = j_{\veps,1}$ and $i_1 = i_{\veps,1}$. Either the two vertices of degree $4$ in $G_{IJ}$ (that is $j_1$ and $i_1$) are in the same connected component or not. If they are in the same connected component, then arguing as in case 1, we get 
\begin{align*}
& \hat m_4(\pi)  =n^{-2k+2}  | \sum_{i,j} \prod_{q=1}^4 (A_{1,q})_{ij} |\prod_{l=2}^{c} |\TR(A_l)| \\
 \quad \hbox{ or } \quad & \hat m_4(\pi) =n^{-2k+2}  | \sum_{i,j} (A_{1,q})_{ii} (A_{2,q})_{jj}  \prod_{q=3}^4 (A_{1,q})_{ij} |\prod_{l=2}^{c} |\TR(A_l)|, 
\end{align*}
with $A_{1,q}, A_l$ as above (the two cases corresponding to the two $4$-regular unlabeled connected multigraphs with $2$ vertices). Moreover, each $A_{1,q}$ contain $B_0$ in its product since $\delta_I = 1$. Hence from  \eqref{eq:boundCp}, we find
$$
\hat m_4(\pi)  \lesssim n^{-2k+2} n^{c-1} \leq n^{-1}.
$$
where the inequality $c \leq 2k -2$ follows from the fact that each  connected component contains one $I$-vertex and the connected component of $j_{1}$ contains at least $3$ (consider a path from $i_1$ to $j_1$ and recall that $(i_{\veps,k},j_1)$ is a $JI$-edge, $k \geq 2$). It concludes the proof of \eqref{eq:mpi4} in this subcase.

In the other subcase, $j_1$ and $i_1$ are not in the same connected component. Then these connected components are composed of two alternating cycles attached to $j_1$ or $i_1$. We obtain the expression:
$$
\hat m_4(\pi) =n^{-2k+2}  | \sum_{j} \prod_{q=1}^2 (A_{1,q})_{jj} | | \sum_{i} \prod_{q=1}^2 (A_{2,q})_{ii} | \prod_{l=3}^{c} |\TR(A_l)|. 
$$
Since $\delta_I = 1$, $A_{1,q}$ contains two factor $B_0$. From \eqref{eq:mpi4}, we get
$$
\hat m_4 (\pi) \lesssim n^{-2k+2} n^{c-1} \leq n^{-1}.
$$
where the inequality $c \leq 2k -2$ follows the fact that each  connected component contains at least two $IJ$-edges and two connected components contain at least $4$. It concludes the proof of \eqref{eq:mpi4} in this case.

{\em Case 3. } The final case to consider is $|I| = |J| = 2k$, $\delta_I = \delta_J =1$, all $JI$-edges have multiplicity $2$. As in the proof of Lemma \ref{le:ZZ1}, it suffices to prove that $\hat m_4(\pi) \lesssim n^{-1}$ where $\hat m_4(\pi)$ is obtained from $m_4(\pi)$ by removing, when  summing over $\pi'$, the constraints that $i'_{\veps,t} \ne i'_{\veps',t'}$ if $i_{\veps,t} \ne i_{\veps',t'}$ and similarly for the $J$-vertices.
In the graph $G_{IJ}$ on $I \sqcup J$ spanned by the edges in $E_{IJ}$, all vertices have now degree $2$. The argument leading to \eqref{eq:mpifinal} implies that 
$$
\hat m_4(\pi) \lesssim n^{-2k+2} n^{c'},
$$
where $c'$ is the number of connected components of $G_{IJ}$ not containing a boundary vertex $j_{\veps,1}, i_{\veps,k}$. To prove  \eqref{eq:mpi4}, it is thus sufficient to check that $c' \leq 2k -3$. Let $c$ be the number of connected components of $G_{IJ}$.

Let us first assume that all vertices $i_{\veps,k}$ are in the same connected component. Then $c' = c-1$ and this connected component has at least $4$, $IJ$-edges.  Moreover, as in the proof of case 2 in Lemma \ref{le:ZZ1}, the assumptions:  $\pi_1 \ne \pi_2$, $\pi_3 \ne \pi_4$ imply that there exist two distinct vertices and each adjacent to two distinct $IJ$-edges all distinct. Combining all possibilities and using $k\geq 2$, these $4$ $IJ$-edges can be chosen to be different from the $4$ $IJ$-edges attached to $i_{\veps,k}$ mentioned above. Since each connected component contains at least $2$ $IJ$-edges, it implies that $c \leq 2k-2$ and thus $c' \leq 2k -3$ as claimed.

Similarly, if not all vertices $i_{\veps,k}$ are in the same connected component then $c' \leq c-2$. Then, as in the proof of case 2 in Lemma \ref{le:ZZ1}, the bound $c \leq 2k-1$ comes from any of the two assumptions  $\pi_1 \ne \pi_2$ or $\pi_3 \ne \pi_4$.
\end{proof}

Our next lemma studies the variance of the variables $S_k(B,B)$ and $S'_k(B,B)$ defined above Lemma \ref{le:stein2}. Recall that the expectation of these variables were computed in \eqref{eq:ESk}.

\begin{lem}\label{le:Sk2}
Let $k \geq 1$ be an integer, and $B,B'$ in $M_{N}(\dC)^{\dN}$ satisfying assumptions \eqref{eq:defB0}-\eqref{eq:hypTCL}. Set $S_k = S_k (B,B')$ and $S'_k = S'_k (B,B')$.  We have
$$
\dE[ |S_k |^ 2 ] \lesssim  1 \quad \hbox{ and } \quad\dE[ |S'_k |^ 2 ] \lesssim  1.
$$
Moreover if $k \geq 2$, then 
$$
\dE [ |S_k - \dE S_k|^2 ] \lesssim n^{-1}\quad \hbox{ and } \quad \dE [ |S'_k - \dE S'_k|^2 ] \lesssim n^{-1}.
$$ 
For $k = 1$, we have, with $B_0 = u \otimes \bar v$ and $B'_0 = u' \otimes \bar v'$,
$$
\dE [ |S_1 - \dE S_1|^2 ] \lesssim \|u \|_4^2 \| u'\|^2_ 4 \|v \|_4^2 \| v'\|_4^2   \quad \hbox{ and } \quad \dE [ |S'_1 - \dE S'_1|^2 ] \lesssim \|u \|_4^2 \| u'\|^2_ 4 \|v \|_4^2 \| v'\|_4^2 .
$$ 
\end{lem}

\begin{proof}
We start with the first claim. Thanks to Assumption $(H_4)$, we have
$$
\dE |S_k| ^2 \vee \dE |S'_k|^2 \lesssim \sum_{(\pi_1,\pi_2) \in Q^{\times 2}_k } m_4(\pi_1,\pi_2,\pi_1,\pi_2), 
$$
where $m_4$ was defined in \eqref{eq:mpi4} on $Q^{\times 4}_k$. We next remark that 
$$
m_4(\pi_1,\pi_2,\pi_1,\pi_2) \leq m_2(\pi_1,\pi_1) m_2(\pi_2,\pi_2) 
$$
where for $m_2$ was defined in  \eqref{ZZ1VARS1} on $Q^{\times 2}_k$. From \eqref{eq:mpiI1} applied to $(\pi,\pi)$, $\pi \in Q_k$, we have  $m_2(\pi,\pi) \lesssim n^{-k+1} n^{|I|-1}$ where $I$ is the number of $I$-vertices in $G(\pi)$. Since $|I| \leq k$, we deduce that  $m_2 (\pi,\pi) \lesssim 1$ for any $\pi \in Q_k$ and that
$
m_4 (\pi_1,\pi_2,\pi_1,\pi_2) \lesssim 1.
$
It concludes the proof of the first claim of the lemma. 

We now prove the second claim for $k\geq 2$. We treat the case of $S_k$, the case of $S'_k$ being identical. For ease of notation, for $\pi \in Q_k$, we set 
$$
Y_\pi = \sum_{l=1}^k \PAR{ |X_{\pi\backslash (j_{l}i_{l})}|^2 + |X_{\pi}|^2}.
$$
We have
$$
\dE |S_k - \dE S_k|^2 = \sum_{(\pi_1,\pi_2) \in Q^{\times 2}_k } |M_{\pi_1}|^2 |M_{\pi_2}|^2 \dE \left[ (Y_{\pi_1} - \dE Y_{\pi_1} ) ( Y_{\pi_2} - \dE Y_{\pi_2} ) \right].
$$
In particular, if $E_{JI} (\pi_1) \cap E_{JI}(\pi_2) = \emptyset$, then the summand is $0$, where $E_{JI} (\pi_\veps)$ is defined in Lemma \ref{le:ZZ1}. We denote by $Q_{k,k}$ the subset of $(\pi_1,\pi_2)$ in $Q_k \times Q_k$  such that 
$ E_{JI} (\pi_1) \cap E_{JI}(\pi_2) \ne \emptyset $. Thanks to Assumption $(H_4)$, it suffices to prove that for any $(\pi_1,\pi_2) \in Q_{k,k}$, we have 
$$
m_4(\pi_1,\pi_2,\pi_1,\pi_2)  \lesssim n^{-1}. 
$$
where $m_4$ was defined in \eqref{eq:mpi4}. Set $\pi = (\pi_1,\pi_2,\pi_1,\pi_2)$. The assumption $(\pi_1,\pi_2) \in Q_{k,k}$ implies that $\pi$ has at least one $JI$-edge of multiplicity at least $4$. Hence, counting the sum of  multiplicities of $JI$-edges, we get $2 (|I| -1) + 4 \leq 4k$. Hence $|I| \leq 2k -1$ the bound $m_4(\pi_1,\pi_2,\pi_1,\pi_2)  \lesssim n^{-1}$ is a consequence of \eqref{eq:boundmpi1} and case 2 in the proof of Lemma \ref{le:stein2}.

We finally consider the case $k=1$. By definition, we simply have 
$$
S_1 = \sum_{i,j} u_{i} \bar u'_i  v'_j \bar v_j  (1 +|X_{ij}|^2). 
$$
Hence, 
$$
\dE |S_1 - \dE S_1 |^2 = \dE \ABS{ \sum_{i,j} u_{i} \bar u'_i  v'_j \bar v_j ( |X_{ij}|^2 - 1)}^2 = \sum_{i,j}|u_{i}|^2 |u'_{i}|^2 |v_j|^2 |v'_j|^2  \dE [  ( |X_{ij}|^2 - 1)^2 ] = \kappa \|u\cdot u'\|^2_2\|v \cdot v'\|_2^2 ,
$$
with $\kappa = \dE [  ( |X_{ij}|^2 - 1) )^2 ]$ and, for vectors $w,w'$, $w \cdot w'$ is the vector with coordinates for $(w \cdot w')_i = w_i w'_i$. From Cauchy-Scwharz inequality, the conclusion follows for $S_1$. For $S'_1$, the computation is identical.
\end{proof}

Our final technical lemma is a deviation bound of $\hat L - L$.
\begin{lem}\label{le:stein3}
Let $k \geq 1$ be an integer, and $B$ in $M_{N}(\dC)^{\dN}$ satisfying assumptions \eqref{eq:defB0}-\eqref{eq:hypTCL}. Then, for $L = L_k (B)$ and $\hat L = \hat L_k (B)$, we have for $k \geq 2$,
$$
\dE | L - \hat L |^4 \lesssim  n^{-3}.
$$
For $k = 1$,
$$
\dE | L - \hat L |^4 \lesssim n^{-2} \|u \|_4^4 \|v \|_4^4.
$$
 
\end{lem}

\begin{proof}
We first assume that $k \geq 2$. The proof is again a variant of Lemma \ref{le:ZZ1}. Recall our notation: we have $\hat X = X^{pq}$ and $\hat L = L^{pq}$ with $(p,q)$ uniformly distributed on $\{1,\ldots,n\}^2$. As already observed, for any $p,q$, we have
$$
L^{pq} - L = \sum_{\pi \in Q_k} M_{\pi} (X^{pq}_\pi - X_\pi) = \sum_{l=1}^k \sum_{\pi \in Q^{l,pq}_{k}} M_{\pi} (X^{pq}_\pi - X_\pi) = \sum_{l=1}^k   D^{pq}_l ,
$$
where $Q^{l,pq}_{k} \subset Q_k$ is the subset of paths $\pi = (j_1,i_1,\ldots,j_k,i_k)$ such that $(p,q) =  (j_l,i_l)$. From Hölder inequality, we deduce that 
$$
|L^{pq} - L |^ 4 \leq k ^3  \sum_{l=1}^k |  D^{pq}_l |^4 .
$$
Hence, to prove the lemma when $k\geq 2$, it is sufficient to prove that for any $1 \leq l \leq k$, 
\begin{equation}\label{eq:SI04}
\dE [ |D^{pq}_{l}|^4 ] \lesssim n^{-3} ,
\end{equation}
where $(p,q)$ is uniform on $\{1,\ldots,n\}^2$.  To this end, we write,
\begin{equation}\label{eq:SI4}
\dE [|D^{pq}_{l}|^ 4] =  \frac{1}{n^2} \sum_{p,q} \sum_{\pi \in (Q^{l,pq}_{k})^{\times 4}} M_{\pi_1} M_{\pi_2} \bar M_{\pi_3} \bar M_{\pi_4}  \dE \left[(X^{pq}_{\pi_1} - X_{\pi_1})(X^{pq}_{\pi_2} - X_{\pi_2})\overline{(X^{pq}_{\pi_3} - X_{\pi_3})}\overline{(X^{pq}_{\pi_4} - X_{\pi_4})}\right],
\end{equation}
where $\pi = (\pi_1,\pi_2,\pi_3,\pi_4)$. Next, as already observed, for $\pi \in Q^{l,pq}_{k}$, we have $$X^{pq}_\pi - X_\pi = X_{\pi\backslash pq} (X'_{pq} - X_{pq}).$$

We  use the notation of the proof of Lemma \ref{le:ZZ1} and Lemma \ref{le:stein2}. For $\pi \in Q_k^{\times 4}$, we consider the bipartite oriented graph $G(\pi)$ with vertex set $I \cup J$ and edge sets  $E_{IJ} = \cup_\veps E_{IJ} (\pi_\veps)$, $E_{JI} = \cup_\veps E_{JI} (\pi_\veps) $.  Due to the independence of the entries of $X$, the  expectation in \eqref{eq:SI4} is zero unless each edge in $E_{JI}$ is visited at least twice. We denote by $\bar Q^l_{k,k,k,k} \subset Q_{k}^{\times 4}$ the collection of  elements $\pi = (\pi_1,\pi_2,\pi_3,\pi_4)$ with $\pi_{\veps} = (j_{\veps,1}, \ldots, i_{\veps,k})$ such that $(j_{\veps,l},i_{\veps,l})  = (j_{\veps',l},i_{\veps',l})$ for all $\veps, \veps' \in \{1,\ldots,4\}$ and such that each edge in $E_{JI}$ is visited at least twice.

We consider the same equivalence class on $Q_k ^{\times k}$ than in the proof of Lemma \ref{le:stein2}. Since there is a finite number of equivalence classes and the expectation on right-hand side of \eqref{eq:SI4} is constant over each equivalence class,   to prove \eqref{eq:SI04}, it is thus sufficient to prove that for any fixed $ \pi = (\pi_1,\pi_2,\pi_3,\pi_4) \in  \bar Q^l_{k,k,k,k}$, we have
\begin{equation}\label{eq:mgpi}
m_4(\pi) =  \ABS{ \sum_{  \pi' \sim \pi}  M_{\pi'_1} M_{\pi'_2} \bar M_{\pi'_3} \bar M_{\pi'_4}  } \lesssim n^{-1}.
\end{equation}
where $\pi' =(\pi'_1,\pi'_2,\pi'_3,\pi'_4)$. Note that $m_4$ was already defined in \eqref{eq:mpi4}.

By construction, each edge, say $(j,i)$, in $E_{JI}$ has multiplicity at least $2$ and $ (j_{\veps,l},i_{\veps,l}) \in E_{JI}$ has multiplicity at least $4$. In particular, $ 2(|I|-1) + 4 \leq  4 k$ and $2 (|J|-1) +  4 \leq 4k$. The bound \eqref{eq:mgpi} is then a consequence of \eqref{eq:boundmpi1} and case 2 in the proof of Lemma \ref{le:stein2}.

It remains to check the lemma for $k=1$. This is straightforward. We have
$$
L^{pq} - L = \bar v_p ( X'_{pq} - X_{pq} ) u_q.  
$$
Hence $\dE |\hat L - L |^4 = n^{-2} \|u \|_4^4 \|v \|_4^4 \dE | X'_{11} - X_{11} |^4$. It concludes the proof of the lemma. \end{proof}

All technical lemmas are now gathered to prove our main statement.

\begin{proof}[Proof of Theorem \ref{th:steinTCL}]
By Lemma \ref{le:ZZ1}, it suffices to prove the statement with $L = (L_{k_t} (B_t))_{t \in T}$ in place of $Z = (Z_{k_t}(B_t))_{t \in T}$. We set  $W = (  W_{k_t}(B_t))_{t \in T}$, $\tilde W = (\tilde W_{k_t}(B_t))_{t \in T}$ and $\Delta = \max_t \|u_t \|_\infty \vee \| v_t \|_\infty$. Note that since $\|u_t\|_2 , \|v_t\|_2 \lesssim 1$ by assumption \eqref{eq:hypTCL}, we have $\|u_t\|_4 \leq (\|u_t\|_\infty \|u_t \|_2 )^{1/2} \lesssim \Delta^{1/2}$.

We start with the case $\Delta = o(1)$. We apply Theorem \ref{th:stein} in Appendix with $d = |T| $, $X = L$, $X' = \hat L$, $\cF = \sigma(X)$, $\Lambda = k/n^2$ and $E_0 = 0$. Then Lemmas \ref{le:hatSigma}-\ref{le:stein3} imply that if $f$ be a smooth function on $\dC^{T}$, $
\ABS{ \dE f (L) - \dE f (W) } \lesssim n^2 \times ( n^{-5/2} + n^{-9/4} + n^{-2} \Delta^{3}) \lesssim n^{-1/4} + \Delta^3$, where the constant hidden in $\lesssim$ depend on the first three derivatives of $f$, $|T|$ and $\dE [|X_{ij}|^4]$. It concludes the proof of the last claim of the theorem.

We now prove the first claim of the theorem. Let $T_2 = \{ t \in T : k_t \geq 2 \}$ and $T_1 = T \backslash T_2$. The idea to use  Theorem \ref{th:stein}  for elements in $T_2$ after having resampled a few entries of $X$ so that to have removed most of the dependency from $L_{k_t} (B_t)$, $t\in T_1$. More precisely, let $X''$ be an independent copy of $X$ and $X'$. Let $S  \subset \{1,\ldots,n\}^2$ and define $L^S$ as the random variable $L = L(X^S)$ where $X^S_{ij} = X_{ij}$ if $(i,j) \notin S$ and $X^S_{ij} = X''_{ij}$ if $(i,j) \in S$.

We write $S = \{(p_1q_1),\ldots,(p_m q_m) \}$ and, set $S_0 = S$ and for $1 \leq l \leq |S|$, $S_l = \{(p_1q_1),\ldots,(p_lq_l)\}$.  From the triangle inequality, we have, for $L = L_{k_t}(B_t)$, $t \in T_2$,  
\begin{equation}\label{eq:Lresample}
\dE [|L^S - L|] \leq \sum_{l=1}^{|S|}  \dE [|L^{S_l} - L^{S_{l-1}}|] \lesssim |S| n^{-1/4},
\end{equation}
where the last inequality is a consequence of Lemma \ref{le:stein3}: $\max_{p,q} \dE | L^{pq} - L |^4 \leq \sum_{p,q} \dE | L^{pq} - L |^4 \lesssim n^{-1}$ if $L = L_k (B)$ with $k \geq 2$.

Now, we set $m = n^{1/5}$ and for each $t \in T$, let $I_{t} \subset \{1,\ldots,n\}$ be the $m$-largest coordinate entries of $u_{t}$, $(u_{1,t})_i = (u_t)_i \IND_{i \in I_t}$ and $u_{2,t} = u_t - u_t$. Since $\| u_t\|_2 \lesssim 1$, we have  
$$
\| u_{2,t} \|_{\infty} \leq \|u_t\|_2  /\sqrt m \lesssim n^{-1/10}.
$$
We define $J_t$ and $v_{\veps,t}$ similarly with the $m$-largest entries of $v_t$. We set $S = \cup_t I_t \times J_t$. From what precedes, $\dE [|L^S_{k_t}(B_t) - L_{k_t} (B_t)|] \lesssim n^{-1/20}$ for all $t \in T_2$.

We note also that if $t \in T_1$, then $k_t = 1$ by definition For ease of notation, we set $L_1(B) = L_1(u,v)$. From the bilinearity of the scalar product and the definition of $S$, $$L_{1} (B_t) = L_1 (u_t,v_t) = L_1 (u_{1,t},v_{1,t}) + L^S_1 (u_{2,t},v_{2,t}) + L^S_1 (u_{2,t},v_{1,t})  +L^S_1 (u_{1,t},v_{2,t}). $$

We consider the vector $$((L^S_{k_t}(B_t))_{t \in T_2}, (L^S_1 (u_{2,t},v_{2,t}),L^S_1 (u_{1,t},v_{2,t}) ,L^S_1 (u_{2,t},v_{1,t}))_{t \in T_1}).$$ This vector is independent of $(L_1 (u_{1,t},v_{1,t}))_{t \in T_1}$ by construction. Also,  as above, we may apply Theorem \ref{th:stein}: we deduce from  Lemmas \ref{le:hatSigma}-\ref{le:stein3} that this vector is close in distribution to the corresponding Gaussian vector: $$(W_{k_t}(B_t))_{t \in T_2}, (W_1 (u_{2,t},v_{2,t}),W_1 (u_{1,t},v_{2,t}) ,W_1 (u_{2,t},v_{1,t}))_{t \in T_1}).$$
It remains to use the bilinearity of $L_{1}$ again: $L_1 (u_t,v_t) $ is close in distribution to $L_1(u_{2,t},v_{2,t}) +  W_1 (u_{2,t},v_{2,t})  + W_1 (u_{1,t},v_{2,t})  + W_1 (u_{2,t},v_{1,t})$. It concludes the proof. \end{proof}

\subsection{Tightness of long random matrix products}

The goal of this subsection is to prove that the random variable $Z = Z_k(B)$ is tight in $(n,k,B)$ under natural assumptions on $B$ and as long as $k$ is not too large. The main difference with the preceding subsection being that we will now follow more closely the dependence in $k$.

Our next proposition will be quantitative. We make the following assumption on the growth rate of $C_k = \dE |X_{ij}|^k$.


Finally, we will have two type technical conditions which are restrictions due the limitations of our proofs. We will assume either that Assumption \Hsym holds 
or that the matrices $B_k$ are diagonal:
\begin{eqnarray}\label{eq:diagB}
\hbox{for all $k\geq 1$, $B_k$ is diagonal}.
\end{eqnarray}

Our main result in this subsection is the following quantitative estimate. 

\begin{prop}\label{prop:tight}
Assume that assumption $(H_k)$ holds and $C_k = (\dE |X_{ij}|^k)^{1/k} \leq n^{0.02}$ and let $B \in M_n(\dC)^{\dN}$ such that \eqref{eq:defB0} holds. We assume also that either \Hsym or \eqref{eq:diagB} holds.  Then, there exists a numerical constant $c >0$   such that for all $k \leq n^{c}$, we have
$$
\dE | Z_k ( B) |^2 \leq c k^{12} \prod_{t=0}^{k-1}\| B_k\|^2.
$$
\end{prop}

The proof of Proposition \ref{prop:tight} starts as in the proof of Lemma \ref{le:ZZ1}. The technical challenge is that we now track carefully the dependency $k$, this is where the extra assumptions \Hsym or \eqref{eq:diagB} come into play. When assumption \eqref{eq:diagB} holds, Proposition \ref{prop:tight} is closely related to  \cite[Proposition 6.1]{bordenave-capitaine16}where the assumptions on $B_0$ are however stronger. When assumption \Hsym holds, our proof is significantly more difficult and requires more ideas.

For ease of notation, we set $Z = Z_k(B)$. We have
$$
\dE  [ |Z|^2 ]   = n^{-k+1} \sum_{\pi \in \bar P_{k,k}}  M_{\pi_1} \bar M_{\pi_2}  \dE [ X_{\pi_1} \bar X_{\pi_2} ],
$$
where $\pi = (\pi_1,\pi_2)$ with $\pi_{\veps} = (j_{\veps,1},\ldots, i_{\veps,k})$, $\bar P_{k,k} \subset P_k \times P_k$ was defined in the proof of Lemma \ref{le:ZZ1} below \eqref{eq:vardedef} and $M_{\pi_\veps}$, $X_{\pi_\veps}$ was defined below \eqref{eq:devZ}. We recall that $\bar P_{k,k}$ is the set of $\pi = (\pi_1,\pi_2)$ such that each $JI$-edge has multiplicity at least two.

We start with the proof of Proposition \ref{prop:tight} when \eqref{eq:diagB} holds.
Recall the definition of equivalence classes on $P_k \times P_k$ introduced in the proof of Lemma \ref{le:ZZ1}.  Since $\dE [ X_{\pi_1} \bar X_{\pi_2} ]$ is constant over each equivalent class, we get
\begin{equation}\label{eq:Z2expr}
\dE  [ |Z|^2 ] = \sum_{\pi \in \bar \cP_{k,k}} \dE [ X_{\pi_1} \bar X_{\pi_2} ] \PAR{\sum_{\pi' \sim \pi }  M_{\pi'_1} \bar M_{\pi'_2} },
\end{equation}
where $\bar \cP_{k,k}$ is the set of equivalence classes in $\bar P_{k,k}$. Thanks to assumption \eqref{eq:diagB}, we may further restrict our attention to $\widetilde P_{k,k}$ and $\widetilde \cP_{k,k}$ defined as the subsets of $\bar P_{k,k}$ and $\bar \cP_{k,k}$ such that $j_{\veps,t+1} = i_{\veps,t}$ for all $t \in \{1,\ldots,k-1\}$, $\veps = 1,2$.

 For integer $1 \leq \ell \leq k$, we denote by $\widetilde \cP_{k,k}(\ell)$, the set of $\pi$ in $\widetilde \cP_{k,k}$  such its associated bipartite directed graph $G(\pi)$ satisfies $|I| \wedge |J| = \ell$. Recall that, by degree counting, $|I| \vee |J| \leq k$ for all $\pi \in \bar P_{k,k}$. We deduce from \eqref{eq:mpiI1} that, if \eqref{eq:diagB} holds,
\begin{equation}\label{eq:Z2exprD}
\dE  [ |Z|^2 ] \leq \beta^{2k} \sum_{\ell= 1}^{k} n^{\ell -k} \sum_{\pi \in \widetilde \cP_{k,k}(\ell)} \ABS{   \dE [X_{\pi_1}  \bar X_{\pi_2} ]}.
\end{equation}

Our first lemma gives a bound $\dE [  | X_{\pi_1} X_{\pi_2} |]$, which is classical in random matrix theory, see for example \cite[p24]{MR2760897}. Recall  that  $C_k^k = \dE |X_{ij}|^k$.
\begin{lem}\label{le:moment}
If $\pi \in \bar P_{k,k}$ is such that $|E_{JI}|= \ell$,
$$
\ABS{ \dE [  X_{\pi_1} \bar X_{\pi_2}  ]} \leq C_k^{6(k-\ell)}.
$$
Moreover if $\dE X_{ij}^3 = \dE X_{ij} \bar X_{ij}^2 = 0$ then $
\ABS{ \dE [  X_{\pi_1} \bar X_{\pi_2}  ]} \leq C_k^{4(k-\ell)}$.
\end{lem}

\begin{proof}
If $e= (j,i) \in E_{JI}$, we denote by $m_e , m_{1,e} , m_{2,e}$ the multiplicities of $e$ in $\pi = ( \pi_1,\pi_2)$, $\pi_1$ and $\pi_2$. We have $m_e = m_{1,e} + m_{2,e} \geq 2$. Using the independence of the entries $X_{ij}$ and $\dE |X_{ij}|^2 = 1$, we get 
\begin{equation}\label{eq:EXpi}
\ABS{ \dE [ X_{\pi_1} \bar X_{\pi_2} ] }= \ABS{ \prod_{e \in E_{JI}} \dE  [ X_e ^{m_{1,e}} \bar X_e ^{m_{2,e}}] } \leq \prod_{e \in E_{JI} : m_e \geq 3} \dE  [ |X_e |^{m_e}].
\end{equation}
Since $m_e \leq k$, we deduce from H\"older inequality that 
$$
\dE [ | X_{\pi_1} X_{\pi_2} | ]\leq \prod_{e : m_e \geq 3} C_k ^{m_e}.
$$
Finally, we observe that $\sum_{e} m_e\IND_{m_e \geq 3} = \sum_e m_e - 2 \sum_{e} \IND_{m_e = 2} = 2( k -  a_2)$, where $a_2$ is the number of $JI$-edges of multiplicity $2$.  Observe that
$$
 2 a_2 + 3 (a - a_2) \leq 2k.  
$$ 
In other words, $2 (k-a_2) \leq 6 ( k -\ell)$. The conclusion follows. If $\dE X_{ij}^3 = \dE X_{ij} \bar X_{ij}^2 = 0$, we repeat the above argument and use that $m_e > 2$ implies $m_e \geq 4$.
\end{proof}

Our final lemma for Proposition \ref{prop:tight} when assumption \eqref{eq:diagB} holds is an upper bound on $\widetilde \cP_{k,k}(\ell)$. A slightly weaker statement is proved in \cite[Lemma 6.1]{MR2760897}.

\begin{lem}\label{le:nbclasseq}
For any integers $1 \leq \ell \leq k$, we have
$$
| \widetilde \cP_{k,k}(\ell) | \leq  2(2k)^{6 (k-\ell)+11}.
$$\end{lem}

\begin{proof}
Up to multiplying by a factor $2$, we give an upper bound on the number  of equivalence classes of $\widetilde P_{k,k}$ with $|J| = \ell$.  We equip the set $T =    \{1,2\} \times \{1,\ldots, k\}$  with the lexicographic order, $ (1,1) < (1,2) < \cdots < (1,k) < (2,1) < \cdots < (2,k)$. We think of as an element of $T$ as a time, if $t = (\veps,l)$, $l \geq 2$, we write $t-1 \in T$ for the immediate predecessor of $t$. 
Recall that $j_{\veps,l+1} = i_{\veps,l}$ for $\pi \in \widetilde P_{k,k}$, hence $\pi$ is uniquely determined by $(j_{t})_{t \in T}$ and $i_{\veps,k}$, $\veps = 1,2$.

In each equivalence class, we define a canonical element  by saying that $\pi$ is {\em canonical} if it is minimal 
for the lexicographic order in its equivalence class. For example, if $\pi \in \bar P_{k,k}$ is canonical then 
$j_{1,1} = 1$,  $J = \{1,\ldots, |J| \}$ and the $J$ vertices appear for the first time in the sequence $(\pi_{t})_{t \in T}$ in order ($2$ before $3$ and so on).  
Our goal is then to give an upper bound on the number of canonical elements in $\tilde P_{k,k}$ with $|J| = \ell$.

Fix such canonical element $\pi = (\pi_{t})_{t\in T}$ in $\tilde P_{k,k}$. We build a non-decreasing sequence of trees $(\cT_t)$ on $J$ as follows. Initially, for $t=(1,1)$, $\cT_{(1,1)}$ has no edge and vertex set $j_{1,1} = \{1 \}$. At time $t > (1,1) \in T$, if the addition to $\cT_{t-1}$ of the oriented edge 
$(j_{t-1},j_{t})$ does not create a cycle, then $\cT_t$ is the tree spanned by $\cT_t$ and $(j_{t-1},j_{t})$, we then say that 
$(j_{t-1},j_{t})$ is a {\em tree edge}. Otherwise $\cT_t = \cT_{t-1}$. By construction, the vertex set of $\cT_{t}$ is $J_t = \{ j_s : s \leq t\}= \{ 1, \ldots , j\}$ where $j$ is the number of visited $J$-vertices so far (since $\pi$ is canonical). In particular, $\cT_{(2,k)}$ is a spanning tree of $J$, and the number of tree edges is $e =\ell-1$. 

Next, we divide the elements of $T$ into types. We say that $t \in T$ is a {\em first time} if $\cT_t = \cT_{t-1} \cup (j_{t-1},j_t)$. 
We say that 
$t \in T$ is a {\em tree time} if  $(j_{t-1},j_{t})$ is a tree edge of $\cT_{t-1}$. Finally, the other times $t \in T$ are 
called {\em important times}. We have thus partitioned $T$ into first times, tree times, and important times. By construction, $j_t \in J_{t-1}$ for important and tree times. We also add another special time, the {\em merging time} $m_\star$ defined as follows. If $J(\pi_2) \cap J(\pi_1) = \emptyset$, then $m_\star = 0$. Otherwise $m_\star = (2,l)$ where $l$ is minimal with the property that $j_{2,l} \in J_{1,k}$. In this last situation, the merging time is a first time.

We also note that all edges are visited at least twice. We deduce that the number of important times is at most 
$$
2k - 2 e = 2( k - \ell) +2. 
$$

Observe that if $t = (\veps,l)$, $l < k$ is a first time and $t \ne m_\star$ then $t+1$ cannot be a tree time. The sequence $(j_{\veps,l}), 1 \leq l \leq k$, can thus be decomposed as 
the successive repetitions of: $(i)$ a sequence of first times (possibly empty), $(ii)$ an important time or the merging time, 
$(iii)$ a sequence of tree times (possibly empty).

We mark each important time $t = (\veps,l)$ by the vector $(j_{t}, j_{\veps,\tau-1})$ where $\tau > l$ is the next time which is not a tree time if this time exists. Otherwise $t = (\veps,l)$ is the last important time of $\pi_{\veps}$ and we put the mark $(j_t,j_{\veps,k})$. Similarly, if the merging time $m_\star = (2,l)$, $1 \leq l \leq k$, we mark it  by the vector $(j_{m*}, j_{2,\tau-1})$ where $\tau > l$ is is the next time which is not a tree time if this time exists or otherwise we mark it by $(j_{m*}, j_{2,k})$.

We claim that the canonical sequence $\pi$ is 
uniquely determined by $i_{1,k},i_{2,k}$, the positions of the important times, their marks, and the value of the merging time $m_\star$ and its mark. Indeed, if $t = (\veps,l)$ is an important time or the merging time, the sequence $(j_{\veps,l},\ldots, j_{\veps,\tau-1})$ is the unique path in $\cT_t$ from $j_{\veps,l}$ to $j_{\veps,\tau-1}$ (there is a unique path between two vertices of a tree). Moreover, if $\sigma \geq \tau$ is the next important time, merging time or $\sigma = k+1$ if $t$ is the last important time of $\pi_\veps$, $(\veps,\tau), \ldots, (\veps,\sigma-1)$ is a sequence of first times (if $\tau = \sigma$ this sequence is empty). It follows that if $j = |J_t|$, we have $j_{\veps,\tau +s-1} = j+s$, for $s = 1, \cdots, \sigma- \tau$, by the minimality of canonical paths.

There are at most $2k$ possibilities for the position of an important time and the value of the merging time and $\ell^2$ possibilities for its mark. In particular, the number of distinct canonical paths in $\widetilde P_{k,k}$ 
with $|J| = \ell$  is bounded by 
$$
k^2 (2\ell^3)^{2 (k-\ell)+3},
$$
(the factor $k^2$ is front accounts for the possibilities for $(i_{1,k},i_{2,k})$). It gives the requested bound. \end{proof}

We are ready for the proof of Proposition \ref{prop:tight} when assumption \eqref{eq:diagB} holds. 

\begin{proof}[Proof of Proposition \ref{prop:tight} when assumption \eqref{eq:diagB} holds.] By Lemma \ref{le:moment}, if $\pi \in \widetilde \cP_{k,k}(\ell)$, we have $$\ABS{ \dE [X_{\pi_1} \bar X_{\pi_2} ] } \leq  C_k^{6 (k-\ell)},$$
 (indeed, since each $I$ vertex is preceded by an edge in $E_{JI}$, we have $ | E_{JI} | \geq \ell$ and similarly for $J$).  It follows from  \eqref{eq:Z2exprD} and Lemma \ref{le:nbclasseq} that
$$
\dE  [ |Z|^2 ] \leq 2 (2k)^{11} \prod_{t=0}^{k-1}\| B_k\|^2 \sum_{\ell= 1}^{k} \rho^{k -\ell},
$$
where 
$$
\rho = \frac{C_k^6 (2k)^6}{n}. 
$$
Thanks to the assumption $C_k \leq n^{0.02}$, we have $\rho \leq 1/2$ for all $n$ large enough and $k \leq n^{c_1}$ for $c_1 >0$ small enough,. The conclusion follows for all $n$ large enough. For smaller $n$, the conclusion of the proposition trivially holds if the constant $c$ is large enough. \end{proof}

The remainder of this subsection is devoted to the proof of Proposition \ref{prop:tight} when \Hsym holds. The upper bound \eqref{eq:mpiI1} is not sufficient when the matrices $B_k$ are not diagonal, essentially because the number of equivalence classes is much larger than the equivalence classes considered in Lemma \ref{le:nbclasseq}.

Let $\pi = (\pi_1,\pi_2) \in  P_{k} \times P_{k}$ with associated graph $G = G(\pi)$. Observe from \eqref{eq:EXpi}  that $\dE [X_{\pi_1} \bar X_{\pi_2} ]$  depends only on the multiplicities $(m_{1,e}, m_{2,e})$ of each $JI$-edge $e$ of $G$. We introduce a new equivalence relation on $P_k \times P_k$, we write $\pi \hat \sim \pi'$ if there exists a bijection $\psi$ of $\{1,\ldots,n\}^2$ such that $\psi(j_{\veps,t} , i_{\veps,t}) = (j'_{\veps,t} , i'_{\veps,t})$ for all $(\veps,t)$. By construction, this equivalence relation is coarser that the equivalence relation $\sim$ and $\dE [X_{\pi_1} \bar X_{\pi_2}]$ is constant over each equivalence class.
 We rewrite \eqref{eq:Z2expr} as 
\begin{equation*}
\dE  [ |Z|^2 ]   =  \sum_{\pi \in \hat \cP_{k,k}}  \dE [ X_{\pi_1} \bar X_{\pi_2} ] \PAR{ \sum_{\pi' \hat \sim \pi} M_{\pi'_1} \bar M_{\pi'_2} } ,
\end{equation*}
where $\pi = (\pi_1,\pi_2)$ with $\pi_{\veps} = (j_{\veps,1},\ldots, i_{\veps,k})$, $\hat \cP_{k,k}$ is the set of equivalence classes of elements in $P_k \times P_k$ such that each $JI$-edge has even multiplicity (otherwise $\dE [ X_{\pi_1} \bar X_{\pi_2} ]$ is zero by assumption \Hsym and the independence of entries).

We next introduce $\hat \cP_{k,k}(\ell,c) \subset \hat \cP_{k,k}$ as the set of elements $\pi$ such that $G = G(\pi)$ has $|E_{JI} | =  \ell$ distinct $JI$-edges and the graph $G_{IJ}$ spanned by $IJ$-edges has $c$ connected components (beware that $\ell$ is not $|I| \wedge |J|$ contrary to as above). Observe that $c \leq \ell$ (in each connected component, pick a distinguished $I$-vertex, this vertex is neighbor to a $JI$-edge). By Lemma \ref{le:moment}, we arrive at
\begin{equation}\label{eq:Z2exprS}
\dE  [ |Z|^2 ]   \leq  \sum_{\ell = 1}^{k} \sum_{c =1 }^{\ell}   \sum_{\pi \in \hat \cP_{k,k}(\ell,c)}  C_k^{4(k-\ell)} \ABS{ \sum_{\pi' \hat \sim \pi} M_{\pi'_1} \bar M_{\pi'_2} }.
\end{equation}
The graph $G_{JI} = (I\sqcup J ,E_{JI})$ is a multigraph with $2k$ oriented edges. If $e = (i_{\veps,t},j_{\veps,t+1}) \in E_{JI}$, we set $M_e = B_{t}$ if $\veps = 1$ and $M_e = \bar B_{t}$ otherwise, we also set $e_- = i_{\veps,t}$, $e_+ = j_{\veps,t+1}$. We may then write 
$$
\mu_2(\pi) = \sum_{\pi' \hat \sim \pi} M_{\pi_1} \bar M_{\pi_2} = n^{-k+1} \sum_{ \bm i' , \bm j' } \prod_{e \in E_{IJ}} (M_e)_{i'_{e_-}j'_{e_+}},
$$
where the sum is over all $\bm i' = (i'_u)_{u \in I}$, $\bm j' = (j'_v)_{v \in J}$, with $i'_u , j'_v \in \{1,\ldots,n\}$, and for all $(\veps_1,t_1)$, $(\veps_2,t_2)$, $(j'_{j_{\veps_1,t_1}},i'_{i_{\veps_1,t_1}}) \ne (j'_{j_{\veps_2,t_2}},i'_{i_{\veps_2,t_2}}) $  if $(j_{\veps_1,t_1},i_{\veps_1,t_1}) \ne (j_{\veps_2,t_2},i_{\veps_2,t_2})$ (this is the constraint imposed by the equivalence relation $\hat \sim$). 

We set 
$$
\hat \mu_2 (\pi ) = n^{-k+1} \sum_{ \bm i' , \bm j' } \prod_{e \in E_{IJ}} (M_e)_{i'_{e_-}j'_{e_+}},
$$
where the sum is over all $\bm i' = (i'_u)_{u \in I}$, $\bm j' = (j'_v)_{v \in J}$, with $i'_u , j'_v \in \{1,\ldots,n\}$.  Arguing as in the proof of Lemma \ref{le:ZZ1}, if we remove iteratively the constraints $(j'_{j_{\veps_1,t_1}},i'_{i_{\veps_1,t_1}}) \ne (j'_{j_{\veps_2,t_2}},i'_{i_{\veps_2,t_2}}) $  if $(j_{\veps_1,t_1},i_{\veps_1,t_1}) \ne (j_{\veps_2,t_2},i_{\veps_2,t_2})$, we deduce that 
$$
\mu_2(\pi) =  \hat \mu_2 (\pi) + \sum_{\pi'< \pi} \pm \hat \mu_2(\pi'),
$$  where the sum is over all $\pi' \in \hat \cP_{k,k} (\ell',c)$, $\ell' > \ell$ is obtained from $\pi$ choosing a collection of edges in $E_{JI}$ and gluing them (that is, identify their end vertices).

We thus obtain from \eqref{eq:Z2exprS} that 
$$
\dE  [ |Z|^2 ]   \leq  \sum_{\ell=1}^{k} \sum_{c =1 }^{\ell}   \sum_{\pi \in \hat \cP_{k,k}(\ell,c)}  C_k^{4(k-\ell)} L(\pi) \ABS{ \hat \mu_2(\pi)},
$$
where $L(\pi)$ is the number of $\pi'$ which in their iterative  decomposition from $\mu_2 (\pi')$ to $\hat \mu_2(\pi') $ have a term $\pm \hat \mu_2(\pi)$. Note that $$L(\pi)\leq \prod_{e \in E_{JI}} K(m_e),$$ where  $m_e \in 2 \dN$ is the multiplicity $e \in E_{JI}$ and $K(m)$ is the number of ways to partition $\{1,\ldots,m\}$ into blocks of even size. We have $K(2) = 1$ and for $m \geq 4$, 
\[
K(m) \leq \binom{m}{m/2}  \cB_{m/2},
\]
where $\cB_k$ is the $k$-th Bell's number (number of partitions of $\{1,\ldots,k\}$). Since $\cB_k \leq k^k$ and $m/2 \leq m-2$ for $m \geq 4$, we get 
$$
K(m) \leq (2m)^{m-2}.
$$
Since $m \leq 2k$, we obtain the rough upper bound: 
$$
L(\pi) \leq (4k)^{\sum_{e} (m_e - 2)} = (4k)^{2(k -  \ell)}.
$$
We arrive at the expression 
$$
\dE  [ |Z|^2 ]   \leq  \sum_{m=1}^{2k} \sum_{c =1 }^{\ell}   \sum_{\pi \in \hat \cP_{k,k}(\ell,c)}  \PAR{ 4k C^2_k }^{2(k-\ell)} \ABS{ \hat \mu_2(\pi)}.
$$
The reason we have introduced the equivalence relation $\hat \sim$ was to get an upper bound on $L(\pi)$ where in the exponent $k-\ell$ appears, for the equivalence relation $\sim$ the exponent would have been too large for large $\ell$ (of order $\ell$).

\begin{lem}\label{le:lemma1112}
Let $p = 1/2$. For any $\pi \in \hat \cP_{k,k} (\ell,c)$, we have 
$$
|\hat \mu_2(\pi) | \leq  n^{c + p (k - \ell) - k}  \prod_{t=0}^{k-1}\| B_k\|^2  .
$$
\end{lem}
\begin{proof}
We note that $n^{k-1} \hat \mu_2(\pi)$ is equal to a product over the connected components of $G_{IJ}$. From assumption \Hsym, all vertices of the oriented graph $G_{IJ}$ have even degrees. We may thus apply Theorem \ref{th:sumproduct} (in Subsection \ref{subsec:sumproduct} of the Appendix), it implies that 
\begin{equation}\label{eq:djeidje}
n^{k-1} \ABS{ \hat \mu_2(\pi) } \leq \prod_{t=0}^{k-1}\| B_k\|^2  \prod_{i=1}^{c} n^{1 - \frac{\IND_{\delta_i > 0}}{2}},
\end{equation}
where $\delta_i \in \{0,1,2\}$ is the number of matrices $B_0$ in this connected component. Note that we arrive at the claim of the lemma. We get 
$$
\ABS{ \hat \mu_2(\pi) } \leq n^{c + 1/2 -k} \prod_{t=0}^{k-1}\| B_k\|^2.
$$
In particular, if  $\ell < k$, we obtain the claimed statement since $(k-\ell) /2 \leq 1/2 $. It remains to check the lemma when $\ell = k$. In this case, all vertices of $G(\pi)$ have degree $2$. It follows that each weak connected component of $G(\pi)$ is a cycle. In particular $n^{k-1}  \hat \mu_2(\pi) $ is a product of traces of the $B_t$'s or their transposes. Since $\TR(M) \leq \mathrm{rank} (M) \| M \|$, the statement is straighforward.
\end{proof}

We have thus proved so far that 
\begin{equation}\label{eq:exprZ2S}
\dE  [ |Z|^2 ]   \leq  \left( \prod_{t=0}^{k-1}\| B_k\|^2  \right)  \sum_{m=1}^{2k} \sum_{c =1 }^{\ell}  | \hat \cP_{k,k}(\ell,c)|  \PAR{ 4k C^2_k }^{2(k-\ell)} n^{c + p (k - \ell) - k} .
\end{equation}
Our final lemma is an upper bound on $\hat \cP_{k,k}(\ell,c)$.

\begin{lem}\label{le:nbclasseqS}
For any integers $1 \leq c \leq \ell \leq k$, we have
$$
| \hat \cP_{k,k}(\ell,c) | \leq  (2k )^{12 (k-\ell) + 4 (\ell - c) +12}.
$$\end{lem}

\begin{proof}
The proof is a variant of the proof Lemma \ref{le:nbclasseq}. We equip $T = \{1,2\} \times \{1,\ldots, k\}$ with the lexicographic order. We write $\pi = (\pi_1,\pi_2) \in \hat P_{k,k}$, $\pi_{\veps} = (j_{\veps,1},\ldots,i_{\veps,k})$ as $\pi = (\pi_t)_{t \in T}$ with $\pi_{t} = (j_{\veps,l},i_{\veps,l})$ for $t = (\veps,l)$.  We say that $\pi \in \hat P_{k,k}$ is {\em canonical}, if its distinct $JI$-edges are $e_1 = (1,1), e_2 = (2,2), \ldots, e_\ell = (\ell,\ell)$ where $\ell  = |E_{JI}|$ and if these edges are visited for the first time in increasing order. By construction, there is a unique canonical element in each equivalence class for $\hat \sim$. For any canonical element $\pi \in \hat \cP_{k,k}(\ell,c) $, we have $E_{JI} = \{e_1,\ldots,e_\ell\}$. The set $C = \{1,\ldots,c\}$ is thought as the set of connected components of $G_{IJ}$ ordered by the order of appearance of its $I$-vertices. For each $t \in T$, we denote by $c_t \in C$ the index of the connected component of $i_t$. 

Similarly to the proof of  Lemma \ref{le:nbclasseq}, we build an non-decreasing sequence of edge-colored spanning trees $\cT_t$, $t \in T$, on the vertex set $C_t =\{c_{s} : s \leq t \}$. A colored edge from $p$ to $q$ is a triple $(p,e,q)$ with $p,q \in C$ and $e \in E_{JI}$. At time  $t=(1,1)$, $\cT_{(1,1)}$ has no edge and vertex set $c_{1,1} = \{ 1 \}$. At time $t > (1,1) \in T$, if the addition to $\cT_{t-1}$ of the oriented edge 
$(c_{t-1},c_{t})$ does not create a cycle, then $\cT_t$ is $\cT_{t-1} \cup (c_{t-1},\pi_{t-1},c_t)$, we then say that 
$(c_{t-1},\pi_{t-1},c_t)$ is a {\em tree edge}. Otherwise $\cT_t = \cT_{t-1}$. By construction, the vertex set of $\cT_{t}$ is $C_t$. In particular, $\cT_{(2,k)}$ is a spanning tree of $C$, and the number of tree edges is $c-1$. 

As in the proof of  Lemma \ref{le:nbclasseq}, we divide the elements of $T$ into types. We say that $t \in T$ is a {\em first time} if $\cT_t = \cT_{t-1} \cup (c_{t-1},\pi_{t-1},c_t)$. 
We say that 
$t \in T$ is a {\em tree time} if  $ (c_{t-1},\pi_{t-1},c_t)$ is a tree edge of $\cT_{t-1}$. The other times $t \in T$ are 
called {\em important times}. Finally, the {\em merging time} $m_\star$ is $0$ if $E_{JI} (\pi_2) \cap E_{JI} (\pi_1) = \emptyset$, otherwise $m_\star = (2,l)$ where $l$ is minimal with the property that $c_{2,l} \in \Gamma_{1,k}$

As in the proof of  Lemma \ref{le:nbclasseq}, the sequence $(\pi_{\veps,l}), 1 \leq l \leq k$, can be decomposed as 
the successive repetitions of: $(i)$ a sequence of first times (possibly empty), $(ii)$ an important time or the merging time, 
$(iii)$ a sequence of tree times (possibly empty).

We mark each important time $t = (\veps,l)$ by the vector $(\pi_{t-1},c_t,c_{\veps,\tau-1})$ where $\tau > l$ is the next time which is not a tree time if this time exists. Otherwise $t = (\veps,l)$ is the last important time of $\pi_{\veps}$ and we put the mark $(\pi_{t-1},c_t,c_{\veps,k})$. Similarly, if the merging time $m_\star = (2,l)$, $1 \leq l \leq k$, we mark it  by the vector $(\pi_{m*-1},c_{m^*}, c_{2,\tau-1})$ where $\tau > l$ is is the next time which is not a tree time if this time exists or otherwise we mark it by $(\pi_{m*-1},c_{m^*}, c_{2,k})$.

Arguing exactly as in the proof of  Lemma \ref{le:nbclasseq},  the canonical sequence $\pi$ is  uniquely determined by the positions of the important times, their marks, and the value of the merging time $m_\star$ and its mark.

We are left with the task of giving an upper bound on the  number of important times. Assume that $(c_{t-1} , \pi_{t-1},c_{t})$  is a tree edge and $c_{t-1} =p \in C$,  $\pi_{t-1}= e_i = (i,i)$, $\pi_{t} = e_j = (j,j)$. Then, $i \ne j$ are in $p$-th  connected component of $G_{IJ}$. In particular, if $j$ is the only $J$ -vertex in the $p$-th  connected component of $G_{IJ}$ and since $e_i$ is visited at least twice, the tree edge $(c_{t-1} , \pi_{t-1},c_{t})$  will be visited at least twice (indeed, otherwise, $(e_i,e_{j'}) = (\pi_{s-1},\pi_s)$ with $s \ne t$ and $j \ne j'$ and we would have $(i,j,j')$ in the same connected component of $G_{IJ}$). Let $\kappa_{2}$ be the number of connected components of $G_{IJ}$ which have at least two distinct $J$-vertices. It follows that the number of tree edges which are visited at least twice is $t_2 \geq c - 1 -\kappa_2$. 

If $n_p$ is the number of $J$-vertices in the $p$-th connected component of $G_{IJ}$, we have 
$$
|J| - c = \sum_p (n_p -1 ) \geq \kappa_2.
$$
As already pointed, we have $|J| \leq \ell$. It follows that $\kappa_2 \leq \ell -c$ and $t_2 \geq 2 c - \ell - 1$. Let $t_1$ be the number of tree edges visited once. We have $t_2 + t_1 = c - 1$. the number of important times is at most 
$$
2k - 2  t_2 - t_1 =  2k - t_2 - c +1  \leq  2 k + \ell - 3 c +2. 
$$

Now, there are at most $2k$ possibilities for the position of an important time and the value of the merging time and $c^2\ell$ possibilities for its mark. In particular, the number of distinct canonical paths in $\widetilde P_{k,k}$ 
with $\ell = |E_{JI}|$ and $c$ connected components in $G_{IJ}$ is upper bounded by 
$$
 (2k c^2\ell )^{2 k+ \ell - 3 c +3}.
$$
It gives the requested bound. \end{proof}

\begin{proof}[Proof of Proposition \ref{prop:tight} when assumption \Hsym holds.]
From \eqref{eq:exprZ2S} and Lemma \ref{le:nbclasseqS}, we have 
\begin{equation}\label{eq:dedeoid}
\dE[ | Z|^2] \leq \left( \prod_{t=0}^{k-1}\| B_k\|^2  \right)  (2k)^{12} \sum_{\ell=1}^{ k } \sum_{c=1}^\ell \rho_1^{k-\ell} \rho_2^{\ell - c},
\end{equation}
where 
$$
\rho_1 = \frac{(4k C^2_k )^2 (2k )^{12}}{n^{1-p}}
\quad \hbox{ and } \quad 
\rho_2 = \frac{(2k )^{4}}{n} .
$$
Thanks to assumption $C_k \leq n^{0.02}$, $\rho_1$ and $\rho_2$ are less than $1/2$ for all $n$ large enough and $k \leq n^{c_1}$ with $c_1>0$ small enough. The conclusion follows.
\end{proof}

We conclude this subsection with a variant of Proposition \ref{prop:tight} which will be useful in the sequel. We introduce the normalized Frobenius norm:  for $M \in M_n(\dC)$, 
$$
\| M \|_2 = \sqrt{ \frac{1}{n} \TR ( M M ^*) }.
$$
We have $\| M \|_2 \leq \|M \| \leq \sqrt n \| M \|_2$. We consider the  assumption :  for some $0 < \veps < 1/2$,
\begin{equation}\label{eq:hypB20}
\hbox{for all $k\geq 1$ : } \| B_k \| \leq  n^{\veps} \| B_k \|_2.
\end{equation}
We have the following strengthening of Proposition \ref{prop:tight}.
\begin{prop}\label{prop:tight2}
Assume that assumption $(H_k)$ holds and $C_k = (\dE |X_{ij}|^k)^{1/k} \leq n^{0.02}$ and let $B \in M_n(\dC)^{\dN}$ such that \eqref{eq:defB0} and \eqref{eq:hypB20} hold for some $\veps >0$. We assume also that either \Hsym or \eqref{eq:diagB} holds.  Then,  there exists  a numerical constant $c >0$ such that, if  $\veps < 1/c$, for all $k \leq n^{c}$, we have
$$
\dE | Z_k ( B) |^2 \leq c k^{12}  \| B_0\|^2 \prod_{t=1}^{k-1} \| B_t \|^2_2.
$$
\end{prop}

\begin{proof}
We only give the proof when assumption \Hsym holds. The proof extends immediately to the simpler case when \eqref{eq:diagB} holds. The only difference  with Proposition \ref{prop:tight} lies in the application of Lemma \ref{le:lemma1112}.  As pointed in the proof of Lemma \ref{le:lemma1112}, $n^{k-1} \hat \mu_2 (\pi)$ is equal to a product over the connected components of $G_{IJ}$, say 
$$n^{k-1} \hat \mu_2 (\pi) = \prod_{i=1}^c L_i.$$ Since all degrees in $G_{IJ}$ are at least two, each connected component has at least two (oriented) edges. If the $i$-th connected component has exactly two edges, say $e,f$ then, we may directly upper bound $L_i$ by 
$$
\sum_{j,j'} |(M_e)_{jj'} (M_f)_{jj'}| \leq \sqrt{ \TR (M_e M_e^*) \TR (M_f M_f^*)},
$$
where the inequality follows from the Cauchy-Schwarz inequality. If $M_e \notin \{ B_0,\bar B_0\}$, we use $\TR (M_e M_e^*) \leq n \|M_e\|_2^2$. If $M_e  \in \{B_0,\bar B_0\}$, thanks to \eqref{eq:defB0}, we have 
$$
\TR (B_0 B_0^*) = \| B_0 \|^2.
$$

If $L_i$ has more than $2$ edges, we apply Theorem \ref{th:sumproduct} instead. Therefore, using \eqref{eq:hypB20}, Equation \eqref{eq:djeidje} is improved in 
$$
n^{k-1} \ABS{ \hat \mu_2(\pi) } \leq   n^{\veps m}  \| B_0\|^2 \prod_{t=1}^{k-1} \| B_t \|^2_2 \prod_{i=1}^{c} n^{1 - \frac{\IND_{\delta_i >0} }{2}}, 
$$
where $m$ is the number of  edges $e$ of $G_{IJ}$ such that $M_e \notin \{B_0,\bar B_0\}$ and which are in a connected component with more than two  edges. Consequently, the proof of Lemma \ref{le:lemma1112} gives the improved bound
$$
|\hat \mu_2(\pi) | \leq   n^{c + p (k - \ell) - k} n^{\veps m}  \| B_0\|^2 \prod_{t=1}^{k-1} \| B_t \|^2_2 .
$$
We next claim that 
\begin{equation}\label{eq:mbd}
m \leq 6 (k-c).
\end{equation}
Indeed, let $m ' \geq m$  be the number of edges of $G_{IJ}$  which are in a connected component with more than two edges. Let $c_2$ (resp. $c_{3}$) be the number of connected components with two edges (resp. at least $3$). We have 
$$
c_2 + c_{> 2} = c \quad \hbox{ and } \quad 2 c_2 + 3 c_{> 2} \leq 2k.
$$ 
(thanks to the symmetry assumption \Hsym the factor $3c_{>2}$ could in fact be improved in $4 c_{>2}$). We deduce that $c_2 \geq 3 c - 2k$. Since $m' = 2k -2 c_2$, we deduce \eqref{eq:mbd}. We write $k-c = (k-\ell) + (\ell -c)$  and then \eqref{eq:dedeoid} becomes
$$
\dE[ | Z|^2] \leq \left( \| B_0\|^2 \prod_{t=1}^{k-1} \| B_t \|^2_2  \right) \sum_{\ell=1}^{ k } \sum_{c=1}^\ell \rho_1^{k-\ell} \rho_2^{\ell - c},
$$
where 
$$
\rho_1 = \frac{(4k C^2_k )^2 (2k )^{12} n^{6 \veps } }{n^{1-p}}
\quad \hbox{ and } \quad 
\rho_2 = \frac{(2k )^{4} n^{6 \veps} }{n} .
$$
 The conclusion follows.
\end{proof}

\section{Appendix}

\label{sec:appendix}
\subsection{Random analytic functions}

\label{subsec:RAF}
We recall here some results from Shirai \cite{shirai12}. Let $D \subset \mathbb{C}$ be a   bounded connected open   set in
the complex plane. Denote by ${\mathcal H} (D)$  the space of complex analytic functions in $D$, endowed 
 with the distance
\begin{equation}\label{eq:distHU}
d(f,g)  = \sum_{j \geq 1} 2^{-j} \frac{ \| f -g\|_{K_j} }{ 1 + \| f - g \|_{K_j}},
\end{equation}
where $(K_j)_{j \geq 1}$ is an exhaustion by compact sets of $D$ and $\|f-g\|_{K_j}$ denotes the supremum norm of $f-g$ on $K_j$. We recall that it is a complete separable metric space.

The space ${\mathcal H} (D)$ is equipped with the (topological)
Borel $\sigma$-field ${\mathcal B}({\mathcal H} (D))$ and the set of probability measures on $({\mathcal H} (D), {\mathcal B}({\mathcal H} (D)))$ is
denoted by ${\mathcal P}({\mathcal H}(D))$.
By {\it a random analytic function on D} we mean an ${\mathcal H} (D)$-valued
random variable  on a probability space. The probability
law of a random analytic function is uniquely determined by its finite dimensional
distributions.

The following proposition from  \cite{shirai12} provides a typical example of random analytic functions.

\begin{prop}[Prop.~2.1 in \cite{shirai12}] 
\label{prop2.1Shirai}
Let  $\{\Psi_k\}_k \subset {\mathcal H} (D)$ be a sequence of independent centered random analytic functions defined on the same probability space. Suppose that $\sum_{k=1}^\infty \mathbb{E}\left( |\Psi_k|^2\right)$
is a locally integrable function of $z$ in $D$. Then, 
$X(z)=  \sum_{k=1}^\infty \Psi_k(z)$ is convergent in ${\mathcal H} (D)$ almost surely and thus defines a random analytic function on $D$.
\end{prop}

The following  lemmas from \cite{shirai12} turn out to be useful to prove  tightness results.
\begin{prop}[Prop.~2.5 in \cite{shirai12}]
 Let $(f_n)_n$ be a sequence of random analytic functions in $D$. If $(\Vert f_n\Vert_K)_n$ is tight for any compact set K, then the sequence of distributions of $f_n$, $({\mathcal L}(f_n))_n$, is
tight in ${\mathcal P}({\mathcal H}(D))$.
\end{prop}

\begin{lem}[Lemma~2.6 in \cite{shirai12}]\label{Shirai}
For any compact set K in D, there exists $\delta > 0$ such that for all  $f \in {\mathcal H}(D)$,
\[
\left\| f \right\|_K^p \leq (\pi \delta^2)^{-1} \int _{\overline{K_\delta}} \left| f(z)\right|^p m(dz), 
\]
for any $p > 0$, where $\overline{K_\delta}\subset D$ is the closure of the $\delta$-neighborhood of K  and $m$ denotes the Lebesgue measure.
\end{lem}
Using that, by Markov's inequality,   for any $C>0$ and any $p>0$, 
\begin{equation}\label{ineg0}\mathbb{P} \left( \left\| f_n \right\|_K >C \right) \leq \frac{1}{C^p} \mathbb{E} \left( \left\| f_n \right\|^p_K \right),\end{equation} this leads to the following tightness criterion  of sequence of random analytic functions. 
\begin{lem}\label{criterion}
Let $D \subset \dC$ be an open connected set in the complex plane.
Let $K \subset  D$ be a compact set  and $U\subset K$ be an open set. Let $(f_n)$ be a sequence of random analytic functions  on $D$. If there exists $p>0$ and $C>0$ such that for all large $n$, $\sup_{z\in K} \dE \vert f_n(z)\vert^p < C $ then $(f_n)$ is a tight sequence of random analytic functions  on $U$.
\end{lem}

Let us first recall some basic facts on  point processes (we refer to Daley and Vere-Jones \cite{daley-verejones} for details and terminology).

Define $\mathcal N(\C)$ the set of boundedly finite Borel measures  $\mu$ on $\C$, that is $\mu(A)<\infty$ for every bounded Borel set $A$. We endow $\mathcal N (\C)$ with the usual weak topology: $\mu_n \to \mu$ in $\mathcal N(\C)$ if for all continuous function $f$ with compact support, $\int f d\mu_n \to \int f d\mu$ (see \cite[Prop.~A.2.6.II]{daley-verejones}). The space $\mathcal N(\C)$ is a complete separable metric space (see \cite[Thm.~A.2.6.III]{daley-verejones}). The set $\mathcal P( \mathcal N(\C))$ of probability measures on $\cN(\C)$ is a complete separable metric space and the L\'evy-Prokhorov distance is a metric for the weak convergence of measures in $\cP ( \cN(\C))$.

For instance, for any Borel set $B$, and any  random matrix $M$,
\[
\sum_{ \lambda \in \spectrum M \cap B} \delta_\lambda
\]
is an example of a  $\mathcal N(\C)$-valued random variable. 

Let $D \subset \C$ be an open connected set. 
For a nonzero analytic function $f\in {\mathcal H}(D)$, we denote by $Z_f$ the set of zeroes of $f$ and define
\[
\xi_f=\sum_{z\in Z_f} m_z \delta_z
\]
where $m_z$ is the multiplicity of a zero $z$.  If moreover $f$ is  random, then $\xi_f$ is $\mathcal N(\C)$-valued. 

The following proposition is fundamental for our purpose.

\begin{prop}[Prop.~2.3 in \cite{shirai12}] \label{Shirai:zeroes}
Suppose that a sequence of random analytic functions $\{X_n\}$ converges in law to $X$ in $\mathcal H(D)$. Then the zero process $\xi_{X_n}$ converges in law to $\xi_X$ provided that $X\neq 0$ almost surely,  i.e. for all continuous function $\varphi$ with compact support in $D$, 
\[
\int \varphi d \xi_{X_n} \overset{\text{weakly}}{\underset{n \to \infty}{\longrightarrow}} \int \varphi d \xi_{X} .
\]
\end{prop}

 Following the terminology of Hough, Krishnapur, Peres, Vir{\'a}g \cite{HKP09}, we say that $g$ is a {\it Gaussian analytic function} on a domain $D$, if $g$ is a random analytic function on $D$, $(g(z), z \in D)$ is a Gaussian process and  for $z , w \in D$, $\mathbb{E}( g(z)  g(w)) =0$. The distribution of a Gaussian analytic function is characterized by its kernel $K(z,w) = \mathbb{E}( g(z) \overline{g(w)})$ (see \cite[Chapter 2]{HKP09}). The  Edelman-Kostlan's formula takes a particularly simple form for Gaussian analytic functions: the intensity of zeros of $g$ is given by $\frac 1 {4 \pi} \Delta \ln K(z,z)$, where $\Delta$ is the usual Laplacian. 

\begin{rem}
Let $d \geq 1$ be an integer. By taking the product topology on $M_d(\dC) \simeq \dC^{d^2}$, the topologies defined on $\cH(D)$ and $\cP(\cH(D))$ extend immediately for the set $\cH_d(D)$ of complex analytic functions $D \to M_{d}(\dC)$.
\end{rem}

\subsection{Stein's method of exchangeable pairs}

We use the following  normal approximation theorem which relies on Stein's method of exchangeable pairs. It is a straightforward extension of \cite[Theorem 3]{MR2797946}. 

Let $X \in \dC^d$ be a centered complex random variable and $X' \in \dC^d$ such that $(X,X')$ and $(X',X)$ have the same law.

We denote by $Z \in \dC^d$ be the centered complex Gaussian vector whose covariance is characterized by $\dE [Z Z^*] = \Sigma_1$ and $\dE [ ZZ^{\T}] = \Sigma_2$. Let $\cF$ be a $\sigma$-subalgebra of our underlying probability space such that $X$ is $\cF$-measurable.

\begin{thm}\label{th:stein}
Let $\Lambda \in M_d(\dC)$ be an invertible matrix. Assume that, for some $\cF$-measurable $E_0 \in \dC^d$ and $E_1,E_2 \in M_d(\dC)$,  
$$
\dE [ X' - X  | \cF ] = - \Lambda X + E_0,
$$
$$
\dE [ (X' - X) (X'-X)^*  | \cF ] = 2 \Lambda \Sigma_1 + E_1 \quad   \hbox{ and } \quad  \dE [ (X' - X) (X'-X)^{\T}  | \cF ] = 2 \Lambda \Sigma_2 + E_2.
$$
 Then for any function $f \in C^3 (\dC^d)$, 
\begin{align*}
&| \dE f (X) - \dE f (Z) | \leq \\
& \| \Lambda^{-1} \| \left( M_1(f) \dE \|E_0\|_2 + \frac{\sqrt{d}}{4} M_2(f) (\dE \NRMHS{E_1} + \dE \NRMHS{ E_2}) + M_3(f) \dE \| X - X' \|_2^3 \right).  
\end{align*}
 where for integer $k \geq 1$, $M_k(f) = \sup_{x \in \dC^n} \| D^k f(x) \|$ and $D^k$ denotes the $k$-th derivative. 
\end{thm}
\begin{proof}
If $X$ is real-valued then the claim is contained in \cite[Theorem 3]{MR2797946}. In the general complex case, we associate to $X$ the real vector of dimension $2d$, $Y = (\Re(X),\Im(X))$ of covariance matrix
$$
\Sigma = \dE [ Y Y^\T]  = \begin{pmatrix} \dE[\Re(X) \Re(X)^\T] & \dE[\Re(X) \Im(X)^\T] \\
\dE[\Im(X) \Re(X)^\T] & \dE[\Im(X) \Im(X)^\T]
\end{pmatrix} = \begin{pmatrix} \Sigma_{11} & \Sigma_{12} \\
\Sigma_{21} & \Sigma_{22}
\end{pmatrix}.
$$
By construction, we have $\Sigma_{11} = \Re ( \Sigma_1 + \Sigma_2)/2$, $\Sigma_{22} = \Re ( \Sigma_1 - \Sigma_2)/2$ and $\Sigma_{21} = \Sigma_{12}^{\T} = \Im ( \Sigma_1 + \Sigma_2)/2$. Similarly, we have
$$
\dE [ Y' - Y  | \cF ] = - \Lambda' Y + E'_0, \quad \hbox{ and } \quad  \dE [ (Y' - Y)(Y'-Y)^\T  | \cF ] = 2 \Lambda' \Sigma + E, 
$$
with 
$$
\Lambda' = \begin{pmatrix}
\Re(\Lambda) & - \Im (\Lambda) \\
\Im(\Lambda) &\Re (\Lambda)
\end{pmatrix} \;, \quad E'_0 = \begin{pmatrix}
\Re (E_0) \\ \Im(E_0)
\end{pmatrix}\quad \hbox{ and } \quad E = \begin{pmatrix}
E_{11} & E_{12} \\
E_{21} & E_{22}
\end{pmatrix},
$$
where $E_{11} = \Re ( E_1 + E_2)/2$, $E_{22} = \Re (E_1 - E_2)/2$ and $E_{21} = E_{12}^{\T} = \Im ( E_1 + E_2)/2$. The map $\dC \to \dR^2 :  z \mapsto (\Re(z),\Im(z))$ being an isometry, we find that $\| \Lambda'^{-1} \| = \| \Lambda ^{-1} \|$, $\|E'_0\|_2 = \| E_0\|_2$ and $2 \NRMHS{E}^2 =  \NRMHS{E_1}^2  + \NRMHS{E_2}^2 $. It then remains to apply \cite[Theorem 3]{MR2797946} to $Y$.
\end{proof}

\subsection{Convergence of Riemann sum of rational functions}

In this paragraph, we prove the following improvement on classical rates for the convergence of Riemann sums. 

\begin{prop}\label{prop:convratfunc}
Let $P,Q$ be polynomials in $\dC[X]$. Denote by   $z_i$'s the roots of $Q$,  $m_i$ their  multiplicities  and set $\delta_i =  \min ( |z_i|, 1/|z_i|)$. Then, if $\max_i \delta_i <1$, for some constant $C >0$ (depending only on the degrees of $P,Q$), for all integers $n \geq \max_i m_i$,
$$
\ABS{\frac{1}{n} \sum_{k=0}^{n-1} \frac{P(e^{\frac{2 \iC \pi k}{n}})}{Q(e^{\frac{2 \iC \pi k}{n}})} - \frac{1}{2 \iC \pi} \oint_{\dS^1} \frac{P(\omega)}{\omega Q(\omega)} d\omega } \leq C \max_i \left( \frac{n^{m_i-1} \delta_i^{n} }{\min(|z_i|, |z_i|^{m_i} ) ( 1- \delta_i)} \right).
$$
\end{prop}

\begin{proof}
We set $\omega_n = e^{\frac{2 \iC \pi }{n}}$,
$$
I_n = \frac{1}{n} \sum_{k=0}^{n-1} \frac{P(\omega_n^k)}{Q(\omega_n^k)} \quad \hbox{ and } \quad I = \frac{1}{2 \iC \pi }  \oint_{\dS^1} \frac{P(\omega)}{\omega Q(\omega)} d\omega = \frac{1}{2 \pi} \int_0^{2\pi} \frac{P(e^{\iC \theta})}{Q(e^{\iC \theta})} d\theta.
$$
We can decompose $P/Q$ into simple elements and use the linearity of sums and integrals. To prove the proposition, it follows that it is sufficient to check the claim in the following cases, for some integer $l \geq q \geq  0$,
\begin{enumerate}[(i)]
\item $P(x) = x^l, Q(x) = 1$, 
\item $P(x) = x^q$, $Q(x) = (z - x)^{l+1}$ with  $|z| > 1$,
\item $P(x) = x^q$, $Q(x) = (z - x)^{l+1}$ with $|z| < 1$ (and for $z = 0$ for $q = 0$).
\end{enumerate}

Case (i) is trivial. If $l = 0$, $I_n = I = 1$, if $l \geq 1$, $I_n = I = 0$. We next consider the case (ii). We have, for $|x| < 1$, the converging Taylor series: 
\begin{equation}\label{eq:Taylor}
(1- x)^{-l-1} = \sum_{p=0}^\infty x^p \binom{l+p}{p}.
\end{equation}
Since $\binom{l+p}{p} \sim p^l/l!$ as $p$ goes to infinity, it is standard to check that for all $n \geq 1$,
$$
\sum_{p=n}^\infty x^p \binom{l+p}{p} \leq C n^l \frac{x^n}{1-x},
$$
for some constant $C$ depending on $l$.

In particular, for $|z| > 1$ and $\omega \in \dS^1$, 
$$
(z- \omega)^{-l-1} = \frac{1}{z^{l+1}}\sum_{p=0}^{n-q-1} \left(\frac{\omega}{z}\right) ^p \binom{l+p}{p} + R_n(\omega),
$$
with $|R_n(\omega) | \leq C n^l |z|^{-n+q-l} / ( |z| - 1)$.
We deduce that 
$$
I_n =   \sum_{p=0}^{n-q-1} \frac{1}{z^{l+1+p}}  \binom{l+p}{p} \frac {1}{n} \sum_{k=0}^{n-1} \omega^{ k(p+q)}_n  + R_n = \frac{\IND_{q=0}}{z^{l+1}} + R_n,
$$
with $|R_n  | \leq C n^l |z|^{-n+q-l} / ( |z| - 1)$ and we have used that $p+q = 0 \hbox{ mod}(n)$ has no solution on $0 \leq p \leq n-q-1$ except for $q=p=0$. It concludes the proof in this case since $I = \IND_{q = 0} / z^{l+1}$ by the Residue Theorem (the worst bound corresponding to $q = l$).

The final case (iii) is similar. From \eqref{eq:Taylor}, for $|z| < 1$ and $\omega \in \dS^1$, we have for $n \geq l-q+1$, 
$$
(z- \omega)^{-l-1} = \frac{1}{(-\omega)^{l+1}}\sum_{p=0}^{n-l+q-2} \left(\frac{z}{\omega}\right) ^p \binom{l+p}{p} + \tilde R_n(\omega),
$$
with $|\tilde R_n(\omega) | \leq C n^l |z|^{n-l+q-1} / ( 1 - |z|)$.
We deduce that 
$$
I_n =   (-1)^{l+1} \sum_{p=0}^{n-l+q - 2} z^p  \binom{l+p}{p} \frac {1}{n} \sum_{k=0}^{n-1} \omega^{ -k( p -q + l + 1)}_n  + \tilde R_n =  \tilde R_n,
$$
where $|\tilde R_n  | \leq C n^l |z|^{n-l+q-1} / (1 -  |z|)$ and we have used that $p-q + l+1 = 0 \hbox{ mod}(n)$ has no solution on $0 \leq p \leq n-l+q-2$. It concludes the proof since $I =0$ by the Residue Theorem (the worst bound corresponding to $q=0$). \end{proof}

\subsection{Norm bound for sum-product matrix factors}
\label{subsec:sumproduct}

In this subsection, we consider the following setting. Let $G = (V,E)$ be a finite directed graph on the vertex set $V$. The graph $G$ is allowed to be a multigraph, that is, the edge set $E$ is a finite set equipped with two maps $o,t : E \to V$ where $o(e)$ is interpreted as the origin vertex of the edge and $t(e)$ as the terminal vertex. Note that $o(e) = t(e)$ is not forbidden. For ease of notation, we set $e_- = o(e)$ and $e_+ =t(e)$. A (weak) {\em connected component} of $G$ is a connected component of the undirected graph obtained from $G$ by removing edge orientations.  The degree of $v \in V$ is $\deg(v) = \sum_{e} \IND( v = o(e) ) + \IND( v = t(e)) $.  We say that $G$ is {\em even} if all vertices have even degree. Note that the degree of a vertex is even if and only if its loop-less degree is even.

We consider a collection $M = (M_e)_{e \in E}$ of elements in $M_n(\dC)$.  We are interested in evaluating the expression 
\begin{equation}\label{eq:defTGM}
T_G (M) = \sum_{ \bm i  } \prod_{e \in E} (M_e)_{i_{e_-}i_{e_+}},
\end{equation}
where the sum is over all $\bm i = (i_v)_{v \in V}  \in \{1,\ldots,n\}^V$.

For example, if $V = \{1,2\}$ and $E = \{ e_1,e_2 \}$ with $e_{\veps,-} = 1$ and $e_{\veps,+} = 2$, then 
$$
T_G (M) = \sum_{i,j} (M_{e_1})_{ij}(M_{e_2})_{ij} = \TR(M_{e_1} M_{e_2}^\intercal). 
$$   

The main result of this subsection is the following general upper bound on $T_G(M)$ for even graphs.  The statement is tailored for our needs. The first statement is a corollary of \cite[Theorem A.31]{BaiSilversteinbook} which built upon \cite{MR0727035}, see also \cite{zbMATH06017643} for an improved bound. We will give an independent proof for completeness.

\begin{thm}\label{th:sumproduct}
Let $G = (V,E)$ be a finite weakly connected directed even graph and $M \in M_n(\dC)^E$. We have
$$
| T_G(M) | \leq n  \prod_{e \in E} \| M_e\|. 
$$
Moreover, if $f \in E$ and $M_{f}$ has rank $r$,  we have 
$$
| T_G(M) | \leq \sqrt{r n }    \prod_{e \in E} \| M_e\|.
$$
\end{thm}

 The remainder of the subsection is devoted to the proof of  Theorem \ref{th:sumproduct}. We introduce the following $\ell^p\to \ell^q$ norms in $M_n(\dC)$: for $M \in M_n(\dC)$, 
$$
\| M \|_{1,\infty} = \max_{i,j} |M_{ij}| \quad  \hbox{ and } \quad   \| M \|_{2,\infty} = \max_{i} \sqrt{ \sum_{j} |M_{ij}|^2 }. 
$$
Note that $\| M \|_{1,\infty} \leq \| M \|_{2 , \infty } \leq \|M \|_{2,2} = \| M  \| \leq \|M \|_{\rF} \leq \sqrt{n} \|M\|$.

Now, let $G = (V,E)$ be as above, we fix an ordering on $E$ and write $ E = \{ e_1, e_2 , \ldots ,e_m\}$.  We classify elements of $E$ as follows. Let $e = e_t$. We say that $w\in \{e_-,e_+\}$  {\em starts at $e$} if $w \notin \cup_{s < t}\{e_{s,-}, e_{s,+} \}$. We say that $w$ {\em ends at $e$} if $w \notin \cup_{t  <  s }\{e_{s,-}, e_{s,+} \}$. We say that $e$ is {\em bad} if $e_+$ and $e_-$ both starts at $e$ or both ends at $e$. Finally, $w$ is {\em open at $e$} otherwise.

Next, for $e \in E$ and $M \in M_n(\dC)$, we set 
$$
\| M \|_{e} = \left\{ \begin{array}{ll} \|M\|_{\rF} & \hbox{ if $e$ is bad},\\
\|M\|_{2,\infty} \vee \|M^* \|_{2,\infty} &  \hbox{ if $e_{\pm}$ is open and $e_{\mp}$ is extreme at $e$}, \\
 \|M\| &  \hbox{ if $e_{\pm}$ starts at $e$ and $e_{\mp}$ ends at $e$}, \\
\|M \|_{1,\infty}  &  \hbox{ if $e_-$ and $e_+$ are open at $e$}
\end{array} 
\right.
$$ 

\begin{lem}\label{le:tenseur}
Let $G = (V,E)$ be a finite connected directed graph such that $|V|\geq 2$ and all vertices have degree at least two. Then, for any ordering of the edge set  $E = \{e_1,\ldots,e_m\}$, we have  
$$
T^+_G(M) \leq \prod_{e \in E} \| M_e \|_{e} \leq \prod_{e \in E_b} \| M_e \|_{\rF} \prod_{e \in E \backslash E_b} \| M_e\|,
$$
where   $E_b$ is the set of bad edges.
\end{lem}
\begin{proof}
Note that since $\|M_e\|_e$ only depends on the absolute values of the entries of $M_e$, it is sufficient to prove the required bound for $T_G(M)$. We define for $e = (e_- ,e_+)$, the matrix in $ M_{n ^{|V|}} (\dC) \simeq M_n (\dC)^{\otimes V}$,
$$
\tilde M_e = \sum_{i_-,i_+}  (M_e)_{i_-i_+} \bigotimes_{v \in V} J_{v,e,i_-,i_+} ,
$$
where $J_{v,e,i_-,i_+} = I_n$ if $v \notin \{e_-,e_+\}$,  $J_{v,e,i_-,i_+} = E_{1i_\pm}$ if $e_\pm =v$ and $v$ starts at $e$, $J_{v,e,i_-,i_+} = E_{i_\pm 1}$ if $e_\pm =v$ and $v$ ends at $e$, finally  $J_{v,e,i_-,i_+} = E_{i_\pm i_\pm}$ if  $e_\pm =v$ and $v$ is open at $e$ (note that the case $e_-= e_+ = v$ where $v$ starts and ends at $e$ is not possible since $G$ connected and $|V| \geq 2$). We have the following identity
$$
\tilde M_{e_1} \cdots \tilde M_{e_m} = E_{11}^{\otimes V} T_G (M).
$$
Hence, using the sub-multiplicativity of the operator norm,
$$
\ABS{ T_G(M) } \leq \prod_{e \in E} \| \tilde M_e \|.
$$
We finally observe that 
$$
\| \tilde M_e \| = \| M \|_e.
$$
Indeed, since $\| I \otimes A \| = \| A \|$, it suffices to check to following identities, for $M \in M_n(\dC)$,
\begin{eqnarray*}
&  \| \sum_{i,j} M_{ij} E_{ii} \otimes E_{jj} \| = \| M \|_{1,\infty} \, ,   \quad \| \sum_{i,j} M_{ij} E_{ii} \otimes E_{1j} \| = \| M \|_{2,\infty}  \\
& \| \sum_{i,j} M_{ij} E_{1i} \otimes E_{j1} \| = \| M \| \; \hbox{ and } \; \| \sum_{i,j} M_{ij} E_{1i} \otimes E_{1j} \|  = \| M\|_{\rF}.
\end{eqnarray*}
The conclusion follows.
\end{proof}

The next lemma uses the probabilistic method to prove the existence of an ordering which does not have  too many bad edges. 

\begin{lem}\label{le:PM}
Let $G = (V,E)$ be a finite directed even graph with $|V| \geq 2$. There exists an ordering of $E$ with two  bad edges. Moreover, if $f \in E$, there exists an ordering of $E$ such that $f$ is among the two bad edges.
\end{lem}

\begin{proof}
Let $\bar G = (V,\bar E)$ be the un-directed multi-graph associated to $G$.  Let $m = |E| = |\bar E|$ be the number of edges. Since $\bar G$ is even, from Euler's Theorem, there exists an Eulerian circuit on $\bar G$, that is a sequence $(\bar \pi_1,\ldots,\bar \pi_m)$ in $\bar E$ such that  for all $t = 1,\ldots,m$  $\bar \pi_t$ and $\bar \pi_{t+1}$ share at least one neighboring vertex with $\bar \pi_{m+1} = \bar \pi_1$. For each $t$, let $\pi_t \in E$ be the corresponding edge in $E$. We consider the ordering $\pi_1,ldots,\pi_m$. By construction, $\pi_1$ and $\pi_m$ are the only two bad edges. Moreover, using the cyclicity of the Eulerian circuit, we can choose the circuit such that  $\pi_1 = f$.
\end{proof}

We are ready for the proof of Theorem \ref{th:sumproduct}.
\begin{proof}[Proof of Theorem \ref{th:sumproduct}]
If $|V| = 1$ then the statement is a straightforward consequence of the bounds $\| M \|_{\rF} \leq \sqrt{\mathrm{rank}(M)} \| M \|$. If $|V| \geq 2$, it is a consequence of Lemma \ref{le:tenseur} and Lemma \ref{le:PM}.
\end{proof}

\bibliographystyle{alpha}
\bibliography{toeplitz_bib}

\end{document}